\documentclass[reqno,11pt]{amsart}
\usepackage{amsfonts,amsmath,amssymb,enumerate} 
\usepackage{mathtools}

\numberwithin{equation}{section}

\newtheorem{theorem}{Theorem}[section]
\newtheorem{lemma}[theorem]{Lemma}
\newtheorem{corollary}[theorem]{Corollary}
\newtheorem{proposition}[theorem]{Proposition}
\newtheorem{definition}[theorem]{Definition}
\newtheorem{remark}[theorem]{Remark}

\begin{document}

\title{Multispeed Klein-Gordon systems in dimension three}

\author{Yu Deng}
\address{Courant Institute of Mathematical Sciences, New York University, New York, NY 10012}
\email{yudeng@cims.nyu.edu}

\thanks{}

\subjclass[2010]{}

\date{}

\dedicatory{}

\begin{abstract}We consider long time evolution of small solutions to general multispeed Klein-Gordon systems on $\mathbb{R}\times\mathbb{R}^{3}$. We prove that such solutions are always global and scatter to a linear flow, thus extending partial results obtained previously in \cite{G11, GM11, IP12}. The main new ingredients of our method is an improved \emph{linear} dispersion estimate exploiting the asymptotic spherical symmetry of Klein-Gordon waves, and a corresponding \emph{bilinear} oscillatory integral estimate.
\end{abstract}

\maketitle


\setlength{\parskip}{0.8ex}
\section{Introduction}\label{intro3}
In this paper we consider a system of quasilinear, multispeed Klein-Gordon equations in space dimension three, namely
\begin{equation}\label{nlkg}(\partial_{t}^{2}-c_{\alpha}^{2}\Delta+b_{\alpha}^{2})u_{\alpha}=\mathcal{Q}_{\alpha}(u,\partial u,\partial^{2}u),\qquad 1\leq\alpha\leq d,\end{equation} where the speeds $c_{\alpha}$ and masses $b_{\alpha}$ are arbitrary positive numbers, and $\mathcal{Q}$ is a quadratic, quasilinear nonlinearity satisfying a suitable symmetry condition. 

We are mostly interested in the behavior of small solutions to (\ref{nlkg}), which has been studied in many previous works. A typical phenomenon that happens for such solutions is global regularity and scattering, which is suggested by the dispersive nature of linear Klein-Gordon equation; to justify this one needs to combine the standard energy estimate with a suitable pointwise decay estimate to pass to arbitrarily long times. The difficulty of this process depends on the dimension (with higher dimensional cases being generally easier), and on the exact combination of speeds $c_{\alpha}$ and masses $b_{\alpha}$ (which determines the structure of the so-called ``resonance set''). The easiest case is a single Klein-Gordon equation in $3D$, and was first settled independently by Klainerman \cite{Kl85} and Shatah \cite{S85} (see also Guo \cite{Gu98} for the Euler-Poisson system, which linearizes to Klein-Gordon). In the case of (\ref{nlkg}) with the \emph{same} speed ($c_{\alpha}=1$) and multiple masses, Hayashi-Naumkin-Wibowo \cite{HNW} obtained global existence in $3D$, and Delort-Fang-Xue \cite{DFX} proved the same result in $2D$ under a non-resonance assumption. Note that the above works are more or less based on physical space analysis.

The case of (\ref{nlkg}) with \emph{multiple} speeds, even in $3D$, is significantly more challenging. In recent years there have been works developing new, Fourier-based techniques to solve these problems (see a brief discussion in Section \ref{strategykg} below). These include for example Germain \cite{G11} where a specific case of (\ref{nlkg}) with the same mass was considered in $3D$, and Ionescu-Pausader \cite{IP12}, where the general system (\ref{nlkg}) was studied in $3D$, with speeds and masses $(b_{\alpha},c_{\alpha})$ satisfying two nondegeneracy conditions. Moreover, the techniques involved in all these works have also been used in analyzing many other dispersive equations or systems which are not necessarily Klein-Gordon, see for example  \cite{GM11}, \cite{GMS09}, \cite{GMS12}, \cite{GMS4}, \cite{IP11}, \cite{IP13}.

In this paper we will completely solve the small data problem of (\ref{nlkg}) in $3D$, removing the assumptions made in \cite{IP12}. More precisely, we will prove the following
\begin{theorem}\label{mainkg} Consider the system (\ref{nlkg}) in $\mathbb{R}\times\mathbb{R}^3$. We make the following assumptions.
\begin{enumerate}
\item The nonlinearity $\mathcal{Q}_{\alpha}$ in (\ref{nlkg}) has the form\begin{equation}\label{symmetrykg}\mathcal{Q}_{\alpha}(u,\partial u,\partial^{2}u)=\sum_{\beta,\gamma=1}^{d}\sum_{j,k,l=1}^{3}\big(A_{\alpha\beta\gamma}^{jk}u_{\gamma}+B_{\alpha\beta\gamma}^{jkl}\partial_{l}u_{\gamma}\big)\partial_{j}\partial_{k}u_{\beta}+\mathcal{Q}_{\alpha}'(u,\partial u),\end{equation} where $A$ and $B$ are tensors symmetric in $\alpha$ and $\beta$ with norm not exceeding one, and $\mathcal{Q}_{\alpha}'$ is an arbitrary quadratic form (of constant coefficient) of $u$ and $\partial u$;
\item The initial data $u(0)=g,\partial_{t}u(0)=h$ satisfy the bound\begin{equation}\label{initialkg}\|(g,\partial_{x}g,h)\|_{\mathcal{H}^{N}}+\|(g,\partial_{x}g,h)\|_{Z}\leq \varepsilon\leq\varepsilon_{0},\end{equation} where $N=1000$, the norm \begin{equation}\label{normkg}\|u\|_{\mathcal{H}^{k}}=\sup_{|\beta|\leq k}\|\Gamma^{\beta}u\|_{L^{2}},\qquad\Gamma=(\partial_{i},x_{i}\partial_{j}-x_{j}\partial_{i})_{1\leq i<j\leq 3},\end{equation} the $Z$ norm is defined in Definition \ref{normdef}, and $\varepsilon_0$ is small enough depending on $b_{\alpha}$ and $c_{\alpha}$.
\end{enumerate} 

Then we have the followings:
\begin{enumerate}
\item The system (\ref{nlkg}) has a unique solution $u$, with prescribed initial data, such that \begin{equation}u\in C_{t}^{1}\mathcal{H}_{x}^{N}(\mathbb{R}\times\mathbb{R}^{3}\to\mathbb{R}^{d}),\quad\partial_{x}u\in C_{t}^{0}\mathcal{H}_{x}^{N}(\mathbb{R}\times\mathbb{R}^{3}\to\mathbb{R}^{d});\end{equation}
\item We have energy bound and decay estimates \begin{equation}\label{bound01}\|u(t)\|_{\mathcal{H}^{N}}+\|(\partial_{x},\partial_{t})u(t)\|_{\mathcal{H}^{N}}\lesssim\varepsilon,\end{equation}
\begin{equation}\label{bound02}\sum_{|\mu|\leq N/2}\big(\|\Gamma^{\mu}u(t)\|_{L^{\infty}}+\|\Gamma^{\mu}(\partial_{x},\partial_{t})u(t)\|_{L^{\infty}}\big)\lesssim\frac{\varepsilon}{(1+|t|)\log^{20}(2+|t|)};\end{equation}
\item There exist $\mathbb{R}^{d}$ valued functions $w^{\pm}$ verifying the linear equation
\begin{equation}(\partial_{t}^{2}-c_{\alpha}^{2}\Delta+b_{\alpha}^{2})w_{\alpha}^{\pm}=0,\end{equation} such that we have scattering in slightly weaker spaces.
\begin{equation}\label{scat}\lim_{t\to\pm\infty}\big(\|u(t)-w^{\pm}(t)\|_{\mathcal{H}^{N-1}}+\|(\partial_{x},\partial_{t})(u(t)-w^{\pm}(t))\|_{\mathcal{H}^{N-1}}\big)=0.\end{equation}
\end{enumerate}
 \end{theorem}
 \begin{remark}\label{remark9} Note that $\varepsilon_{0}$ depends on $(b_{\alpha},c_{\alpha})$ in a way that is not necessarily continuous. This is because the proof relies heavily on the structure of the nonlinear phase function\[\Phi(\xi,\eta)=\sqrt{c_{\alpha}^{2}|\xi|^2+b_{\alpha}^2}\pm\sqrt{c_{\beta}^{2}|\xi-\eta|^2+b_{\beta}^2}\pm\sqrt{c_{\gamma}^{2}|\eta|^2+b_{\gamma}^2},\] which could be sensitive to small changes in $c_{\alpha}$ and $b_{\alpha}$. For example, when $b_{\alpha}+b_{\beta}\neq b_{\gamma}$, the constant $K_{0}$ in Section \ref{choiceof} (which ultimately determines $\varepsilon_0$) will depend on the value of $\theta=|b_{\alpha}+b_{\beta}-b_{\gamma}|$, and our argument does not rule out the theoretical possibility that $K_{0}\to\infty$ when $\theta\to 0$ (on the other hand, if $b_{\alpha}+b_{\beta}= b_{\gamma}$ then we have a different argument that gives a finite $K_0$).
 
 We believe that this annoyance is of technical nature, and can be avoided by more careful analysis of $\Phi$. However, since this is not very relevant to the main idea of this paper, and will make the proof unnecessarily long, we have chosen to state Theorem \ref{mainkg} as it is.
 \end{remark}

\subsection{Notations and choice of parameters}\label{choice}
\subsubsection{Notations}\label{notatkg} We fix an even smooth function $\varphi:\mathbb{R}\to[0,1]$ that is supported in $[-8/5,8/5]$ and equals $1$ in $[-5/4,5/4]$. It is well-known (see for example \cite{Ro}) that we can assume $\varphi\in\mathcal{G}_{4}(\mathbb{R})$, where $\mathcal{G}_{\sigma}$ is the Gevrey space,\[\mathcal{G}_{\sigma}(\mathbb{R}^{d})=\bigg\{f\in C^{\infty}(\mathbb{R}^{d}):\sup_{x\in K}|\partial^{\alpha}f(x)|\leq C_{K}^{|\alpha|}(\sigma|\alpha|)!\bigg\},\qquad\forall\textrm{ compact }K\subset\mathbb{R}^{d}.\] Abusing notation, we will regard $\varphi$ as a radial function on $\mathbb{R}^{3}$ by  $\varphi(x)=\varphi(|x|)$; We then have $\varphi\in \mathcal{G}_{6}(\mathbb{R}^{3})$. Next, define
\begin{equation*}\varphi_k(x):=\varphi(x/2^k)-\varphi(x/2^{k-1})\quad\text{ for any }k\in\mathbb{Z},\qquad \varphi_I:=\sum_{m\in I\cap\mathbb{Z}}\varphi_m\quad\text{ for any }I\subseteq\mathbb{R},
\end{equation*} and define functions like $\varphi_{\leq B}$ similarly.
For any $x\in\mathbb{Z}$ define \[x_+:=\max(x,0),\qquad x_-:=-\min(x,0).\] Let the set
\begin{equation*}
\mathcal{J}:=\{(k,j)\in\mathbb{Z}\times\mathbb{Z}_+:\,k+j\geq 0\},
\end{equation*} and for any $(k,j)\in\mathcal{J}$ let
\begin{equation*}
\varphi_{j}^{(k)}(x)=\left\{\begin{aligned}
&\varphi_{j}(x),&\quad\text{if }j&>\max(0,-k);\\
&\varphi_{\leq j}(x),&\quad\text{if }j&=\max(0,-k).
\end{aligned}\right.
\end{equation*} Let $P_k$ denote the operator on $\mathbb{R}^3$ defined by the Fourier multiplier $\xi\to \varphi_k(\xi)$; 
let $P_{\leq B}$ (respectively $P_{I}$) denote the operators on $\mathbb{R}^3$ defined by the Fourier 
multipliers $\xi\to \varphi_{\leq B}(\xi)$ (respectively $\xi\to \varphi_{I}(\xi)$). For $(k,j)\in\mathcal{J}$ let
\begin{equation}\label{qjk}
(Q_{jk}f)(x):=\varphi^{(k)}_j(x)\cdot P_kf(x);\qquad f_{jk}:=Q_{jk}f, \qquad f_{jk}^{*}:=P_{[k-2,k+2]}Q_{jk}f.
\end{equation}

Let $\Lambda_{\alpha}=\sqrt{b_{\alpha}^{2}-c_{\alpha}^{2}\Delta}$ be the linear phase, and define \[\Lambda_{\alpha}(\xi):=\sqrt{c_{\alpha}^{2}|\xi|^{2}+b_{\alpha}^{2}},\quad1\leq\alpha\leq d;\qquad b_{-\alpha}=-b_{\alpha},\quad c_{-\alpha}=c_{\alpha},\quad\Lambda_{-\alpha}=-\Lambda_{\alpha}.\]
Let $\mathcal{P}=\mathbb{Z}\cap([1,d]\cup[-d,-1])$; for $\sigma,\mu,\nu\in \mathcal{P}$, we define the associated nonlinear phase
\begin{equation}\label{phasedef}
\Phi_{\sigma\mu\nu}(\xi,\eta):=\Lambda_{\sigma}(\xi)-\Lambda_{\mu}(\xi-\eta)-\Lambda_{\nu}(\eta),
\end{equation} 
and the corresponding parallel phase
\begin{equation*}
\Phi^{+}_{\sigma\mu\nu}(\alpha,\beta):=\Phi_{\sigma\mu\nu}(\alpha e,\beta e)=\Lambda_{\sigma}(\alpha)-\Lambda_{\mu}(\alpha-\beta)-\Lambda_{\nu}(\beta),
\end{equation*}
where $e\in\mathbb{S}^{2}$ and $\alpha,\beta\in\mathbb{R}$.

For later purposes, we also need the spherical harmonics decomposition, which is described below. Let $(r,\theta,\varphi)$ be polar coordinates for $\xi$ and $Y_{q}^{m}$ be spherical harmonics, and note $\Delta_{\mathbb{S}^{2}}Y_{q}^{m}=-q(q+1)Y_{q}^{m}$. For any function $f$, we have the expansion\begin{equation}\label{decomp}f(\xi)=\sum_{q=0}^{\infty}\sum_{m=-q}^{q}f_{q}^{m}(r)Y_{q}^{m}(\theta,\varphi),\end{equation}  and may define accordingly\begin{equation}\label{spherical}S_{l}f:=\sum_{q=0}^{\infty}\sum_{m=-q}^{q}\varphi_{l}(q)f_{q}^{m}(r)Y_{q}^{m}(\theta,\varphi)\end{equation} for $l\geq 1$, and with $\varphi_{0}$ replaced by $\varphi_{\leq0}$ for $l=0$. Note that $S_{l}$ commutes with each $P_{k}$ and $Q_{jk}$.
\subsubsection{Choice of parameters}\label{choiceof}
Throughout the proof, we will fix some parameters as follows:
\begin{equation}\label{param}
\begin{aligned}
N&=1000, &\quad \delta&=1/990,&N_{0}&=10^{10};\\
D_{0}&\gg_{\mathcal{B}}1,&K_{0}&\gg_{D_{0}}1,&\varepsilon_{0}&\ll_{K_{0}}1,
\end{aligned}
\end{equation}where $\mathcal{B}:=\{(b_{\alpha},c_{\alpha})\}$. Note that $\delta=1/(N-10)$. In addition, let $A$ denote any large absolute constant such that $A\ll N_{0}$ (say $A\sim 10^5$), whose exact value may vary at different places; in the same way, let $o$ denote any small constant, so for example $o\ll\delta^{4}$.
\subsection{Description of the methods}\label{strategykg} 
\subsubsection{Energy estimates and decay in $L^{\infty}$}
Letting $v_{\sigma}=(\partial_{t}-i\Lambda_{\sigma})u_{\sigma}$, we can reduce (\ref{nlkg}) to \begin{equation}\label{diag}(\partial_{t}+i\Lambda_{\sigma})v_{\sigma}=\mathcal{N}_{\sigma}(v,v),\quad\nu\in\mathcal{P},\end{equation} where $\mathcal{N}_{\sigma}$ is some quasilinear quadratic term. The starting point of our analysis is an energy estimate of form \begin{equation}\label{energy0}\partial_{t}\mathcal{E}\approx\int_{\mathbb{R}^{3}}\partial^{N}v\cdot\partial^{N}v\cdot\partial^{2}v,\qquad\mathcal{E}\approx\|v\|_{H^{N}}^{2}.\end{equation} To obtain this, which is needed even in local theory, one need to perform suitable symmetrization which relies on the assumption that the nonlinearities $\mathcal{Q}_{\alpha}$ are symmetric.

Next, from (\ref{energy0}) and Gronwall we can estimate the energy $\mathcal{E}=\mathcal{E}(t)$ by\[\mathcal{E}(t)\lesssim \mathcal{E}(0)\exp\bigg(\int_{0}^{t}\|\partial^{2}v\|_{L^{\infty}}\bigg).\] To bound the energy globally, it then suffices to study decay estimates of form \begin{equation}\label{decay98}\|\partial^{2}v(t)\|_{L^{\infty}}\lesssim(1+|t|)^{-1-\delta}.\end{equation}
\subsubsection{Space localization and the $Z$ norm}
The idea of obtaining (\ref{decay98}) is to introduce the \emph{profile} \[f_{\sigma}(t)=e^{it\Lambda_{\sigma}}v_{\sigma}(t)\]by taking back the linear flow. Under the scattering ansatz, $f_{\sigma}$ is expected to settle down (thus uniformly bounded) for long time, and also enjoy localization in space. This localization is captured in a carefully designed ``$Z$-norm'' which was used in \cite{IP12}; see also similar localization bounds in \cite{G11} and \cite{GM11}.  One then has the linear dispersion estimate \begin{equation}\label{decaykg}\|v_{\sigma}(t)\|_{L^{\infty}}=\|e^{-it\Lambda_{\sigma}}f_{\sigma}(t)\|_{L^{\infty}}\lesssim (1+|t|)^{-1-\delta}\|f_{\sigma}(t)\|_{Z},\end{equation} provided that the $Z$ norm is chosen correctly.

 Proving that $\|f_{\sigma}(t)\|_{Z}$ is uniformly bounded is achieved by using the Duhamel formula\begin{equation}\label{duhamel}\widehat{f_{\sigma}}(t,\xi)=\widehat{f_{\sigma}}(0,\xi)+\int_{0}^{t}\int_{\mathbb{R}^{3}}e^{{{i}}s\Phi_{\sigma\mu\nu}(\xi,\eta)}m(\xi,\eta)\widehat{f_{\mu}}(s,\xi-\eta)\widehat{f_{\nu}}(s,\eta)\,\mathrm{d}\eta\mathrm{d}s.\end{equation} Here one can notice that, at least in the approximation when the inputs $f_{\mu}$ and $f_{\nu}$ are Schwartz, the main contribution of the integral (\ref{duhamel}) comes from the vicinity of the stationary set\[\mathcal{R}:=\{(\xi,\eta):\Phi_{\sigma\mu\nu}(\xi,\eta)=\nabla_{\eta}\Phi_{\sigma\mu\nu}(\xi,\eta)=0\}.\] This set $\mathcal{R}$, which is called the \emph{spacetime resonance} set, was first introduced in \cite{GMS09} and plays a crucial role in the analysis.

Using this strategy, in \cite{IP12}, the authors were able to perform a robust stationary phase analysis near $\mathcal{R}$, and proved small data scattering of (\ref{nlkg}) under the additional assumptions that
\begin{equation}\label{assm1}(c_{\alpha}-c_{\beta})(c_{\alpha}^{2}b_{\alpha}-c_{\beta}^{2}b_{\beta})\geq 0,\qquad1\leq\alpha,\beta\leq d;\end{equation}
\begin{equation}\label{assm2}b_{\alpha}+b_{\beta}-b_{\gamma}\neq 0,\qquad1\leq\alpha,\beta,\gamma\leq d.\end{equation}
Below we will briefly discuss why these two assumptions are needed in \cite{IP12}, and how we are able to remove them in the current paper.
\subsubsection{Spherical symmetry and rotation vector fields}
It is clear that the choice of the $Z$ norm is the crucial component in the scheme described above; in particular, it should be strong enough such that (\ref{decaykg}) holds, and also weak enough such that the ``second iteration'', namely the output function $f_{\sigma}$ of (\ref{duhamel}) assuming the input functions $f_{\mu}$ and $f_{\nu}$ are Schwartz and independent of time, has bounded $Z$ norm.

Following \cite{IP12}, let \[\|f\|_{Z}=\sup_{j,k}\|f_{jk}\|_{Z_{jk}},\] where $f_{jk}$ is obtained by localizing $f$ at frequency scale $\sim 2^k$ and space scale $\sim 2^{j}$. In \cite{IP12} it turns out that, in order for (\ref{decay98}) to hold, the $Z$ norm has to incorporate strong space localization, and has to be (almost) as strong as \[\|f\|_{Z_{jk}}=2^{j}\|f\|_{L^{2}}.\] Moreover, we can see from a simple volume counting argument that, to guarantee that the second iteration belongs to this space, it is necessary that the phase $\Phi$ is nondegenerate, namely $\det(\nabla_{\eta}^{2}\Phi_{\sigma\mu\nu})\neq 0$ on the spacetime resonance set $\mathcal{R}$. This is why (\ref{assm1}) has to be assumed in \cite{IP12}; it is a convenient algebraic condition that implies this non-degeneracy. In the absence of (\ref{assm1}), the best control we can prove for the output function of (\ref{duhamel}) (i.e., the strongest $Z$ norm possible) is uniform bound in the norm $\|f\|_{Z_{jk}'}=2^{5j/6}\|f\|_{L^{2}}$, which is not strong enough to imply (\ref{decaykg}).

To overcome this obstacle, we will introduce the rotation vector fields, as included the vector field set $\Gamma$ in (\ref{normkg}). The idea of using such vector fields goes back to Klainerman \cite{Kl85}, \cite{Kl86} (see also \cite{Kl87}); in the current situation, we have \begin{equation}\label{impr}\|e^{-it\Lambda}f_{jk}\|_{L^{\infty}}\lesssim (1+t)^{-1-\delta}2^{(1/2+\delta)j}(\|f_{jk}\|_{L^{2}}+\|\Omega f_{jk}\|_{L^{2}}),\end{equation} at least when $j\leq (1-\delta^2)m$ and $|k|\lesssim 1$, where $\Omega$ is the rotational part of $\Gamma$. Note that we gain a power $2^{(1/2-\delta)j}$ at the price of using one $\Omega$, which in particular allows us to close the argument with the weak $2^{5j/6}$ power in the $Z$ norm. For the proof of the improved linear dispersion inequality (\ref{impr}), see Proposition \ref{disper0} below.
\subsubsection{A bilinear lemma}\label{bililemma} The above arguments are only for the second iteration. In general, the input functions $f_{\mu}$ and $f_{\nu}$ are not really Schwartz, which causes trouble when trying to localize $\eta$ to the vicinity of the ``space resonant set'' where $\nabla_{\eta}\Phi=0$, in the study of the integral (\ref{duhamel}). In \cite{IP12}, this was solved by using the (strong) $Z$ norm and an orthogonality trick; these arguments will not be sufficient here, due to our $Z$ norm being strictly weaker.

This issue is solved in the current paper by exploiting further the rotation vector fields $\Omega$; here note that at least for medium frequencies (i.e. frequencies $\sim 1$, which is where spacetime resonance takes place), we have\[\sup_{\theta\in\mathbb{S}^{2}}\|f(\rho\theta)\|_{L_{\rho}^{2}}\lesssim\sup_{|\alpha|\leq 2}\|\Omega_{0}^{\alpha}f\|_{L^{2}},\] thus we can gain a factor of $\delta^{2}$, provided that we can somehow restrict the direction vector of $\eta$ in (\ref{duhamel}) to a region of size $\delta$ in $\mathbb{S}^{2}$. This motivates the second main technical tool in this paper, namely the bilinear integration lemma, Lemma \ref{angular} below, which holds under very mild conditions, and effectively allows one to restrict to $|\sin\angle (\xi,\eta)|\lesssim 2^{-(1/2-\epsilon)m}$ for medium frequencies. With this additional gain, we can then close the $Z$ norm estimate using Schur's bound; for details see Section \ref{med}.

We remark that the use of rotation vector fields in this paper, and in particular the gain in both linear (Proposition \ref{disper0}) and bilinear (Lemma \ref{angular}) estimates, is a quite general feature and has since been applied to different dispersive models, see \cite{DIP14}, \cite{DIPP16}; see also \cite{IP-ww} and \cite{IP-cp} for slightly different uses of vector fields.
\subsubsection{Low frequencies, and sharp integration by parts} The other assumption (\ref{assm2}) in \cite{IP12} was used to ensure that $(0,0)\not\in\mathcal{R}$. This greatly simplifies the analysis by allowing to integrate by parts in $s$ (and obtain enough gain) for all low frequency inputs.

Without (\ref{assm2}), one has to study low frequency inputs on a case by case basis; moreover, the point $(0,0)$, which can now be resonant, can actually be very degenerate, meaning it is possible to have$(\nabla_{\eta}^{\alpha}\Phi)(0,0)=0$ for $|\alpha|\leq 3$ (which is clearly a disadvantage in terms of stationary phase analysis). One example is when\[\Phi(\xi,\eta)=\sqrt{|\xi|^{2}+1}-\sqrt{2|\xi-\eta|^{2}+4}+\sqrt{|\eta|^{2}+1}=\xi\cdot\eta/2+O(|\xi|^{4}+|\eta|^{4}).\] To see the effect of this degeneracy, we consider a ``rescaled Schwartz'' component $f_{j,-j}$ localized in frequency at scale $\sim 2^{-j}$ and in space at scale $\sim2^j$. Suppose that $|s|\sim 2^{m}$ in (\ref{duhamel}), and that the inputs are \[\widehat{f_{\mu}}(s,\xi-\eta)=2^{cj}\chi(2^{j}(\xi-\eta)),\qquad \widehat{f_{\nu}}(s,\eta)=2^{cj}\chi(2^{j}\eta)\] with $j=m/4$ and some constant $c$, then in the region where $|\eta|\sim|\xi-\eta|\sim 2^{-j}$ and $|\xi|\sim 2^{-3j}$, we have $|\Phi|\lesssim 2^{-m}$, so the oscillation factor $e^{is\Phi}$ is irrelevant, thus the output will be a rescaled Schwartz function supported at frequency scale $|\xi|\sim 2^{-3j}$ with $L_{\xi}^{\infty}$ norm bounded by\[2^{m}2^{-3j}2^{2cj}=2^{(3j)(2c+1)/3}.\] If we start with $c=0$ (imagine cutting off a Schwartz function at scale $|\xi|\sim 2^{-j}$), then in finitely many iterations we obtain a profile with form $2^{cj}\chi(2^{j}\xi)$, where $c$ can be arbitrarily close to $1$ (which is the unique fixed point of the map $c\mapsto (2c+1)/3$). Therefore, if we were to choose $\|f\|_{Z_{jk}}=2^{j}2^{\lambda k}\|f\|_{L^{2}}$ as in \cite{IP12}, then we \emph{must} have $\lambda\geq 1/2$. On the other hand, it can be proved that when $\lambda\geq 1/2$, the best decay rate for $\|e^{-it\Lambda_{\nu}}f\|_{L^{\infty}}$ is precisely $t^{-1}$, which is not integrable. Note that this cannot be helped by the use of rotation vector fields $\Omega$, since it does nothing to rescaled Schwartz functions.

In this paper, we circumvent this difficulty by adding a log factor and choosing  \[\|f\|_{Z_{jk}}=2^{j}2^{k/2}\langle j\rangle^{N_{0}}\|f\|_{L^{2}}.\] Note that this makes the decay rate integrable, and does \emph{not} violate the above intuitions. The drawback is that, when integrating by parts with low frequencies, we have to be precise ``up to log factors''. In order to achieve this, we use a sharp integration by parts lemma, Proposition \ref{ips0}, which requires that we integrate by parts $A$ times with $A$ \emph{depending on} the functions themselves. This lemma is proved in Section \ref{norms}, and the analysis of low frequencies is carried out in Section \ref{lowlowtolow}.
\subsubsection{Remarks on the two dimensional case} The study of (\ref{nlkg}) in two dimensions will be much harder, and one should not expect a general result as the one in this paper. In fact, with carefully chosen speeds and masses, one can construct a Klein-Gordon system that is expected to blowup in finite time even for small data, on a heuristic level (though actually constructing blowup solutions would also be hard). On the other hand, for \emph{generic} choices of speeds and masses, one has correct non-degeneracy and separation conditions (see \cite{DIP14}; see also \cite{DIPP16}, \cite{G11} and \cite{GM11} for other models), and can still expect global well-posedness and scattering. See \cite{DIP14} for a system that exhibits the typical behavior for generic Klein-Gordon systems in two dimensions.
\subsubsection{Plan of this paper}In Section \ref{norms} we define the relevant notations and in particular the $Z$ norm, and prove the crucial linear dispersion bound. In Section \ref{energy} we prove the energy bound, and reduce Theorem \ref{mainkg} to the main $Z$ norm estimate, namely Proposition \ref{main01}. In Section \ref{beginning} we make several reductions and get rid of some easy instances of Proposition \ref{main01}; we then discuss the low frequency case in Section \ref{lowlowtolow}, and the medium frequency case, which requires more care, in Section \ref{med}. The high frequency case is much easier and is dealt with in Section \ref{high}. Finally, in Section \ref{phaseprop} we collect some auxiliary facts about the phase function and spherical harmonics decomposition that will be used throughout this paper.
\section{Norms and basic estimates}\label{norms}
\subsection{The definition of $Z$ norm}
Recall the notations defined in Section \ref{notatkg}, and also the vector field set $\Gamma$ as in (\ref{normkg}). We will define the $Z$ norm as follows.\begin{definition}\label{normdef}Fix $(j,k)\in\mathcal{J}$, define the norm\begin{equation}\label{part}\|f\|_{Z_{jk}}=
\begin{dcases}\sup_{|\mu|\leq N/2+2}\langle j\rangle^{N_{0}}2^{\min(k/2,0)}2^{j}\|\Gamma^{\mu}f\|_{L^{2}},&|k|\geq K_{0};\\
\sup_{|\mu|\leq N/2+2}\big(2^{5j/6}\langle j\rangle^{-N_{0}}\|\Gamma^{\mu}f\|_{L^{2}}+\langle j\rangle^{N_{0}}2^{j}\|\widehat{\Gamma^{\mu}f}\|_{L^{1}}\big),&|k|<K_{0}.\end{dcases}\end{equation}Also define the full $Z$ norm by\begin{equation}\label{fullz}\|f\|_{Z}=\sup_{(j,k)\in\mathcal{J}}\|f_{jk}\|_{Z_{jk}},\end{equation} and the $X$ norm by\[\|f\|_{X}=\sup_{|\mu|\leq N}\|\Gamma^{\mu}f\|_{L^{2}}+\|f\|_{Z}.\]\end{definition} Notice that by (\ref{part}) we have \begin{equation}\label{normcomp}\|f\|_{Z_{jk}}\lesssim \|f\|_{Z_{jk'}},\qquad |k|<K_{0}\leq |k'|\leq O_{K_{0}}(1),\end{equation} if $\widehat{f}$ is supported in $|\xi|\leq O_{K_0}(1)$; this simple observation will be convenient later.

\subsection{Linear and bilinear estimates}
\begin{proposition}[Sharp integration by parts]\label{ips0}Suppose $K,\lambda\geq 1$ and $\epsilon_{j}$ are positive parameters, and $h(\eta)$ is some compact supported function on $\mathbb{R}^{3}$, verifying
\begin{equation}\label{ibp1}\|\partial_{\eta}^{\mu}h(\eta)\|_{L^{1}}\lesssim (CN)!\lambda^{N}\end{equation} for some $C\geq 1$, and $N$ and all $|\mu|\leq N$. Moreover, let $\Phi=\Phi(\eta)\in\mathcal{G}_{C}$ be a function such that
\begin{equation}\label{ibp2}|\partial_{\eta}\Phi(\eta)|\geq\epsilon_{1};\qquad|\partial_{\eta}^{\mu}\Phi(\eta)|\leq \epsilon_{|\mu|}\end{equation} holds for each $2\leq |\mu|\leq n$ and at each point $\eta$ where $h$ or one of its derivatives is nonzero, then we have the estimate
\begin{equation}\label{ipsp}\bigg|\int_{\mathbb{R}^{3}}e^{{i}K\Phi(\eta)}h(\eta)\,\mathrm{d}\eta\bigg|\lesssim e^{-\gamma M^{\gamma}},\end{equation}where\begin{equation}M=\min(K\epsilon_{1}^{2}/\epsilon_{2},K\epsilon_{1}\epsilon_{2}/\epsilon_{3},\cdots,K\epsilon_{1}\epsilon_{n-1}/\epsilon_{n},K\epsilon_{1}\epsilon_{n},K\epsilon_{1}/\lambda),\end{equation}and $\gamma$ is small enough depending on $C$. In particular, if we can choose $\epsilon_{j}=\epsilon^{\frac{n-j+1}{l}}$, then we have\[M=\min(K\epsilon^{\frac{n+1}{n}},K\epsilon/\lambda).\]

If we fix a direction $\theta\in\mathbb{S}^{2}$ and replace the $\partial_{\eta}$ in (\ref{ibp1}) and (\ref{ibp2}) by the directional derivative $\theta\cdot\partial_{\eta}$, then the same result will remain true (uniformly in $\theta$).
\end{proposition}
\begin{proof} First consider the case when we have a directional derivative, and assume that it is $\partial_{1}$. We will integrate by parts in $\eta_{1}$ a total of $N$ times, where $N$ is a large integer to be determined. Let the differential operator $D$ be defined by\[Du=\frac{\partial_{1}u}{\partial_{1}\Phi};\,\,\,\,\,\,\,D(e^{{i}K\Phi})={i}Ke^{{i}K\Phi},\] then its dual $D'$ would be \[D'u=-\partial_{1}\bigg(\frac{u}{\partial_{1}\Phi}\bigg).\] Therefore, integrating by parts, we obtain
\begin{equation}\bigg|\int_{\mathbb{R}^{3}}e^{{i}K\Phi(\eta)}h(\eta)\,\mathrm{d}\eta\bigg|\leq K^{-N}\big\|(D')^{N}h\big\|_{L^{1}}.\end{equation} In order to bound $(D')^{N}h$, we will use the following explicit formula, namely
\begin{equation}\label{induct2}(D')^{N}h=\sum_{r=0}^{N}\sum_{\alpha_{0}+\cdots+\alpha_{r}=N-r}\Omega(n,r;\alpha_{0},\cdots,\alpha_{r})\cdot\partial_{1}^{\alpha_{0}}h\cdot\frac{\partial_{1}^{\alpha_{1}+2}\Phi\cdots\partial_{1}^{\alpha_{r}+2}\Phi}{(\partial_{1}\Phi)^{r+N}},\end{equation} where the coefficient \[|\Omega|\leq 3^{N}(N+1)!.\] The formula (\ref{induct2}) is easily proved by induction.

Now, using (\ref{ibp2}), (\ref{ibp1}) and (\ref{induct2}), we have\begin{equation}\label{boundp}K^{-N}\big\|(D')^{N}h\big\|_{L^{1}}\lesssim ((3C+3)N)!\cdot K^{-N}\lambda^{\alpha_{0}}\cdot\epsilon_{1}^{-N-r}\prod_{j=0}^{n-2}\epsilon_{j+2}^{r_{j}}.\end{equation} Here we denote by $r_{j}$ the number of $q\geq 1$ such that $\alpha_{q}=j$. Let\begin{equation}\label{prepare}\rho:=r-\sum_{j=0}^{n-2}r_{j}\geq 0;\,\,\,\,\,\,\,\alpha_{0}+\sum_{j=0}^{n-2}jr_{j}\leq N-r-(n-1)\rho.\end{equation}

We now denote $K\epsilon_{1}/\lambda=M_{0}$ and $K\epsilon_{1}\epsilon_{j}/\epsilon_{j+1}=M_{j}$, where $\epsilon_{n+1}=1$. We could then solve that\[\epsilon_{j}=(M_{1}\cdots M_{j-1}K)^{\frac{j-n-1}{n+1}}(M_{j}\cdots M_{n})^{\frac{j}{n+1}};\]\[\lambda=(M_{1}\cdots M_{n}K)^{\frac{1}{n+1}}M_{0}^{-1}.\]Plugging into (\ref{boundp}) we obtain\begin{eqnarray}K^{-N}\big\|(D')^{N}h\big\|_{L^{1}}&\lesssim&((3C+3)N)!K^{-N}(M_{1}\cdots M_{n}K)^{\frac{\alpha_{0}}{n+1}} M_{0}^{-\alpha_{0}}\times(M_{1}\cdots M_{n})^{\frac{-N-r}{n+1}}\times\nonumber\\
&\times&K^{\frac{n(N+r)}{n+1}}\prod_{j=0}^{n-2}(M_{1}\cdots M_{j+1}K)^{\frac{r_{j}(j-n+1)}{n+1}}(M_{j+2}\cdots M_{n})^{\frac{r_{j}(j+2)}{n+1}}\nonumber\\
&=&((3C+3)N)!K^{\sigma}\prod_{j=0}^{n}M_{j}^{\tau_{j}},\end{eqnarray} where \[\sigma=-N+\frac{\alpha_{0}+n(N+r)}{n+1}+\sum_{j=0}^{n-2}\frac{r_{j}(j-n+1)}{n+1}\leq 0,\] and $\tau_{0}=-\alpha_{0}\leq 0$, and\[\tau_{j}=\frac{\alpha_{0}-N-r}{n+1}-\sum_{i=j-1}^{n-2}r_{i}+\frac{1}{n+1}\sum_{i=0}^{n-2}r_{i}(i+2)\leq -\rho\leq 0\]for $1\leq j\leq n$, and \begin{eqnarray}\sum_{j=0}^{n}\tau_{j}&=&\frac{-\alpha_{0}-n(N+r)}{n+1}+\frac{n}{n+1}\sum_{i=0}^{n-2}r_{i}(i+2)-\sum_{i=0}^{n-2}r_{i}(i+1)\nonumber\\
&=&-\frac{nN}{n+1}-\frac{1}{n+1}\big(\alpha_{0}+\sum_{i=0}^{n-2}r_{i}(i+1)+n\rho\big)\leq-\frac{n}{n+1}N\nonumber,\end{eqnarray} all the inequalities being consequences of (\ref{prepare}). This then implies that\[K^{-N}\big\|(D')^{N}h\big\|_{L^{1}} \lesssim ((3C+3)N)!\min(M_{0},\cdots,M_{l})^{-\frac{nN}{n+1}}.\]If we now choose $N$ appropriately, an easy computation will show that
\begin{equation}\bigg|\int_{\mathbb{R}^{3}}e^{{i}K\Phi(\eta)}h(\eta)\,\mathrm{d}\eta\bigg|\lesssim e^{-\gamma M^{\gamma}}\end{equation} for $\gamma$ small enough.

Now suppose the directional derivative is replaced by the full gradient. By (\ref{ibp2}), we have that\[h(\eta)=h(\eta)\cdot\bigg(1-\prod_{j=1}^{3}\varphi_{0}(2\epsilon_{1}^{-1}\partial_{j}\Phi(\eta))\bigg),\] so we only need to bound the integral with $h$ replaced by $h_{1}=h\cdot\varphi_{0}(2\epsilon_{1}^{-1}\partial_{1}\Phi)\varphi_{0}(2\epsilon_{1}^{-1}\partial_{2}\Phi)$ (the other similar terms are bounded in the same way).

Since (\ref{ibp2}) now holds with $\partial_{\eta}$ replaced by $\partial_{1}$ at each point where $h_{1}$ or one of its derivative is nonzero (possibly with different constants), we only need to show that $h_{1}$ also verifies the bound (\ref{ibp1}), but with $\lambda$ replaced by the maximum of $\lambda$ and each $\epsilon_{j+1}/\epsilon_{j}$ (again, let $\epsilon_{n+1}=1$). Using Leibniz rule, we can further reduce to proving the same result for $\varphi_{0}(2\epsilon_{1}^{-1}\partial_{2}\Phi)$ but with $\lambda$ replaced by the maximum of $\epsilon_{j+1}/\epsilon_{j}$, the $L^{1}$ norm replaced by the $L^{\infty}$ norm, and restrict to the subset where $h$ or one of its derivatives is nonzero (note that $\varphi_{0}(2\epsilon_{1}^{-1}\partial_{1}\Phi)$ is estimated in the same way).

Now, using (\ref{induct}) we have
\begin{equation}\partial_{\eta}^{\mu}(\varphi_{0}(2\epsilon_{1}^{-1}\partial_{2}\Phi))=\sum_{r=0}^{N}(2\epsilon_{1}^{-1})^{r}\sum_{\mathcal{T}}\Omega(\mathcal{T})\cdot(\partial_{j_{1}}\cdots\partial_{j_{r}}\varphi_{0})(2\epsilon_{1}^{-1}\partial_{2}\Phi)\cdot\prod_{q=1}^{r}\partial_{x}^{\mu_{q}}\partial_{j_{q}}\partial_{2}\Phi,\end{equation} here the summation is taken over all tuples\[\mathcal{T}=(\mu_{1},\cdots,\mu_{d},j_{1},\cdots,j_{d}),\] and we have \[\sum_{i=1}^{r}|\mu_{i}|=N-r,\qquad|\Omega(\mathcal{T})|\leq (N+1)!.\]Let, for each $0\leq j\leq n-2$, the number of $q$'s such that $|\mu_{q}|=j$ be $r_{j}$, then we will have \[\rho=r-\sum_{j=0}^{n-2}r_{j}\geq0,\qquad\sum_{j=0}^{n-2}jr_{j}\leq N-r-(n-1)\rho.\] Therefore, since we are in the set where the second part of (\ref{ibp2}) holds, we can bound\begin{eqnarray}\big\|\partial_{\eta}^{\mu}(\varphi_{0}(2\epsilon_{1}^{-1}\partial_{2}\Phi))\big\|_{L^{\infty}}&\lesssim& ((3C+8)N)!\cdot\epsilon_{1}^{-r}\prod_{j=0}^{n-2}\epsilon_{j+2}^{r_{j}}\lesssim ((3C+8)N)!\epsilon_{1}^{-r}\prod_{j=0}^{n-2}(\epsilon_{1}(\lambda')^{j+1})^{r_{j}}\nonumber\\
&\lesssim & ((3C+8)N)!\epsilon_{1}^{-\rho}(\lambda')^{N-n\rho}\nonumber\\
&\lesssim& ((3C+8)N)!(\lambda')^{N}\nonumber,\end{eqnarray} where $\lambda'$ is the maximum of $\epsilon_{j+1}/\epsilon_{j}$, since we have $\epsilon_{1}^{-1}\lesssim (\lambda')^{n}$. This completes the proof.
\end{proof}
\begin{remark}\label{remark0} When $n=1$, we will also use a slightly more general version that allow\[|\partial_{\eta}^{\mu}\Phi(\eta)|\lesssim (CN)!(\lambda')^{N}\] for some $\lambda'\geq 1$. In this case we simply rescale to reduce to the case proved above, and obtain that (\ref{ipsp}) holds with $M=\min(K(\lambda')^{-2}\epsilon^{2},K\epsilon/\lambda)$.
\end{remark}
\begin{proposition}[Improved dispersion decay]\label{disper0}
Suppose $\|f\|_{X}\lesssim 1$, and that $\langle t\rangle\sim 2^{m}$. Also let $\Lambda=\Lambda_{\nu}$ with $\nu\in\mathcal{P}$ and fix $(j,k)\in\mathcal{J}$. 
\begin{enumerate}
\item If $k\leq -K_{0}$, we have\begin{equation}\label{disper1}\sup_{|\mu|\leq N/2+2}\big\|e^{{i}t\Lambda}\Gamma^{\mu}S_{l}f_{jk}^{*}\big\|_{L^{\infty}}\lesssim\langle j\rangle^{-N_{0}}\min(2^{-j+k},2^{-(3m-j+k)/2}).\end{equation}

\item If $|k|<K_{0}$ and $ |j-m|\leq A\log m$, we have
\begin{equation}\label{disper4}\sup_{|\mu|\leq N/2+2}\big\|e^{{i}t\Lambda}\Gamma^{\mu}S_{l}f_{jk}^{*}\big\|_{L^{\infty}}\lesssim 2^{-m}\langle m\rangle^{-N_{0}+A}.\end{equation}

\item If $|k|<K_{0}$, and $|j-m|\geq A\log m$, we have
\begin{equation}\label{disper3}\sup_{|\mu|\leq N/2+2}\big\|e^{{i}t\Lambda}\Gamma^{\mu}S_{l}f_{jk}^{*}\big\|_{L^{\infty}}\lesssim 2^{-3m/2-j/3}\min(2^{l}\langle m\rangle^{N_0+A},2^{4\delta(m+j)}).\end{equation} 

\item If $k\geq K_{0}$ and $|j-m|\leq A\log m$, we have \begin{equation}\label{disper2}\sup_{|\mu|\leq N/2+2}\big\|e^{{i}t\Lambda}\Gamma^{\mu}S_{l}f_{jk}^{*}\big\|_{L^{\infty}}\lesssim\langle m\rangle^{-N_{0}+A}2^{-m+3k/2}.\end{equation}

\item If $k\geq K_{0}$ and $|j-m|\geq A\log m$, we have \begin{equation}\label{disper5}\sup_{|\mu|\leq N/2+2}\big\|e^{{i}t\Lambda}\Gamma^{\mu}f_{jk}^{*}\big\|_{L^{\infty}}\lesssim 2^{-(3m+j)/2}\min(2^{4k+l}\langle m\rangle^{-N_{0}+A},2^{2\delta(3m+j)}).\end{equation}

\end{enumerate}
\end{proposition}
\begin{proof} First we recall the standard dispersion estimate for the Klein-Gordon flow\begin{equation}\label{stkg}\|P_{k}e^{it\Lambda}f\|_{L^{\infty}}\lesssim 2^{-3m/2}(1+2^{3k})\|f\|_{L^{1}},\end{equation} see \cite{IP12}, Lemma 5.2.

Now (\ref{disper4}) is a direct consequence of Hausdorff-Young and the definition of the $Z$ norm; the bounds (\ref{disper1}) and (\ref{disper2}) are also easily deduced. In fact, when $(j,k)\in\mathcal{J}$ and $k\leq -K_{0}$, from Hausdorff-Young we have
\begin{equation}\big\|e^{{i}t\Lambda}\Gamma^{\mu}f_{jk}^{*}\big\|_{L^{\infty}}\lesssim\big\|\widehat{\Gamma^{\mu}f_{jk}^{*}}\big\|_{L^{1}}\lesssim\|\varphi_{[k-2,k+2]}(\xi)\|_{L^{2}}\|\widehat{\Gamma^{\mu}f_{jk}}\|_{L^{2}}\lesssim \langle j\rangle^{-N_{0}}2^{-j+k},\end{equation} while using (\ref{stkg}) and H\"{o}lder we have
\begin{equation}\big\|e^{{i}t\Lambda}\Gamma^{\mu}f_{jk}^{*}\big\|_{L^{\infty}}\lesssim2^{-3m/2}\|\Gamma^{\mu}f_{jk}\|_{L^{1}}\lesssim2^{-3(m-j)/2}\|\Gamma^{\mu}f_{jk}\|_{L^{2}}\lesssim \langle j\rangle^{-N_{0}}2^{-(3m-j+k)/2}.\end{equation} Similarly we have (\ref{disper2}).

Now we prove (\ref{disper3}) in the case $j\leq m-A\log m$. The case $j\geq m+A\log m$ only needs minor changes. Let $\Gamma^{\mu}f_{jk}=F$ and $\Gamma^{\mu}f_{jk}^{*}=F^{*}$; if $|x|\leq 2^{m}\langle m\rangle^{-A}$, recall that
\begin{equation}\label{express}(e^{{i}t\Lambda}F^{*})(x)=\int_{\mathbb{R}^{3}}e^{{i}t\Lambda(\xi)+{i}x\cdot \xi}\psi(\xi)\widehat{F}(\xi)\,\mathrm{d}\xi,\end{equation} where $\psi(\xi)=\varphi_{[k-2,k+2]}(\xi)$ is a cutoff because $|k|\leq K_{0}$. Since $F(z)$ is supported in $|z|\sim 2^{j}$ (unless $j=\max(-k,0)$, in which case $j=O(1)$ and the estimate will be trivial), we may rewrite the above expression as \begin{equation}\big(e^{{i}t\Lambda}F^{*}\big)(x)=\int_{\mathbb{R}^{3}}\psi_{0}(2^{-j}z)F(z)\,\mathrm{d}z\int_{\mathbb{R}^{3}}e^{{i}(t\Lambda(\xi)+(x-z)\cdot \xi)}\psi(\xi)\,\mathrm{d}\xi,\end{equation} with another cutoff $\psi_{0}(z)$ supported in the region $|z|\sim 1$. Now fix any $z$ such that $|z|\sim 2^{j}$, we can use Proposition \ref{ips0} with \[K=2^{\max(m,j)},\quad n=1,\quad \epsilon,\lambda\sim 1\] to bound the $\xi$-integral by $2^{-10m}$, which is clearly enough for (\ref{disper3}). The same argument also applies for $|x|\geq 2^{m}\langle m\rangle^{A}$.

Therefore, to prove (\ref{disper3}), we may assume $|x|\sim 2^{r}$ where $|r-m|\leq A\log m$; without loss of generality we may also assume $x=(\alpha,0,0)$, where $\alpha=|x|$. Decomposing $F$ into $S_{l}F$ as in (\ref{spherical}), we may also assume $l\leq 7\delta m$, since otherwise we have\[\|S_{l}F\|_{L^{2}}\lesssim 2^{-Nl/2}\sup_{|\nu|\leq N/2}\|\Omega^{\nu}S_{l}F\|_{L^{2}}\lesssim 2^{-7m/2},\] from which (\ref{disper2}) follows easily. Since the $\psi$ in (\ref{express}) is radial, we may write $\psi(|\xi|)\widehat{S_{l}F}(\xi)=g(|\xi|,\xi/|\xi|)$, then we have
\begin{equation}\label{polar}\big(e^{{i}t\Lambda}S_{l}F^{*}\big)(x)=\int_{\mathbb{R}}\int_{\mathbb{S}^{2}}e^{{i}(t\Lambda(\rho)+\alpha\rho\theta^{1})}\rho^{2}\psi_{1}(\rho)g(\rho,\theta)\,\mathrm{d}\rho\mathrm{d}\omega(\theta),\end{equation} where $\mathrm{d}\omega$ is the surface measure on $\mathbb{S}^{2}$ and $\psi_{1}$ is another cutoff.

Since $g$ is (qualitatively) a Schwartz function of $\rho$ and $\theta$, we may decompose $g=g_{1}+g_{2}$, where for each $\theta$, $\mathcal{F}_{\rho}g_{1}(\rho,\theta)(\tau)$ is supported in the region $|\tau|\lesssim M_{0}$, and $\mathcal{F}_{\rho}g_{2}(\rho,\theta)(\tau)$ is supported in the region $|\tau|\gg M_{0}$, where $M_{0}=2^{j}\langle m\rangle^{A/8}$. To estimate $g_{2}$, we only need to estimate \[\int_{|\tau|\gg M_{0}}\big|(\mathcal{F}_{\rho}g(\rho,\theta))(\tau)\big|\,\mathrm{d}\tau\] uniformly in $\theta$. But for each fixed $\tau$ with $|\tau|\gg M_{0}$ we have\begin{equation*}(\mathcal{F}_{\rho}g(\rho,\theta))(\tau)=\int_{\mathbb{R}}e^{-{i}\rho\tau}g(\rho,\theta)\,\mathrm{d}\rho=\int_{\mathbb{R}}e^{-{i}\rho\tau}\psi(\rho)\,\mathrm{d}\rho\int_{\mathbb{R}^{3}}\psi_{0}(2^{-j}z)S_{l}F(z)e^{-{i}\rho(\theta\cdot z)}\,\mathrm{d}z.\nonumber\end{equation*} Now if we fix $z$, then the integral in $\rho$ can be bounded by $(|\tau|+2^m)^{-10}$, which will be acceptable. Therefore in (\ref{polar}) we may replace the function $g$ by $g_{1}$. Using (\ref{sobolev}) and also noticing the bounded Fourier support of $g_{1}(\cdot,\theta)$ for each $\theta$, we have
\begin{equation}\label{sobolev2}\|g_{1}\|_{L_{\theta}^{\infty}L_{\rho}^{\infty}}\lesssim M_{0}^{1/2}\|g_{1}\|_{L_{\theta}^{\infty}L_{\rho}^{2}}\lesssim M_{0}^{1/2}\|g_{1}\|_{L_{\rho}^{2}L_{\theta}^{\infty}}\lesssim2^{l}M_{0}^{1/2}\|g_{1}\|_{L_{\rho}^{2}L_{\theta}^{2}}.\end{equation}
If, in the integral (\ref{polar}), we restrict to the region $|(\theta^{1})^{2}-1|\geq 1/100$, since the function $\theta\mapsto\theta^{1}$ has no critical point in this region, we can integrate by parts in $\theta$ many times to bound the left hand side of (\ref{polar}) by $2^{-10m}$, which implies (\ref{disper2}). Therefore, in (\ref{polar}), after replacing $g$ by $g_{1}$, we may also cutoff in the region $|\theta^{1}-1|\leq 1/50$ (the case $|\theta^{1}+1|\leq 1/50$ is treated in the same way), on which we can use $(\theta^{2},\theta^{3})$ as local coordinates, so that the integral (\ref{polar}) reduces to\begin{equation}\label{fin}I=\int_{\mathbb{R}\times\mathbb{R}^{2}}e^{{i}\big(t\Lambda(\rho)+\alpha\rho\sqrt{1-(\theta^{2})^{2}-(\theta^{3})^{2}}\big)}\psi_{1}(\rho)\psi_{2}(\theta^{2},\theta^{3})h(\rho,\theta^{2},\theta^{3})\,\mathrm{d}\rho\mathrm{d}\theta_{2}\mathrm{d}\theta_{3},\end{equation} where $\psi_{1}$ and $\psi_{2}$ are cutoff functions, and $h$ is obtained from $g_{1}$ after change of variables.

Now, recall $\alpha\sim 2^{r}$, let \begin{equation}\label{seteps}\epsilon'=2^{-r/2}\langle m\rangle^{2A};\qquad\epsilon''=2^{\max(j-m,-m/2)}\langle m\rangle^{2A},\end{equation} we proceed to estimate (\ref{fin}) in the region where $|\theta^{2}|+|\theta^{3}|\gtrsim \epsilon'$. After inserting a cutoff $(1-\varphi_{0})((\epsilon')^{-1}\theta^{2},(\epsilon')^{-1}\theta^{3})$, we will use Proposition \ref{ips0} to integrate by parts in $(\theta^{2},\theta^{3})$, for any fixed $\rho$, choosing \[K=\alpha, \quad n=1, \quad \epsilon\sim\epsilon', \quad\lambda\sim(\epsilon')^{-1},\] so that $I$ is bounded by $\exp(-\gamma\langle m\rangle^{A\gamma/2})$, which is $\lesssim 2^{-10m}$ if $A$ is large enough. Here we have used that $2^{l}\leq 2^{7\delta m}\leq(\epsilon')^{-1}$ and $\|\partial_{\theta}^{\mu}h(\rho,\theta)\|_{L_{\theta}^{2}}\lesssim 2^{|\mu|l}$, which is because $g_{1}(\rho,\cdot)$ is still a linear combination of spherical harmonics of degree $\lesssim 2^{l}$.

Now, in (\ref{fin}), we will restrict to the region where $|\theta^{2}|+|\theta^{3}|\ll \epsilon'$, and then fix $\theta^{2}$ and $\theta^{3}$. Let\[\Xi(\rho):=\partial_{\rho}\big(\Lambda(\rho)+t^{-1}\alpha\rho\theta^{1}\big)=\Lambda'(\rho)+t^{-1}\alpha\theta^{1},\] we will then consider the part where $|\Xi(\rho)|\gtrsim \epsilon''$. In this case, we will use Proposition \ref{ips0}, and set\[K=t,\quad n=1,\quad\epsilon\sim\epsilon'',\quad\lambda\sim\max(M_{0},(\epsilon'')^{-1})\] to bound $I\lesssim 2^{-10m}$. Here we have used $\|\partial_{\rho}^{\mu}h(\rho,\theta)\|_{L^{2}}\lesssim M_{0}^{|\mu|}$, because $\mathcal{F}_{\rho}g_{1}(\rho,\theta)(\tau)$ is supported in $|\tau|\lesssim M_{0}$.

Therefore, we can restrict to the region $|\Xi(\rho)|\lesssim \epsilon''$. Using H\"{o}lder and (\ref{sobolev2}), we have\[|I|\lesssim (\epsilon')^{2}\min\big(\epsilon''\|g_{1}\|_{L_{\theta}^{\infty}L_{\rho}^{\infty}},(\epsilon'')^{1/2}\|g_{1}\|_{L_{\theta}^{\infty}L_{\rho}^{2}}\big)\lesssim(\epsilon')^{2}2^{l}\min\big(\epsilon''M_{0}^{1/2},(\epsilon'')^{1/2}\big)\|g_{1}\|_{L_{\rho}^{2}L_{\theta}^{2}},\] which implies\begin{equation}\label{extraomega}|I|\lesssim \langle m\rangle^{A}2^{(j-3m)/2}2^{l}\|F\|_{L^{2}}\lesssim 2^{l}\langle m\rangle^{N_{0}+A}2^{-3m/2-j/3}\end{equation} by (\ref{part}). Since we trivially have\[\sup_{|\mu|\leq N/2+2}\big\|e^{it\Lambda}\Gamma^{\mu}S_{l}f_{jk}^{*}\big\|_{L^{\infty}}\lesssim2^{-(N/2-2)l},\] (\ref{disper3}) follows by interpolation.

Finally, (\ref{disper5}) is proved in the same way as (\ref{disper3}), with changes made in a few places: since $k\geq K_{0}$ and $\rho=|\xi|\sim 2^{k}$, in (\ref{polar}) the factor $\rho^{2}$ will count as $2^{2k}$, and the measure of the set where $|\Xi(\rho)|\lesssim \epsilon''$ is now $2^{3k}\epsilon''$ instead of $\epsilon''$ since $|\Lambda''(\rho)|\sim 2^{-3k}$. Moreover in (\ref{seteps}) we can set $\epsilon'=2^{-(r+k)/2}\langle m\rangle^{2A}$ instead of $2^{-r}\langle m\rangle^{2A}$, thus in total we have a loss of $2^{4k}$ compared with the proof of (\ref{disper3}), thus we have \[\sup_{|\mu|\leq N/2+2}\big\|e^{{i}t\Lambda}\Gamma^{\mu}f_{jk}^{*}\big\|_{L^{\infty}}\lesssim 2^{-(3m+j)/2}2^{4k+l}\langle m\rangle^{-N_{0}+A}.\] Since we trivially have \[\sup_{|\mu|\leq N/2+2}\big\|e^{{i}t\Lambda}\Gamma^{\mu}f_{jk}^{*}\big\|_{L^{\infty}}\lesssim \min(2^{-(N/2-2)k},2^{-(N/2-2)l}),\] we easily get (\ref{disper5}) by interpolation. 
\end{proof}
\begin{corollary}\label{cor} Suppose $f=f(x)$ is a function with $\|f\|_{X}\lesssim \varepsilon$, and let $v=e^{{i}t\Lambda}f$ with $\Lambda$ equaling one of the $\Lambda_{\alpha}$, then we have\begin{equation}\sum_{|\mu|\leq N/2}\|\Gamma^{\mu}v\|_{L^{\infty}}\lesssim\frac{\varepsilon}{(1+|t|)\log^{(N_{0}-2)}(2+|t|)}.\end{equation} 
\end{corollary}
\begin{proof} Let $1+|t|\sim 2^{m}$, and we decompose\begin{equation}v=\sum_{(j,k)\in\mathcal{J}}e^{{i}t\Lambda}f_{jk}^{*}.\end{equation} Now (\ref{disper1})-(\ref{disper5}) in particular implies\[\big\|e^{{i}t\Lambda}\Gamma^{\mu}f_{jk}^{*}\big\|_{L^{\infty}}\lesssim\varepsilon(1+2^{k/2})^{-1}(\max\langle m\rangle,\langle j\rangle)^{-N_{0}}2^{-\max(m,j)}\]for each $|\mu|\leq N/2$. Therefore
\begin{eqnarray}\|\Omega^{\mu}v\|_{L^{\infty}}&\lesssim &\sum_{(j,k)\in\mathcal{J}}\varepsilon(1+2^{k/2})^{-1}\max(\langle m\rangle,\langle j\rangle)^{-N_{0}}\cdot2^{-\max(m,j)}\nonumber\\
&\lesssim&\varepsilon\sum_{j\geq 0}2^{-\max(m,j)}\max(\langle m\rangle,\langle j\rangle)^{-N_{0}}\cdot\sum_{k=-j}^{\infty}(1+2^{k/2})^{-1}\nonumber\\
&\lesssim&\varepsilon\sum_{j\geq 0}2^{-\max(m,j)}(\max\langle m\rangle,\langle j\rangle)^{-(N_{0}-1)}\nonumber\\
&\lesssim &\varepsilon2^{-m}\langle m\rangle^{-(N_{0}-2)},
\end{eqnarray} which is what we need.
\end{proof}
\begin{proposition}[Basic bilinear estimates]\label{bilinear}
For the bilinear operator $T$ defined by
\begin{equation}\label{def}\mathcal{F}T(f,g)(\xi)=\int_{\mathbb{R}^{3}}K(\xi,\eta)\widehat{f}(\xi-\eta)\widehat{g}(\eta)\,\mathrm{d}\eta,\end{equation} we have the followings:
\begin{enumerate}
\item If $\|\mathcal{F}_{\xi,\eta}^{-1}K\|_{L^{1}}\leq 1$, and \[p,q,r\in[1,\infty],\qquad \frac{1}{r}=\frac{1}{p}+\frac{1}{q},\] then we have
\begin{equation}\label{biliest}\|T(f,g)\|_{L^{r}}\lesssim\|f\|_{L^{p}}\|g\|_{L^{q}}.\end{equation}
\item  If $K$ satisfies \[\sup_{\xi}\int_{\mathbb{R}^{3}}|K(\xi,\eta)|^{2}\,\mathrm{d}\eta+\sup_{\eta}\int_{\mathbb{R}^{3}}|K(\xi,\eta)|^{2}\,\mathrm{d}\xi\lesssim 1,\] then we have \begin{equation}\label{biliest2}\|T(f,g)\|_{L^{2}}\lesssim\|f\|_{L^{2}}\|g\|_{L^{2}}.\end{equation}
\item Suppose $f$ and $g$ are \emph{radial}, and that $K$ is bounded by $1$ and is supported where \[|\xi|\sim 2^{k},\quad|\xi-\eta|\sim 2^{k_{1}},\quad |\eta|\sim 2^{k_{2}};\quad|\Phi(\xi,\eta)|\leq\epsilon,\] where $\Phi=\Phi_{\sigma\mu\nu}$ is defined in (\ref{phasedef}). Then we have\begin{equation}\label{radimp}\|T(f,g)\|_{L^{2}}\lesssim\min\big(2^{k/2},2^{-\min(k,0)/2}\epsilon^{1/2}\big)2^{-(k_{1}+k_{2})}\|\widehat{f}\|_{L^{1}}\|\widehat{g}\|_{L^{1}}.\end{equation}
\end{enumerate}
\end{proposition}
\begin{proof} (1) This is standard; see \cite{IP12}.

(2) We may assume that $F=\widehat{f}$ and $G=\widehat{g}$ are nonnegative. Let $\mathcal{F}T(f,g)=H$, by Cauchy-Schwartz we have\[|H(\xi)|^{2}\lesssim\int_{\mathbb{R}^{3}}K(\xi,\eta)F(\xi-\eta)G^{2}(\eta)\,\mathrm{d}\eta\cdot\int_{\mathbb{R}^{3}}K(\xi,\eta)F(\xi-\eta)\,\mathrm{d}\eta.\] The second factor is bounded by $\|F\|_{L^{2}}$ by Cahchy-Schwartz again, thus we have\[\|H\|_{L^{2}}^{2}\lesssim \|f\|_{L^{2}}\cdot\int_{\mathbb{R}^{3}}\int_{\mathbb{R}^{3}}K(\xi,\eta)F(\xi-\eta)G^{2}(\eta)\,\mathrm{d}\eta\mathrm{d}\xi,\] where the last integral equals\[\int_{\mathbb{R}^{3}}G^{2}(\eta)\,\mathrm{d}\eta\int_{\mathbb{R}^{3}}K(\xi,\eta)F(\xi-\eta)\mathrm{d}\xi\lesssim\int_{\mathbb{R}^{3}}G^{2}(\eta)\cdot \|f\|_{L^{2}}\,\mathrm{d}\eta=\|f\|_{L^{2}}\cdot\|g\|_{L^{2}}^{2},\] so this proves (2).

(3) Let $|\widehat{f}(\xi-\eta)|=F(|\xi-\eta|)$ and $|\widehat{g}(\eta)|=G(|\eta|)$, we may assume $\xi=(\lambda,0,0)$ and $\xi-\eta=(x,y,z)$, so that\[|\mathcal{F}T(f,g)(\xi)|\leq\int_{\mathbb{R}^{3}}F(\sqrt{x^{2}+y^{2}+z^{2}})G(\sqrt{(\lambda-x)^{2}+y^{2}+z^{2}})\,\mathrm{d}x\mathrm{d}y\mathrm{d}z.\] Make the change of variables \[\rho=\sqrt{x^2+y^{2}+z^2},\quad\tau=\sqrt{(\lambda-x)^{2}+y^2+z^2},\quad\theta=\tan^{-1}\frac{z}{y},\]the (inverse) Jacobian being\[J^{-1}:=\left|\begin{array}{ccc}
\rho_{x}&\rho_{y}&\rho_{z}\\
\tau_{x}&\tau_{y}&\tau_{z}\\
\theta_{x}&\theta_{y}&\theta_{z}
\end{array}\right|=\frac{1}{\rho\tau(y^{2}+z^{2})}\left|\begin{array}{ccc}
x&y&z\\
x-\lambda&y&z\\
0&-z&y\end{array}\right|=\frac{\lambda}{\rho\tau},\] so we have\begin{equation}\label{radconv}|\mathcal{F}T(f,g)(\xi)|\leq\frac{2\pi}{\lambda}\int_{|\Phi|\leq\epsilon}\rho\tau F(\rho)G(\tau)\,\mathrm{d}\rho\mathrm{d}\tau.\end{equation} Note that \[\|F(\rho)\|_{L_{\rho}^{1}}\sim 2^{-2k_{1}}\|\widehat{f}\|_{L^{1}};\qquad \|G(\tau)\|_{L_{\tau}^{1}}\sim 2^{-2k_{2}}\|\widehat{g}\|_{L^{1}(},\]this implies that\[|\mathcal{F}T(f,g)(\xi)|\lesssim 2^{-k-k_{1}-k_{2}}\|\widehat{f}\|_{L^{1}}\|\widehat{g}\|_{L^{1}},\] which implies the first part of (\ref{radimp}) by H\"{o}lder; as for the second part, choose a suitable function $J(\lambda)$ with $\|J\|_{L^{2}}=1$, we have\[\|T(f,g)\|_{L^{2}}\lesssim 2^{k_{1}+k_{2}}\int_{|\Phi|\leq\epsilon}F(\rho)G(\tau)J(\lambda)\,\mathrm{d}\rho\mathrm{d}\tau\mathrm{d}\lambda,\] then we fix $\rho$ and $\tau$ and notice $|\partial_{\lambda}\Phi|\sim 2^{\min(k,0)}$, so the measure of $\{\lambda:|\Phi|\leq\epsilon\}$ is bounded by $2^{-\min(k,0)}\epsilon$, then use H\"{o}lder to conclude.
\end{proof}
\section{Proof of Theorem \ref{mainkg}: Reduction to $Z$-norm estimate}\label{energy}Let $u$ be a solution to (\ref{nlkg}) on $[0,T]$. We define the function $v$ with value in $\mathbb{C}^{2d}$ by
\begin{equation}v_{\nu}=(\partial_{t}-i\Lambda_{\nu})u_{\nu},\qquad\nu\in\mathcal{P},\end{equation} where $u_{-\alpha}=u_{\alpha}$; so we have $v_{-\alpha}=\overline{v_{\alpha}}$. Moreover, let the corresponding profile $f$ be\begin{equation}\label{deff}f_{\nu}(t)=e^{{i}t\Lambda_{\nu}}v_{\nu}(t),\qquad0\leq t\leq T.\end{equation}
\begin{proposition}\label{loc} Suppose $g,h:\mathbb{R}^{3}\to\mathbb{R}^{d}$ are such that $\|(g,\partial_{x}g,h)\|_{\mathcal{H}^{N}}\leq\varepsilon_{0}$, then there exists a unique solution $u$ to (\ref{nlkg}) such that
\begin{equation}u\in C_{t}^{1}\mathcal{H}_{x}^{N}([0,1]\times\mathbb{R}^{3}\to\mathbb{R}^{d}),\quad\partial_{x}u\in C_{t}^{0}\mathcal{H}_{x}^{N}([0,1]\times\mathbb{R}^{3}\to\mathbb{R}^{d});\quad u(0)=g,\partial_{t}u(0)=h.\end{equation} Moreover, if  $\|(g,\partial_{x}g,h)\|_{Z}\leq\varepsilon_{0}$, then we have $f(t)\in C([0,1]\to Z)$, where $f(t)=(f_{\nu}(t))$ is defined as in (\ref{deff}) above.
\end{proposition}
\begin{proof}This is proved, with slightly different parameters, in \cite{IP12}, Proposition 2.1 and 2.4; the proof in our case is basically the same.
\end{proof}
With Proposition \ref{loc}, we can reduce the proof of Theorem \ref{mainkg} to the following a priori estimate.
\begin{proposition}\label{bootstrap} Suppose $u$ is a solution to (\ref{nlkg}) on a time interval $[0,T]$ with initial data $u(0)=g$ and $u_{t}(0)=h$ such that \begin{equation}u\in C_{t}^{1}\mathcal{H}_{x}^{N}([0,T]\times\mathbb{R}^{3}\to\mathbb{R}^{d}),\quad\partial_{x}u\in C_{t}^{0}\mathcal{H}_{x}^{N}([0,T]\times\mathbb{R}^{3}\to\mathbb{R}^{d}),\end{equation}  and let $f(t)$ be defined accordingly. Assume \[\|(g,\partial_{x}g,h)\|_{X}\leq\varepsilon\leq\varepsilon_{0},\qquad \sup_{0\leq t\leq T}\|f(t)\|_{X}\leq\varepsilon_{1}\ll 1,\] then we have \[\sup_{0\leq t\leq T}\|f(t)\|_{X}\lesssim \varepsilon_{1}^{3/2}+\varepsilon.\]
\end{proposition}
\subsection{Control of energy}
From now on we will fix a solution $u$ as described in Proposition \ref{bootstrap}, and the corresponding $f$; in this section we will recover the energy bounds. 
\begin{proposition}\label{econtrol} We have \begin{equation}\label{resenergy}\sup_{0\leq t\leq T}\sup_{|\mu|\leq N}\|\Gamma^{\mu}f(t)\|_{L^{2}}\lesssim \varepsilon_{1}^{3/2}+\varepsilon.\end{equation}
\end{proposition}
\begin{proof} Recall that\[\|\Gamma^{\mu}f(t)\|_{L^{2}}\sim\|\Gamma^{\mu}u(t)\|_{L^{2}}+\|\Gamma^{\mu}(\partial_{x},\partial_{t})u(t)\|_{L^{2}},\]the proof is basically the same as in \cite{IP12}; we will present it here since the norms involved are different. Define the energy\begin{eqnarray}\mathcal{E}(t)&=&\sum_{|\mu|\leq N}\int_{\mathbb{R}^{3}}\bigg(\sum_{\alpha=1}^{d}(|\partial_{t}\Gamma^{\mu}u_{\alpha}|^{2}+b_{\alpha}^{2}|\Gamma^{\mu}u_{\alpha}|^{2}+c_{\alpha}^{2}|\nabla\Gamma^{\mu}u_{\alpha}|^{2})+\\
&+&\sum_{\alpha,\beta=1}^{d}\sum_{j,k=1}^{3}S_{\alpha\beta}^{jk}(u,\partial u)\partial_{j}\Gamma^{\mu}u_{\alpha}\cdot\partial_{k}\Gamma^{\mu}u_{\beta}\bigg)\,\mathrm{d}x\nonumber,
\end{eqnarray} where \[S_{\alpha\beta}^{jk}(u,\partial u)=\sum_{\gamma=1}^{d}\sum_{l=1}^{3}\big(A_{\alpha\beta\gamma}^{jk}u_{\gamma}+B_{\alpha\beta\gamma}^{jkl}\partial_{l}u_{\gamma}\big).\]Note that in the whole time interval $[0,T]$ the $H^{N}$ based norms are small, we thus have\[\mathcal{E}(t)\sim\|\Gamma^{\mu}u(t)\|_{L^{2}}^{2}+\|\Gamma^{\mu}(\partial_{x},\partial_{t})u(t)\|_{L^{2}}^{2}.\] Now using (\ref{nlkg}) and the symmetry assumption of $S$, and integrating by parts, we may compute that
\begin{eqnarray}\frac{1}{2}\partial_{t}\mathcal{E}(t)&=&\frac{1}{2}\sum_{|\mu|\leq N}\sum_{\alpha,\beta=1}^{d}\sum_{j,k=1}^{3}\int_{\mathbb{R}^{3}}\partial_{t}S_{\alpha\beta}^{jk}(u,\partial u)\cdot\partial_{j}\Gamma^{\mu}u_{\alpha}\cdot\partial_{k}\Gamma^{\mu}u_{\beta}\nonumber\\
&-&\sum_{|\mu|\leq N}\sum_{\alpha,\beta=1}^{d}\sum_{j,k=1}^{3}\int_{\mathbb{R}^{3}}\partial_{j}S_{\alpha\beta}^{jk}(u,\partial u)\cdot \partial_{t}\Gamma^{\mu}u_{\alpha}\cdot\partial_{k}\Gamma^{\mu}u_{\beta}\nonumber\\
&+&\sum_{|\mu|\leq N}\sum_{\alpha,\beta=1}^{d}\sum_{j,k=1}^{3}\partial_{t}\Gamma^{\mu}u_{\alpha}\cdot\big[\Gamma^{\mu}(S_{\alpha\beta}^{jk}(u,\partial u)\partial_{j}\partial_{k}u_{\beta})-S_{\alpha\beta}^{jk}(u,\partial u)\Gamma^{\mu}\partial_{j}\partial_{k}u_{\beta}\big]\nonumber\\
&+&\sum_{|\mu|\leq N}\sum_{\alpha=1}^{d}\partial_{t}\Omega^{\mu}u_{\alpha}\cdot\Gamma^{\mu}\mathcal{Q}_{\alpha}'(u,\partial u)\nonumber.
\end{eqnarray} Now, using the equation (\ref{nlkg}) again to eliminate $\partial_{t}^{2}$ terms and using Leibniz rule, we can bound the time derivative by \[|\partial_{t}\mathcal{E}(t)|\lesssim\mathcal{E}(t)\cdot\bigg(\sup_{|\mu|\leq N/2}\|\Gamma^{\mu}u(t)\|_{L^{\infty}}+\sup_{|\mu|\leq N/2}\|\Gamma^{\mu}(\partial_{t},\partial_{x})u(t)\|_{L^{\infty}}\bigg).\] Next, since we have\begin{equation}\label{recover}\partial_{t}u_{\alpha}=\frac{v_{\alpha}+v_{-\alpha}}{2};\qquad u_{\alpha}=\frac{i}{2}\Lambda_{\alpha}^{-1}(v_{\alpha}-v_{-\alpha}),\end{equation} and also $v_{\sigma}(t)=e^{-{i}t\Lambda_{\sigma}}f_{\sigma}(t)$ for $\nu\in\mathcal{P}$ with the bound $\|f_{\sigma}(t)\|_{X}\lesssim \varepsilon_{1}$, we can use Corollary \ref{cor} to deduce that\begin{equation}\label{decay00}\sup_{|\mu|\leq N/2}\|\Gamma^{\mu}u(t)\|_{L^{\infty}}+\sup_{|\mu|\leq N/2}\|\Gamma^{\mu}(\partial_{t},\partial_{x})u(t)\|_{L^{\infty}}\lesssim\frac{\varepsilon_{1}}{(1+|t|)\log^{20}(2+|t|)},\end{equation} and hence (note that $\mathcal{E}(t)\lesssim\varepsilon_{1}^{2}$)\[|\partial_{t}\mathcal{E}(t)|\lesssim\frac{\varepsilon_{1}^{3}}{(1+|t|)\log^{20}(2+|t|)}.\] Since $\mathcal{E}(0)\lesssim\varepsilon^{2}$ and the weight $(1+|t|)^{-1}(\log(2+|t|))^{-20}$ is integrable in $t$, this proves (\ref{resenergy}).
\end{proof}
\subsection{Duhamel formula, and control of $Z$ norm}By definition of $v$ and (\ref{nlkg}), we know that $(\partial_{t}-{i}\Lambda_{\sigma})v_{\sigma}$ equals a (constant coefficient) quadratic form involving at most the second derivative of $u$. Using also (\ref{recover}) and Duhamel formula, we obtain the equation
\begin{equation}\label{duhamelforf}\widehat{f_{\sigma}}(t,\xi)-\widehat{f_{\sigma}}(0,\xi)=\sum_{\mu,\nu\in\mathcal{P}}\int_{0}^{t}\int_{\mathbb{R}^{3}}e^{{i}s\Phi_{\sigma\mu\nu}(\xi,\eta)}{m}_{\sigma\mu\nu}(\xi,\eta)\widehat{f_{\mu}}(s,\xi-\eta)\widehat{f_{\nu}}(s,\eta)\,\mathrm{d}\eta\mathrm{d}s\end{equation} for each $\sigma\in\mathcal{P}$. The weight \begin{equation}\label{weightm}{m}_{\sigma\mu\nu}=\sum_{i=1}^{20}\sum_{k,k_{1},k_{2}}(1+2^{\max(k,k_{1},k_{2})})\psi_{kk_{1}k_{2}}^{\sigma\mu\nu,i,0}\bigg(\frac{\xi}{2^{k}}\bigg)\psi_{kk_{1}k_{2}}^{\sigma\mu\nu,i,1}\bigg(\frac{\xi-\eta}{2^{k_{1}}}\bigg)\psi_{kk_{1}k_{2}}^{\sigma\mu\nu,i,2}\bigg(\frac{\eta}{2^{k_{2}}}\bigg),\end{equation} where the $\psi$'s have the same compact support and are bounded uniformly in $\mathcal{G}_{6}$.

The proof of Proposition \ref{bootstrap} is now reduced to the following $Z$ norm estimate.
\begin{proposition}\label{main01} Fix a choice of $(\sigma,\mu,\nu)$. Suppose \[m\geq0, (j,k),(j_{1},k_{1}),(j_{2},k_{2})\in\mathcal{J};\qquad\max(m,j,|k|,j_{1},j_{2},|k_{1}|,|k_{2}|):=M.\] Let $2^{m}-1\leq a\leq b\leq 2^{m}$, and define the quantity $J$ by
\begin{equation}\label{bilibili}\widehat{J}(\xi)=\int_{a}^{b}\int_{\mathbb{R}^{3}}e^{{i}s\Phi_{\sigma\mu\nu}(\xi,\eta)}{m}_{\sigma\mu\nu}(\xi,\eta)\mathcal{F}_{x}(f_{\mu})_{j_{1}k_{1}}^{*}(s,\xi-\eta)\mathcal{F}_{x}(f_{\nu})_{j_{2}k_{2}}^{*}(s,\eta)\,\mathrm{d}\eta\mathrm{d}s,
\end{equation} then we have\begin{equation}\label{atomest}\|J_{jk}\|_{Z_{jk}}\lesssim \langle M\rangle^{-20}\varepsilon_{1}^{2}.\end{equation}
\end{proposition}
\section{Proof of Proposition \ref{main01}: The setup}\label{beginning}First, for each $|\beta|\leq N/2+2$ we have\footnote{Strictly speaking we should commute $\Gamma^{\sigma}$ with $Q_{jk}$, but commutators (produced by $\partial_x$ and $Q_{jk}$) are lower order and can be estimated easily so we omit them.} \begin{equation}\label{bilibili2}\widehat{\Gamma^{\beta}J}(\xi)=\sum_{|\beta_{1}|+|\beta_{2}|\leq|\beta|}\int_{a}^{b}\int_{\mathbb{R}^{3}}e^{{i}s\Phi_{\sigma\mu\nu}(\xi,\eta)}{m}_{\sigma\mu\nu}^{\beta_{1}\beta_{2}}(\xi,\eta)\mathcal{F}_{x}(\Gamma^{\beta_{1}}f_{\mu})_{j_{1}k_{1}}^{*}(s,\xi-\eta)\mathcal{F}_{x}(\Gamma^{\beta_{2}}f_{\nu})_{j_{2}k_{2}}^{*}(s,\eta)\,\mathrm{d}\eta\mathrm{d}s,
\end{equation}
where ${m}_{\sigma\mu\nu}^{\beta_{1}\beta_{2}}$ is obtained from ${m}_{\sigma\mu\nu}$ by applying $\Omega$, and has the same form as (\ref{weightm}) with uniform bounds for each $|\mu|\leq N/2+2$. In (\ref{bilibili2}), we further decompose $f_{\mu}$ into $S_{l_{1}}f_{\mu}$ and $f_{\nu}$ into $S_{l_{2}}f_{\nu}$, and reduce to estimating $I$, where
\begin{equation}\label{bilibili3}\widehat{I}(\xi)=\int_{a}^{b}\int_{\mathbb{R}^{3}}e^{{i}s\Phi(\xi,\eta)}{m}(\xi,\eta)\widehat{F}(s,\xi-\eta)\widehat{G}(s,\eta)\,\mathrm{d}\eta\mathrm{d}s,\end{equation} where the $(\sigma,\mu,\nu)$ subindices are omitted, and \[F=(\Gamma^{\beta_{1}}S_{l_{1}}f_{\mu})_{j_{1}k_{1}}^{*},\qquad G=(\Gamma^{\beta_{2}}S_{l_{2}}f_{\nu})_{j_{2}k_{2}}^{*}.\] 
\subsection{Bounds for $F$, $G$ and their time derivatives}
\begin{proposition}\label{linearbound0}We have the following bounds for $F$ and $\partial_{t}F$; similar bounds will hold for $G$ and $\partial_{t}G$.
\begin{enumerate}
\item For $F$ we have\begin{equation}\label{linearb1}
\begin{aligned}\|F\|_{L^{2}}&\lesssim 2^{-(N-4)\max(k_{1},l_{1})/2}\varepsilon_{1};&\\\|\widehat{F}\|_{L^{\infty}}&\lesssim 2^{-j_{1}/4}\varepsilon_{1},&\mathrm{if\ }k_{1}\geq -K_{0}^{2};\\\|\widehat{F}\|_{L^{\infty}}&\lesssim 2^{l_{1}}2^{(-3k_{1}-j_{1})/2}\langle j_{1}\rangle^{-N_{0}}\varepsilon_{1},&\mathrm{if\ }k_{1}\leq-K_{0}^{2}.
\end{aligned}\end{equation}
\item If $k_{1}\leq -K_{0}$ we have \begin{equation}\label{linearb2}
\begin{aligned}\|F\|_{L^{2}}&\lesssim \langle j_{1}\rangle^{-N_{0}}2^{-j_{1}-k_{1}/2}\varepsilon_{1}\lesssim \langle j_{1}\rangle^{-N_{0}}2^{-j_{1}/2}\varepsilon_{1};\\\|e^{-it\Lambda_{\mu}}F\|_{L^{\infty}}&\lesssim \langle j_{1}\rangle^{-N_{0}}\min(2^{-j_{1}+k_{1}},2^{-(3m-j_{1}+k_{1})/2})\varepsilon_{1};\\\|e^{-it\Lambda_{\mu}}F\|_{L^{\infty}}&\lesssim\max(\langle m\rangle,\langle j_{1}\rangle)^{-N_{0}}2^{-\max(m,j_{1})}\varepsilon_{1}.
\end{aligned}\end{equation}
\item If $|k_{1}|<K_{0}$ we have
\begin{equation}\label{linearb3}
\begin{aligned}\|F\|_{L^{2}}&\lesssim 2^{-5j_{1}/6}\langle j_{1}\rangle^{N_{0}}\epsilon_{1},&\mathrm{and\ }\|\widehat{F}\|_{L^{1}}&\lesssim 2^{-j_{1}}\langle j_{1}\rangle^{-N_{0}}\varepsilon_{1};\\
\sup_{\theta\in \mathbb{S}^2}\|\widehat{F}(\rho\theta)\|_{L_{\rho}^2}&\lesssim 2^{l_1}2^{-5j_{1}/6}\langle j_{1}\rangle^{N_{0}}\epsilon_{1},&\sup_{\theta\in \mathbb{S}^2}\|\widehat{F}(\rho\theta)\|_{L_{\rho}^1}&\lesssim 2^{2l_1}2^{-j_{1}}\langle j_{1}\rangle^{-N_{0}}\epsilon_{1};\\\|e^{-it\Lambda_{\mu}}F\|_{L^{\infty}}&\lesssim2^{l_{1}}2^{-3m/2-j_{1}/3}\langle m\rangle^{N_{0}+A}\varepsilon_{1},&\mathrm{if\ }|j_{1}-m|&\geq A\log m;\\\|e^{-it\Lambda_{\mu}}F\|_{L^{\infty}}&\lesssim \langle m\rangle^{-N_{0}}2^{-j_{1}}\varepsilon_{1},&\mathrm{if\ }|j_{1}-m|&\leq A\log m.
\end{aligned}\end{equation}
\item If $k_{1}\geq K_{0}$ we have \begin{equation}\label{linearb4}
\begin{aligned}\|F\|_{L^{2}}&\lesssim  \langle j_{1}\rangle^{-N_{0}}2^{-j_{1}}\varepsilon_{1};&\\\|e^{-it\Lambda_{\mu}}F\|_{L^{\infty}}&\lesssim \langle j_{1}\rangle^{-N_{0}+A}2^{4k_1+l_{1}}2^{-(3m+j_{1})/2}\varepsilon_{1},&\mathrm{if\ }|j_{1}-m|\geq A\log m\\\|e^{-it\Lambda_{\mu}}F\|_{L^{\infty}}&\lesssim \langle m\rangle^{-N_{0}+A}2^{-j_{1}+3k_1/2}\varepsilon_{1},&\mathrm{if\ }|j_{1}-m|\leq A\log m.
\end{aligned}\end{equation}
\item For $\partial_{t}F$ we have\begin{equation}
\begin{aligned}\label{linearb5}\|\partial_{t}F\|_{L^{2}}&\lesssim\langle m\rangle^{A}\min(2^{-(3m+k_{1})/2},2^{3k_{1}/2})\varepsilon_{1}^{2}\lesssim\langle m\rangle^{A}2^{-9m/8}\varepsilon_{1}^{2};\\
\|\widehat{\partial_{t}F}\|_{L^{\infty}}&\lesssim \langle m\rangle^{A}2^{2l_1}\min(2^{-(3m-j_1-k_1)/2},2^{(2k_{1}+j_1/2)})\varepsilon^2.
\end{aligned}\end{equation}
\end{enumerate}
\end{proposition}
\begin{proof} The $L^{2}$ bound in (\ref{linearb1}) follows from the energy bound; for the Fourier $L^{\infty}$ bound when $k_1\geq -K_{0}^2$, we may assume $k_{1}\leq 6\delta j_{1}$, and decompose $\widehat{F}$ into spherical harmonics as in (\ref{decomp}), assuming also $l_{1}\leq 6\delta j_{1}$; using Proposition \ref{radial} we obtain that\[\|\widehat{F}\|_{L^{\infty}}\lesssim 2^{j_{1}/2+l_{1}}2^{-(5/6-o)j_{1}}\varepsilon_{1}\lesssim 2^{-j_{1}/4}\varepsilon_{1}.\] The case $k_1\leq -K_0^2$ is settled in the same way, the only difference being that now $\|G(|\xi|)\|_{L_{\xi}^{2}}\sim 2^{-k_{1}}\|G\|_{L_{\mathbb{R}}^{2}}$ for radial $G$ supported in $|\xi|\sim 2^{k_{1}}$, which introduces an additional factor of $2^{-k_{1}}$. Also, the bound on $\sup_{\theta\in \mathbb{S}^2}\|\widehat{F}(\rho\theta)\|_{L_{\rho}^2}$ in (\ref{linearb3}) follows from the same arguments, but without using one dimensional Sobolev, and the Fourier $L^1$ bound follows from directly summing over the angular modes (there are $\sim 2^{2l_1}$ of them).

Next, notice that everything else appearing in (\ref{linearb2}), (\ref{linearb3}) and (\ref{linearb4}) follows from either Definition \ref{normdef} and H\"{o}lder, or Proposition \ref{disper0}, so we only need to prove (\ref{linearb5}).  We start with the $L^2$ bound; recall from (\ref{duhamelforf}) we have that\begin{multline*}\widehat{\partial_{t}F}(t,\xi)=S_{l_{1}}P_{[k_{1}-2,k_{1}+2]}Q_{j_{1}k_{1}}\sum_{\gamma,\tau\in\mathcal{P}}\sum_{j_{2},k_{2},j_{3},k_{3}}\sum_{|\beta_{2}|+|\beta_{3}|\leq|\beta_{1}|}\int_{\mathbb{R}^{3}}e^{{i}s\Phi_{\mu\gamma\tau}(\xi,\eta)}\\
\times{m}_{\mu\gamma\tau}^{\beta_{2}\beta_{3}}(\xi,\eta)\mathcal{F}_{x}(\Gamma^{\beta_{2}}f_{\gamma})_{j_{2}k_{2}}^{*}(s,\xi-\eta)\mathcal{F}_{x}(\Gamma^{\beta_{3}}f_{\tau})_{j_{3}k_{3}}^{*}(s,\eta)\,\mathrm{d}\eta,\end{multline*} similar to (\ref{bilibili2}). Again we make the further $S_{l_{2}}$ and $S_{l_{3}}$ decompositions and reduce to estimating\[I'=\varphi_{[k_{1}-2,k_{1}+2]}(\xi)\int_{\mathbb{R}^{3}}e^{{i}s\Phi(\xi,\eta)}{m}(\xi,\eta)\widehat{K}(s,\xi-\eta)\widehat{L}(s,\eta)\,\mathrm{d}\eta,\] where we assume $\max(k_{2},k_{3},l_{2},l_{3})\leq6\delta m$, and\[K=(\Gamma^{\beta_{2}}S_{l_{2}}f_{\gamma})_{j_{2}k_{2}}^{*},\qquad L=(\Gamma^{\beta_{3}}S_{l_{3}}f_{\tau})_{j_{3}k_{3}}^{*},\] which satisfy the bounds in (\ref{linearb1})-(\ref{linearb4}) above. Now we have\[\|I'\|_{L^{2}}\lesssim 2^{3k_{1}/2}\|I'\|_{L^{\infty}}\lesssim 2^{3k_{1}/2}\|K\|_{L^{2}}\|L\|_{L^{2}}\lesssim 2^{3k_{1}/2}\varepsilon_{1}^{2};\]moreover, using Proposition \ref{bilinear}, we have \begin{equation}\label{basiccon0}\|I'\|_{L^{2}}\lesssim (1+2^{k_2}+2^{k_3})\min\big(\|K\|_{L^{2}}\|e^{-is\Lambda_{\tau}}L\|_{L^{\infty}},\|L\|_{L^{2}}\|e^{-is\Lambda_{\gamma}}K\|_{L^{\infty}}\big).\end{equation}If $\max(k_{2},k_{3})\leq -K_{0}^{2}$ then (\ref{basiccon0}) gives $\|I'\|_{L^{2}}\lesssim2^{\min(\kappa_{1},\kappa_{2})}\varepsilon_{1}^{2}$, where\[\kappa_{1}=-3m/2+(j_{2}-k_{2})/2-j_{3}-k_{3}/2,\qquad \kappa_{2}=-3m/2+(j_{3}-k_{3})/2-j_{2}-k_{2}/2.\] Assume $j_{2}\geq j_{3}$, then\[\kappa_{2}\leq -3m/2-j_{2}/2-(k_{2}+k_{3})/2\leq-3m/2-\max(k_{2},k_{3})\leq -3m/2-k_{1}/2,\] thus the $L^{2}$ bound in (\ref{linearb5}) holds. Otherwise, if $\min(k_{2},k_{3})\geq -K_{0}^{2}$, then (\ref{basiccon0}) implies $\|I'\|_{L^{2}}\lesssim 2^{-1.9m}\varepsilon_{1}^{2}$ if $\max(j_2,j_3)\geq (1-o)m$, and\[\|I'\|_{L^{2}}\lesssim \min\big\{2^{4k_{2}+l_{2}}2^{-3m/2}2^{-10\max(k_{3},l_{3})},2^{4k_{3}+l_{3}}2^{-3m/2}2^{-10\max(k_{2},l_{2})}\big\}\varepsilon_{1}^{2}\] if $\max(j_{2},j_{3})\leq (1-o)m$, which clearly suffices. Finally if $k_{2}\geq -K_{0}^{2}\geq k_{3}$, then we may also assume $k_{2}> -K_{0}$. If $j_{2}\geq (1-o)m$ then (\ref{basiccon0}) gives the correct bound as before, and if $j_{2}\leq (1-o)m$ we have by (\ref{linearb1}) and (\ref{linearb4}) that \[\|I'\|_{L^{2}}\lesssim \min\big\{2^{4k_{2}+l_{2}}2^{-3m/2}2^{-j_{3}/2},2^{-3m/2}2^{j_{3}}2^{-10\max(k_{2},l_{2})}\big\}\varepsilon_{1}^{2},\] which again suffices.

This proves the $L^{2}$ bound in (\ref{linearb5}). As for the $L^{\infty}$ bound, we simply use the fact that \[\big\|\mathcal{F}S_{l}f_{jk}^{*}\big\|_{L^{\infty}}\lesssim 2^{2l+j/2-k}\|f_{jk}\|_{L^{2}}\] which can be shown by using the spherical harmonics decomposition in (\ref{decomp}) as above.
\end{proof}
\subsection{Easy cases}\label{beginning0} We now begin to prove the estimate of $I$. By (\ref{basiccon0}), we have the following basic estimate
\begin{equation}\label{basiccon}\|I\|_{L^{2}}\lesssim 2^{m}\sup_{a\leq s\leq b}(1+2^{k_1}+2^{k_2})\min\big(\|F(s)\|_{L^{2}}\|e^{-s\Lambda_{\tau}}G(s)\|_{L^{\infty}},\|G(s)\|_{L^{2}}\|e^{-s\Lambda_{\sigma}}F(s)\|_{L^{\infty}}\big),\end{equation} which will be used repeatedly below.

First, we will get rid of the case when\[j\geq \max(-k,\min(j_{1},j_{2}),m+[k])+A\log M,\] where $[k]:=\min(\max(k_1,k_2),0)$, and recall that $A\ll N_{0}$ is some absolute constant. In fact, assuming $j_{1}\leq j_{2}$, we have\begin{equation}\label{jbig}I_{jk}(x)=\int_{\mathbb{R}^{3}}\varphi_{k}(\xi)e^{ix\cdot \xi}\,\mathrm{d}\xi\int_{a}^{b}\int_{\mathbb{R}^{3}}e^{{i}s\Phi(\xi,\eta)}{m}(\xi,\eta)\widehat{F}(s,\xi-\eta)\widehat{G}(s,\eta)\,\mathrm{d}\eta\mathrm{d}s\end{equation} for $|x|\sim 2^{j}$. Fixing $s$ and $\eta$, we analyze the integral in $\xi$, which is\[\int_{\mathbb{R}^{3}}\varphi_{k}(\xi)e^{i(x\cdot\xi+s\Phi(\xi,\eta))}{m}(\xi,\eta)\widehat{F}(s,\xi-\eta)\,\mathrm{d}\xi,\] by using Proposition \ref{ips0}, choosing \[K=|x|,\quad n=1,\quad\epsilon\sim1,\quad\lambda\sim2^{\max(0,-k,j_{1})},\] and noticing $|\nabla_{\xi}\Phi|\lesssim 2^{[k]}$, so that we can bound the output $\|I_{jk}\|_{L^{2}}\lesssim 2^{-10D}\varepsilon_{1}^{2}$.

Next, suppose $j\geq 5m/2$. From above, we may also assume that either $M\geq m^{2}$ (in which case $\max(j_{1},k_{2},k_{1},k_{2})\geq M$ and the estimates are trivial), or \[j\leq |k|+A\log m,\quad\mathrm{or}\quad\min(j_{1},j_{2})\geq j-A\log m.\] In the former case, we only need to prove $\|\widehat{I}\|_{L^{\infty}}\lesssim 2^{5j/6}\varepsilon_{1}^{2}$, which follows from estimating both $F$ and $G$ factors in $L^{2}$ in (\ref{bilibili3}), using (\ref{linearb1}). In the latter case, we simply use (\ref{basiccon}) and Proposition \ref{linearbound0}, and exploit the fact $j\geq 5m/2$ to conclude.

Now, assuming $j\leq 5m/2$, we can easily treat the case when $\max(k_{1},l_{1},k_{2},l_{2})\geq6\delta m$ or when $\max(j_{1},j_{2})\geq4m$, using (\ref{linearb1}) and (\ref{linearb2}) respectively. Thus from now on we will assume the inequalities\begin{equation}\label{basic0}j\leq 5m/2,\quad\max(k_{1},k_{2},l_{1},l_{2})\leq6\delta m,\quad M\leq 4m,\end{equation} and\begin{equation}\label{basicj}j\leq \max(-k,\min(j_{1},j_{2}),m+[k])+A\log m.\end{equation} For simplicity, from now on we will use the symbol $X\preceq Y$ to denote $X\leq Y+A\log m$ (and similarly for $X\succ Y$).

\section{Low frequencies}\label{lowlowtolow}In this section we consider the case $\max(k_{1},k_{2})\leq -K_{0}$. By (\ref{normcomp}), we may assume $k\leq -K_{0}$ also. Define $\widetilde{F}=\langle j_{1}\rangle^{N_{0}}F$ and $\widetilde{G}=\langle j_{2}\rangle^{N_{0}}G$, and $\widetilde{I}=\langle j_{1}\rangle^{N_{0}}\langle j_{2}\rangle^{N_{0}}I$; we then only need to prove \begin{equation}\label{lowgoal}2^{\kappa}:=L=2^{j+k/2}\varepsilon_{1}^{-2}\|\widetilde{I}_{jk}\|_{L^{2}}\lesssim \langle M\rangle^{-30}\bigg(\frac{\langle j_{1}\rangle\cdot\langle j_{2}\rangle}{\langle j\rangle}\bigg)^{N_{0}}.\end{equation} In each case below, we will prove that either $L\lesssim 2^{-o\cdot m}$, or $L\lesssim \langle M\rangle^{A}$ and $\min(j_{1},j_{2})\geq o\cdot m$; either will imply (\ref{lowgoal}). During the proof we will use $\rho_i$ to denote some constants that depend only on $\mathcal{B}$.
\subsection{First reduction}\label{firstred}
First assume $j\succ m+\max(k_{1},k_{2})$. We then must have either $j\preceq\min(j_{1},j_{2})$, or $j+k\preceq 0$. In the first situation we may, without loss of generality, assume $k\leq k_{1}+6$, and use (\ref{basiccon}), estimating $\widetilde{F}$ in $L^{2}$, to bound that\[\kappa\preceq j+\frac{k}{2}+m-j_{1}-\frac{k_{1}}{2}-\max(j_2-k_2,(3m-j_2+k_2)/2)\preceq 0.\] Moreover we will actually have $\kappa\leq -m/10$, unless $|j_{2}-k_{2}-m|\leq m/5$ and $|j-j_{1}|\leq m/10$, in which case we must have $\min(j_{1},j_{2})\geq o\cdot m$, so (\ref{lowgoal}) is true.

In the second situation we have $j+k\preceq 0$, thus if we define $2^{\kappa'}=\varepsilon_1^{-2}\|\widehat{P_{k}\widetilde{I}}\|_{L^{\infty}}$, then $\kappa\preceq\kappa'-j$. Since we have assumed that $j\succeq m+\frac{k_{1}+k_{2}}{2}$, we can directly use Young's inequality (estimating both $\widetilde{F}$ and $\widetilde{G}$ in $L^2$) to bound
\[\kappa'\preceq m-j_{1}-\frac{k_{1}}{2}-j_{2}-\frac{k_{2}}{2}\preceq m+\frac{k_{1}+k_{2}}{2}\preceq j.\] Again, we would actually have $\kappa'\leq j-o\cdot m$ from above, unless $|j_{1}+k_{1}|+|j_{2}+k_{2}|+|k_{1}-k_{2}|\leq o\cdot m$, which means we only need to consider the case \begin{equation}\label{typical}j_{1},k_{1},j_{2},k_{2}=o\cdot m,\qquad j,-k=(1+o)m.\end{equation}This would imply that that $\widetilde{F}$ and $\widetilde{G}$ are in Schwartz space with suitable Schwartz norm bounded by $2^{o\cdot m}\varepsilon_{1}$, so we will treat the integral in $\eta$ in (\ref{bilibili3}) as a standard oscillatory integral  with phase $s\Phi(\xi,\eta)$. Using Van der Corput lemma (see Proposition \ref{propp}), we can bound this integral for each $\xi$ by $2^{-m/10}\varepsilon_{1}^{2}$ (which would imply $\kappa'\preceq j-m/10$ since $j\geq (1-o)m$), unless $\Lambda_{\mu}+\Lambda_{\nu}=0$. In this final scenario we have $|\Phi(\xi,\eta)|\gtrsim 1$, so  we can integrate by parts in $s$ to obtain
\begin{multline}\label{ibps00}\mathcal{F}\widetilde{I}(\xi)=\int_{\mathbb{R}^{3}}e^{{i}\Phi(\xi,\eta)}\frac{{m}(\xi,\eta)}{{i}\Phi(\xi,\eta)}\widehat{F}(s,\xi-\eta)\widehat{G}(s,\eta)\,\mathrm{d}\eta\bigg|_{s=a}^{s=b}\\-\int_{a}^{b}\int_{\mathbb{R}^{3}}e^{{i}s\Phi(\xi,\eta)}\frac{{m}(\xi,\eta)}{{i}\Phi(\xi,\eta)}\widehat{\partial_{s}F}(s,\xi-\eta)\widehat{\partial_{s}G}(s,\eta)\,\mathrm{d}s\mathrm{d}\eta\\
-\int_{a}^{b}\int_{\mathbb{R}^{3}}e^{{i}s\Phi(\xi,\eta)}\frac{{m}(\xi,\eta)}{{i}\Phi(\xi,\eta)}\widehat{F}(s,\xi-\eta)\widehat{\partial_{s}G}(s,\eta)\,\mathrm{d}s\mathrm{d}\eta.\end{multline} Denote the three terms by $\widehat{I_{0}}$, $\widehat{I_{1}}$ and $\widehat{I_{2}}$ respectively, by symmetry, we only need to consider $I_{0}$ and $I_{1}$. Since $|\Phi|\gtrsim 1$, (\ref{basiccon}) will remain true for the particular bilinear operator involved in $I_{0}$ and $I_{1}$; combining this with Proposition \ref{linearbound0} we conclude that $\kappa\preceq\max(\kappa_{1},\kappa_{2})$
where \begin{equation}\label{ibpsest}
\begin{aligned}2^{\kappa_{1}}&=\varepsilon_{1}^{-2}\|I_{0}\|_{L^{2}},&\kappa_{1}&\preceq j+\frac{k}{2}-\max\big(m+\frac{j_{1}}{2},\frac{3m-j_{1}+k_{1}}{2}\big)\preceq-\frac{m}{8},\\2^{\kappa_{2}}&=\varepsilon_1^{-2}\|I_{1}\|_{L^{2}},&\kappa_{2}&\preceq j+\frac{k}{2}+m+\min\big(\frac{-3m-k_{1}}{2},\frac{3k_{1}}{2}\big)-m\preceq -\frac{m}{8},\end{aligned}
\end{equation} which again implies (\ref{lowgoal}).

Now we will assume $j\preceq m+\max(k_{1},k_{2})$. Note that from the above proof, we can also assume $j\succ\min(j_{1},j_{2})$. If $j+k\preceq 0$, we may assume $j_{1}\geq j_{2}$ and use (\ref{basiccon}), estimating $\widetilde{F}$ in $L^{2}$, to get\[\kappa\preceq \frac{j}{2}+m-j_{1}-\frac{k_{1}}{2}-\frac{3m}{2}+\frac{j_{2}-k_{2}}{2}\preceq0,\] which is because\[-j_{1}+\frac{j_{2}}{2}-\frac{k_{1}+k_{2}}{2}\leq-\frac{\max(j_{1},j_{2})}{2}-\frac{k_{1}+k_{2}}{2}\leq-\frac{\max(k_{1},k_{2})}{2}.\]This proves (\ref{lowgoal}) when $j_{1},j_{2}\geq o\cdot m$, and if $j_{2}\leq o\cdot m$ then either the stronger bound $\kappa\leq -o\cdot m$ holds, or we have (\ref{typical}). In any case this contribution will be acceptable.

Now we have \[\max(-k,\min(j_{1},j_{2}))\prec j\preceq m+\max(k_{1},k_{2}).\]If $b_{\sigma}-b_{\mu}-b_{\nu}\neq 0$, then $|\Phi|\gtrsim 1$, so we may integrate by parts in $s$ as in (\ref{ibps00}), and estimate the terms $I_0$, $I_1$ and $I_2$ in the same way as before, and notice that (\ref{ibpsest}) is still true in this situation, so (\ref{lowgoal}) still holds.

Now we assume $b_{\sigma}-b_{\mu}-b_{\nu}=0$. We will first treat the easier part when $k_{2}-k\leq -D_{0}$; under this assumption we have $|k_{1}-k|\leq 6$ and hence $|\nabla_{\eta}\Phi(\xi,\eta)|\sim 2^{k}$. We may then assume that $\max(j_{1},j_{2})\succeq m+k$ or we could integrate by parts in $\eta$ and choose\[K=2^{m},\quad n=1,\quad\epsilon\sim 2^{k},\quad\lambda\sim 2^{\max(j_{1},j_{2})}\] in Proposition \ref{ips0} to bound $\|\widetilde{I}\|_{L^{2}}\lesssim 2^{-10m}\varepsilon_{1}^{2}$.

If $j_{1}\succeq \max(j_{2},m+k)$, since we also have $j\preceq m+k$, we will use (\ref{basiccon}), estimating $\widetilde{F}$ in $L^{2}$, to get
\[\kappa\preceq j+\frac{k}{2}+m-j_{1}-\frac{k}{2}-m\preceq0,\] and again we will have $\kappa\leq -D/10$ unless $j_{2}\geq m/10$, thus this contribution will also be acceptable. On the other hand, if $j_{2}\succeq \max(j_{1},m+k)$, we will have $\kappa\preceq\min(\kappa_{1},\kappa_{2})$, where \[\kappa_{1}\preceq j+\frac{k}{2}+m-j_{1}-\frac{k}{2}-j_{2}+k_{2}\preceq m-j_{1}+k_{2}\] is obtained from estimating $\widetilde{F}$ in $L^{2}$, and \[\kappa_{2}\preceq j+\frac{k}{2}+m-\frac{3m}{2}+\frac{j_{1}-k}{2}-j_{2}-\frac{k_{2}}{2}\preceq\frac{-m+j_{1}-k_{2}}{2}\] is obtained from estimating $\widetilde{G}$ in $L^{2}$. Now at least one of $\kappa_{1}$ and $\kappa_{2}$ must be $\preceq0$, therefore we will get the desired bound up to a logarithmic loss, which we can always recover (meaning we have $\min(\kappa_1,\kappa_2)\leq -o\cdot m$) unless \[j_{1},k=o\cdot m,\qquad j,j_{2},-k_{2}=(1+o)m.\] In this final scenario, we will integrate by parts in $s$ to produce $I_{0}$, $I_{1}$ and $I_{2}$ terms. Notice that $|\zeta|:=|\xi-\eta|\sim 2^{k}$, we have $\Phi(\xi,\eta)=\Phi(\xi,0)+\eta\cdot\Psi(\xi,\eta)$, where $\Phi(\xi,0)$ is a fixed nonzero analytic function of $|\xi|^{2}$. Using Taylor expansion and a scaling argument, we can see that \[\widetilde{{m}}(\xi,\eta)=\frac{{m}(\xi,\eta)}{\Phi(\xi,\eta)}\psi_{1}(2^{-k}(\xi-\eta))\psi_{2}(2^{-k_{2}}\eta),\] where $\psi_{1}$ is supported in $|\xi|\sim 1$ and $\psi_{2}$ in $|\xi|\lesssim 1$, verifies the bound $\|\mathcal{F}\widetilde{{m}}\|_{L^{1}}\lesssim 2^{o(m)}$, so we can close the estimate by integrating by parts in time and repeating the arguments in (\ref{ibps00}) and (\ref{ibpsest}) above.

From now on, under the assumption $\max(k,k_{1},k_{2})\leq -K_{0}$, we can assume
\[\max(\min(j_{1},j_{2}),-k)\prec j\preceq m+k_{1};\,\,\,\,|k_{1}-k_{2}|\leq D_{0};\,\,\,\,k\leq k_{1}+D_{0}.\]Also recall that at point $(0,0)$, we may expand the phase function as
\begin{equation}\Phi(\xi,\eta)=\frac{c_{\sigma}^{2}}{2b_{\sigma}}|\xi|^{2}-\frac{c_{\mu}^{2}}{2b_{\mu}}|\xi-\eta|^{2}-\frac{c_{\nu}^{2}}{2b_{\nu}}|\eta|^{2}+O(|\xi|^{4}+|\eta|^{4}).\end{equation} 
\subsection{Second reduction}\label{secondredkg}
Here we assume $c_{\mu}^{2}/b_{\mu}+c_{\nu}^{2}/b_{\nu}=0$. Since $b_{\mu}+b_{\nu}\neq 0$, we must have $\rho_{2}:=-c_{\mu}^{4}/(8b_{\mu}^{3})-c_{\nu}^{4}/(8b_{\nu}^{3})\neq 0$. The phase function is now expanded as
\begin{equation}\Phi(\xi,\eta)=\rho_{0}|\xi|^{2}+\rho_{1}(\xi\cdot\eta)+\rho_{2}|\eta|^{4}+O(|\eta|^{6}+|\xi||\eta|^{3}),\end{equation} for some constants $\rho_0$ and $\rho_1$ with $\rho_{1}\rho_{2}\neq 0$. Note that $|\eta|\sim|\xi-\eta|\sim 2^{k_{1}}$, we will first exclude the case $k_{1}\leq k+D_{0}$. In fact, in this case we would have $|\nabla_{\eta}\Phi|\sim 2^{k}$ in the region of integration, under which assumption every estimate is exactly the same as the case when $k_{2}-k\leq -D_{0}$, which was studied at the end of Section \ref{firstred} above, and we will get acceptable contributions. Therefore, we may assume below that $k_{1}\geq k+D_{0}$.

(1) First, assume $3k_{1}\leq k-D_{0}^{2}$, then we will have $|\nabla_{\eta}\Phi(\xi,\eta)|\sim 2^{k}$ in the region of integration. Moreover, We have $|\nabla_{\eta}^{\nu}\Phi(\xi,\eta)|\lesssim 2^{(4-|\nu|)k_{1}}$ in the region of interest for $2\leq |\nu|\leq 4$. Now if $\max(j_{1},j_{2})\prec m+k$, we will use Proposition \ref{ips0} and take \[K=2^{m}, \quad n=3,\quad \epsilon\sim2^{k}, \quad\lambda\sim2^{\max(j_{1},j_{2})}\] and deduce that\[\|I_{jk}\|_{L^{2}}\lesssim \exp(-c(\min(2^{m+4k/3},2^{m+k-j_{1}},2^{m+k-j_{2}}))^{c})\varepsilon_{1}^{2}\] which will be sufficient, since \[m+4k/3\succeq m+k+k_{1}\succeq m+k-j_{1}\succ 0.\] Therefore by symmetry, we may assume $j_{1}\succeq\max(j_{2},m+k)$. We then use (\ref{basiccon}), estimating $\widetilde{F}$ in $L^{2}$, to obtain an estimate (note also that $j\preceq m+k_{1}$)
\begin{equation}\kappa\preceq(m+k_{1})+\frac{k}{2}+m-j_{1}-\frac{k_{1}}{2}-\frac{3m-j_{2}+k_{1}}{2}\preceq 0.\end{equation} Moreover, we have $\kappa\leq -m/10$ unless $j_{2}\geq m/10$, which proves (\ref{lowgoal}).

(2) Next, assume $3k_{1}\geq k+D_{0}^{2}$, then we will have $|\nabla_{\eta}\Phi(\xi,\eta)|\sim 2^{3k_{1}}$ in the region of integration. Moreover, We have $|\nabla_{\eta}^{\nu}\Phi(\xi,\eta)|\lesssim 2^{(4-|\nu|)k_{1}}$ in the region of interest for $2\leq |\nu|\leq 4$. Now if $\max(j_{1},j_{2})\prec m+3k_{1}$, we will use Proposition \ref{ips0} and take \[K=2^{m}, \quad n=3,\quad\epsilon\sim2^{3k_{1}}, \quad\lambda\sim2^{\max(j_{1},j_{2})}\] and deduce that\[\|I_{jk}\|_{L^{2}}\lesssim\epsilon_{1}^{2} \exp(-\gamma(\min(2^{m+3k_{1}-j_{1}},2^{m+3k_{1}-j_{2}}))^{\gamma})\varepsilon_{1}^{2}\] which will be sufficient. Therefore by symmetry, we may assume $j_{1}\succeq\max(j_{2},m+3k_{1})$. We then use (\ref{basiccon}), estimating $\widetilde{F}$ in $L^{2}$, to obtain (note also that $j\preceq m+k_{1}$)
\begin{equation}\kappa\preceq(m+k_{1})+\frac{3k_{1}}{2}+m-j_{1}-\frac{k_{1}}{2}-\frac{3m-j_{2}+k_{1}}{2}\preceq 0,\end{equation} and also notice that either $j_{2}\geq o\cdot m$ or $\kappa\leq-o\cdot m$, which again proves (\ref{lowgoal}).

(3) Now assume $|3k_{1}-k|\leq D_{0}^{2}$. We may also assume that $k_{1}\succ -m/4$, since otherwise we may assume $j_{1}\geq j_{2}\succeq -k_{1}$, and directly use (\ref{basiccon}), estimating $\widetilde{F}$ in $L^{2}$, to bound\[\kappa\preceq m+k_{1}+\frac{3k_{1}}{2}+m-j_{1}-\frac{k_{1}}{2}-\frac{3m-j_{2}+k_{1}}{2}\preceq \frac{m}{2}+2k_{1}\preceq0,\] which implies (\ref{lowgoal}) since $\min(j_{1},j_{2})\geq |k_{1}|\geq m/5$. Now we will have 
\begin{equation}\label{taylor1}\nabla_{\eta}\Phi(\xi,\eta)=\rho_{1}\xi+4\rho_{2}|\eta|^{2}\eta+O(2^{5k_{1}}),\end{equation} as well as
\begin{equation}\label{taylor2}\Phi(\xi,\eta)=\rho_{1}(\eta\cdot\xi)+\rho_{2}|\eta|^{4}+O(2^{6k_{1}}).\end{equation} Choose a cutoff $\psi_{1}(\tau)$ supported in $|\tau|\ll 1$ such that $1-\psi_{1}=\psi_{2}$ is supported in $|\tau|\gtrsim 1$ (the implicit constants here may depend on $\mathcal{B}$).

We next consider the contribution where the factor $\psi_{2}(2^{-3k_{1}}(\rho_{1}\xi+4\rho_{2}|\eta|^{2}\eta))$ is attached to the weight ${m}(\xi,\eta)$. In this situation we will have $|\nabla_{\eta}\Phi|\sim 2^{3k_{1}}$ and $|\nabla_{\eta}^{\nu}\Phi(\xi,\eta)|\lesssim 2^{(4-|\nu|)k_{1}}$ for $2\leq |\nu|\leq 4$, thus when $\max(j_{1},j_{2})\prec m+3k_{1}$, we will be able to use Proposition \ref{ips0} with \[K=2^{m},\quad n=3,\quad \epsilon\sim2^{3k_{1}},\quad\lambda\sim2^{\max(j_{1},j_{2})}\] to obtain sufficient decay. Note that a new difficulty arises with the introduction of the $\psi_{2}$ factor, but one have\[\big|\partial_{\eta}^{\nu}\psi_{2}(2^{-3k_{1}}(\rho_{1}\xi+4\rho_{2}|\eta|^{2}\eta))\big|\lesssim (C|\nu|)!2^{-k_{1}|\nu|}\] which can be proved by rescaling. Therefore, we may assume that $\max(j_{1},j_{2})\succeq m+3k_{1}$, which allows us to repeat the proof in part (2) above. Here one should note that (\ref{basiccon}) still holds for the modified bilinear operator, because the function\[\chi(2^{-3k_{1}}\xi)\chi(2^{-k_{1}}\eta)\psi_{2}(2^{-3k_{1}}(\rho_{1}\xi+4\rho_{2}|\eta|^{2}\eta))\] has its inverse Fourier transform bounded in $L^{1}$, which is again easily seen by rescaling, so that we can still use Proposition \ref{bilinear}.

Now suppose that $\psi_{1}(2^{-3k_{1}}(\rho_{1}\xi+4\rho_{2}|\eta|^{2}\eta))$ is attached to ${m}(\xi,\eta)$. In this contribution we always have $|\Phi|\sim 2^{4k_{1}}$. Moreover, if we consider the function\begin{equation}\widetilde{{m}}(\xi,\eta)=\frac{2^{4k_{1}}{m}(\xi,\eta)}{\Phi(\xi,\eta)}\psi_{1}(2^{-3k_{1}}(\rho_{1}\xi+4\rho_{2}|\eta|^{2}\eta))\chi(2^{-3k_{1}}\xi)\chi(2^{-k_{1}}\eta),\end{equation} then it will have inverse Fourier transform bounded in $L^{1}$ (simply by rescaling and using the expansion (\ref{taylor2})), thus we will be able to integrate by parts in $s$ to produce the $I_{0}$, $I_{1}$ and $I_{2}$ terms but with additional cutoff factors, then use the variant of (\ref{basiccon}) with multiplier $\widetilde{m}$ and Proposition \ref{linearbound0} to obtain (by symmetry) $\kappa\preceq\max(\kappa_{1},\kappa_{2})$, where\[2^{\kappa_1}=\varepsilon_1^{-2}\|I_{0}\|_{L^{2}},\quad\kappa_{1}\preceq m+k_{1}+\frac{3k_{1}}{2}-4k_{1}-\max(j_{1},j_{2})-\frac{k_{1}}{2}-\frac{3m-\min(j_{1},j_{2})+k_{1}}{2}\preceq 0\] and\[2^{\kappa_2}=\varepsilon_1^{-2}\|I_{1}\|_{L^{2}},\quad\kappa_{2}\preceq m+k_{1}+\frac{3k_{1}}{2}+m-4k_{1}-\frac{3m}{2}-\frac{k_{1}}{2}-m=-\frac{m}{2}-2k_{1}\preceq 0.\] Moreover, we  have either $\kappa\leq -m/10$ or $\min(j_{1},j_{2})\geq|k_{1}|\geq m/10$. Therefore we have finished the proof in the case when $c_{\mu}^{2}/b_{\mu}+c_{\nu}^{2}/b_{\nu}=0$.

\subsection{The final case}\label{lowfinal} Assume here that $c_{\mu}^{2}/b_{\mu}+c_{\nu}^{2}/b_{\nu}\neq0$. In this case we have
\begin{equation}\Phi(\xi,\eta)=\rho_0|\xi|^{2}+\rho_1\xi\cdot\eta+\rho_3|\eta|^{2}+O(|\xi|^{4}+|\eta|^{4})\end{equation} with $\rho_1\rho_3\neq 0$. If $k_{1}\geq k+D_{0}^{2}$, then we will have $|\nabla_{\eta}\Phi|\sim 2^{k_{1}}$, so we can argue exactly as in Section \ref{secondredkg} above; the situation will be the same if $k_{1}\leq k+D_{0}^{2}$, and if we insert a cutoff to restrict to the region $|\rho_1\xi+2\rho_3\eta|\gtrsim 2^{k}$. Note that in the latter case, the introduction of this cutoff will not affect the use of Proposition \ref{bilinear} as in the derivation of (\ref{basiccon}), as will be shown below.

Next, suppose $k_{1}-k\leq D_{0}^{2}$, and we attach some cutoff supported in the region where $|\rho_1\xi+2\rho_3\eta|\ll 2^{k}$, and $\rho_1^{2}-4\rho_0\rho_3\neq 0$. Then we will have $|\Phi|\sim 2^{2k}$, so we can integrate by parts in $s$ and argue as in Section \ref{secondredkg} above. Here one should note that (\ref{basiccon}) still holds, because the function\[\widetilde{{m}}(\xi,\eta)=\frac{2^{2k}{m}(\xi,\eta)}{\Phi(\xi,\eta)}\psi_{1}(2^{-k}(\rho_1\xi+2\rho_3\eta))\psi_{2}(2^{-k}\xi)\psi_{3}(2^{-k}\eta)\] will have its inverse Fourier transform bounded in $L^{1}$ due to scaling.

In the only remaining scenario we have $\rho_1^{2}-4\rho_0\rho_3=0$, so for some constants $\rho_4$ and $\rho_5$ we have $\Phi(\xi,\eta)=\rho_{4}|\eta-\rho_5\xi|^{2}+O(2^{4k})$. We have also assumed that $k_{1}-k\leq D_{0}^{2}$ and $|\eta-\rho_5\xi|\ll 2^{k}$. (again these constants may depend on $\mathcal{B}$).

(1) Suppose $|k|\succeq m/6$. If $|k|\succeq m/2$, then we can directly use (\ref{basiccon}), estimating $\widetilde{F}$ in $L^{2}$, to bound\[\kappa\preceq (m+k)+\frac{k}{2}+m-j_{1}-\frac{k}{2}-j_{2}+k\preceq 2m+4k\preceq 0,\] which implies (\ref{lowgoal}) because $\min(j_{1},j_{2})\geq |k|\geq m/3$. Thus we may assume $-m/2\prec k\preceq -m/6$. By inserting suitable cutoff functions to the weight ${m}(\xi,\eta)$, we may consider the integral in the regions where $|\eta-\rho_5\xi|\sim 2^{-r}$ for $|k|\leq r\prec m/2$, and where $|\eta-\rho_5\xi|\lesssim 2^{-m/2}\langle m\rangle^{A}$.

In the first situation we have $|\nabla_{\eta}\Phi|\sim|\nabla_{\xi}\Phi|\sim 2^{-r}$. Suppose $\max(j_{1},j_{2})\prec m-r$, we will use Proposition \ref{ips0}, setting \[K=2^{m},\quad n=1,\quad \epsilon\sim2^{-r},\quad\lambda\sim2^{\max(j_{1},j_{2},r)},\] to bound (say) $\kappa\leq -10m$ (here the restriction $\lambda\geq 2^{r}$ is due to the newly introduced cutoff factor $\psi(2^{r}(\eta-\rho_5\xi))$). Moreover, if  $j\succ m-r$, we may also assume $j\succ j_{1}$, so we can integrate by parts in $\xi$ with fixed $\eta$ exactly as in the estimate of (\ref{jbig}), where we choose, in Proposition \ref{ips0}, \[K=2^{m},\quad n=1,\quad\epsilon\sim2^{j-m},\quad\lambda\sim2^{\max(r,j_{1})},\] to obtain sufficient decay (note that $j\succ m-r\succ \max(r,m/2)$).

Now we have $j\preceq m-r\preceq \max(j_{1},j_{2})$. Suppose $j_{1}\geq j_{2}$, we then use (\ref{basiccon}), estimating $\widetilde{F}$ in $L^{2}$, to bound\[\kappa\preceq (m-r)+\frac{k}{2}+m-(m-r)-\frac{k}{2}-m=0,\] and recover the logarithmic loss because $\min(j_{1},j_{2})\geq|k|\geq m/7$. Here we still need to verify the validity of (\ref{basiccon}) with weight\[\widetilde{{m}}(\xi,\eta)={m}(\xi,\eta)\psi_{1}(2^{-k}\xi)\psi_{2}(2^{-k}\eta)\psi_{3}(2^{l}(\eta-\rho_5\xi)).\] Now by the algebra property of the norm $\|\mathcal{F}^{-1}M\|_{L^{1}}$, we only need to bound this norm for the function\[\mathfrak{n}(\xi,\eta)=\psi_{1}(2^{-k}\xi)\psi_{3}(2^{l}(\eta-\rho_5\xi)),\] which is easily estimated by a linear transformation and rescaling.

In the second situation above, we must have $j\preceq m/2$, otherwise we could bound the integral in the same way as in the estimate of (\ref{jbig}). Therefore we may assume $j_{1}\geq j_{2}$, and use (\ref{basiccon}), estimating $\widetilde{F}$ in $L^{2}$, to bound \[\kappa\preceq \frac{m}{2}+\frac{k}{2}+m-j_{1}-\frac{k}{2}-\frac{3m}{2}+\frac{j_{2}-k}{2}\preceq 0.\] Since $\min(j_{1},j_{2})\succeq m/6$, this proves (\ref{lowgoal}).

(2) Suppose $|k|\prec m/6$. By inserting suitable cutoff functions, we will consider the integral in regions where $|\eta-\rho_5\xi|\sim 2^{-r}$, for $|k|\leq r\prec 3|k|$, and where $|\eta-\rho_5\xi|\lesssim 2^{3k}\langle m\rangle^{A}$. First suppose $r\prec 3|k|$, then we will have $|\nabla_{\eta}\Phi|\sim 2^{-r}$ as well as $|\nabla_{\xi}\Phi|\sim 2^{-l}$, and note that $r\prec m/2$. This will allow us to assume that $\max(j_{1},j_{2})\succeq m-r\succeq j$, since if $\max(j_{1},j_{2})\prec m-r$, we will be able to use Proposition \ref{ips0} to integrate by parts in $\eta$, setting \[K=2^{m},\quad k=1,\quad\epsilon\sim2^{-l},\quad\lambda\sim2^{\max(j_{1},j_{2},r)}\] to obtain sufficient decay, and if $j\succ m-r$ we will be able to integrate by parts in $\xi$ just as above, having assumed that $j\succ\min(j_{1},j_{2})$.

Now that we have $\max(j_{1},j_{2})\succeq m-r\succeq j$, we may assume $j_{1}\geq j_{2}$ and use (\ref{basiccon}), estimating $\widetilde{F}$ in $L^{2}$ to bound\[\kappa\preceq j+\frac{k}{2}+m-j_{1}-\frac{k}{2}-m\preceq 0.\] Since either $\kappa\leq -m/10$ or $j_{1}\geq j_{2}\geq m/20$, this proves (\ref{lowgoal}).

Finally, let us assume $|k|\prec m/6$ and we are in the region $|\eta-\rho_5\xi|\lesssim 2^{3k}\langle m\rangle^{A}$. Since now $|\nabla_{\eta}\Phi|\lesssim 2^{3k}\langle m\rangle^{A}$, we must have $j\preceq m+3k$ (otherwise we integrate by parts in $\xi$ as before; note also that $3k\prec m/2$). In this situation we need much more careful analysis of the phase function, which we will carry out now.

recall that \begin{equation}\Phi(\xi,\eta)=\sum_{\alpha=1}^{2}\sum_{\beta=0}^{2}\sigma_{\alpha,\beta}|\zeta_{\beta}|^{2\alpha}+O(2^{6k}),\end{equation} where $\zeta_{0}=\xi$, $\zeta_{1}=\xi-\eta$ and $\zeta_{2}=\eta$, and the coefficients are
\begin{equation}\sigma_{1,0}=\frac{c_{\sigma}^{2}}{2b_{\sigma}},\quad \sigma_{1,1}=-\frac{c_{\mu}^{2}}{2b_{\mu}},\quad\sigma_{1,2}=-\frac{c_{\nu}^{2}}{2b_{\nu}};
\end{equation} \begin{equation}\sigma_{2,0}=\frac{c_{\sigma}^{4}}{8b_{\sigma}^{3}},\quad\sigma_{2,1}=-\frac{c_{\mu}^{4}}{8b_{\mu}^{3}},\quad\sigma_{2,2}=-\frac{c_{\nu}^{4}}{8b_{\nu}^{3}}.\end{equation} Now, in order for the quadratic term to be a perfect square (which we have assumed before), we must have $\sigma_{1,0}\sigma_{1,1}+\sigma_{1,1}\sigma_{1,2}+\sigma_{1,2}\sigma_{1,0}=0$, or equivalently
\begin{equation}\label{second}\frac{b_{\sigma}}{c_{\sigma}^{2}}-\frac{b_{\mu}}{c_{\mu}^{2}}-\frac{b_{\nu}}{c_{\nu}^{2}}=0.\end{equation} Recall also that 
\begin{equation}\label{first}b_{\sigma}-b_{\mu}-b_{\nu}=0.\end{equation} Now we compute the fourth-order term; note that apart from an error of at most $O(2^{6k})$, we may assume $\eta=\rho_{5}\xi$, and that $\rho_{5}=-\sigma_{1,0}/\sigma_{1,2}$. We again have two situations.

(A) Suppose the $c$'s are not all equal, then in particular no two of the $c$'s can be equal, and we can compute $\rho_{7}=(c_{\sigma}^{2}-c_{\mu}^{2})/(c_{\nu}^{2}-c_{\mu}^{2})$. Therefore we may assume \[\xi=(c_{\nu}^{2}-c_{\mu}^{2})v,\quad\eta=(c_{\sigma}^{2}-c_{\mu}^{2})v;\quad\xi-\eta=(c_{\nu}^{2}-c_{\sigma}^{2})v,\] where $|v|\sim 2^{k}$. Up to a constant we may also assume \[b_{\sigma}=\frac{1}{c_{\nu}^{2}}-\frac{1}{c_{\mu}^{2}}, \quad b_{\mu}=\frac{1}{c_{\nu}^{2}}-\frac{1}{c_{\sigma}^{2}},\quad b_{\nu}=\frac{1}{c_{\sigma}^{2}}-\frac{1}{c_{\mu}^{2}}.\] Therefore we may compute \[\sum_{\beta=0}^{2}\sigma_{2,\beta}|\zeta_{\beta}|^{4}=\lambda|v|^{4},\] where up to a constant
\begin{equation}\lambda=(c_{\sigma}c_{\mu}c_{\nu})^{6}\sum_{\mathrm{cyclic}}\frac{c_{\mu}^{2}-c_{\nu}^{3}}{c_{\sigma}^{2}}=-(c_{\sigma}c_{\mu}c_{\nu})^{4}(c_{\sigma}^{2}-c_{\mu}^{2})(c_{\mu}^{2}-c_{\nu}^{2})(c_{\nu}^{2}-c_{\sigma}^{2})\neq 0.\end{equation} This means that $\Phi(\xi,\eta)=\rho_{6}|\xi|^{4}+O(2^{6k})$, and also that the function\[\widetilde{{m}}(\xi,\eta)=\frac{2^{4k}{m}(\xi,\eta)}{\Phi(\xi,\eta)}\psi(2^{-k}\xi)\psi_{1}(2^{-3k}(\eta-\rho_{7}\xi)),\] where $\psi$ is supported in $|\xi|\sim 1$ and $\psi_{1}$ supported in $|\xi|\lesssim 1$, has inverse Fourier transform bounded in $L^{1}$ (with bounds depending on $\mathcal{B}$), by a linear transformation and rescaling argument. Therefore, we can integrate by parts in $s$, producing $I_{0}$, $I_{1}$ and $I_{2}$ terms, and use the variant of (\ref{basiccon}) with multiplier $\widetilde{m}$ to bound $\kappa\preceq\max(\kappa_{1},\kappa_{2})$, where\begin{equation}
\begin{aligned}2^{\kappa_1}&=\varepsilon_1^{-2}\|I_0\|_{L^2},&\kappa_{1}&\preceq (m+3k)+\frac{k}{2}-4k-j_{1}-\frac{k}{2}-3m/2+\frac{j_{2}-k}{2}\preceq-\frac{m}{2}-k\leq -\frac{m}{4},\\2^{\kappa_2}&=\varepsilon_1^{-2}\|I_1\|_{L^2},&\kappa_{2}&\preceq (m+3k)+\frac{k}{2}+m-4k-3m/2-\frac{k}{2}-m=-\frac{m}{2}-k\leq -\frac{m}{4},\end{aligned}\end{equation} which proves (\ref{lowgoal}).

(B) Suppose that $c_{\sigma}=c_{\mu}=c_{\nu}$, and $b_{\sigma}-b_{\mu}-b_{\nu}=0$. By extracting a constant we may assume $c_{\sigma}=1$; by symmetry we may assume $b_{\mu},b_{\nu}>0$, so \[\Phi(\xi,\eta)=\sqrt{|\xi|^{2}+b_{\sigma}^{2}}-\sqrt{|\xi-\eta|^{2}+b_{\mu}^{2}}-\sqrt{|\eta|^{2}+b_{\nu}^{2}},\] which implies that $\rho_{5}=b_{\nu}/b_{\sigma}$ and that $\Phi(\xi,\eta)=\partial_{\eta}\Phi(\xi,\eta)=0$ at the point $\eta=\rho_5\xi$ for all $\xi$. Now, if we expand $\Phi$ as a power series of $\xi$ and $\eta-\rho\xi$, the constant term will be $0$ and the second order term will be $\rho_{7}|\eta-\rho_5\xi|^{2}$; moreover, each term of degree four or higher must contain at least \emph{two} factors of $\eta-\rho_5\xi$. Therefore, if we restrict to regions where $|\eta-\rho_5\xi|\sim 2^{-r}$ for $r\prec m/2$, we will have $|\nabla_{\eta}\Phi|\sim |\nabla_{\xi}\Phi|\sim 2^{-r}$, regardless of the relationship between $r$ and $k$. In this situation we can thus repeat the previous arguments and assume $j\preceq m-r\preceq\max(j_{1},j_{2})$ and close the estimate using (\ref{basiccon}) as before.

Now we may restrict to the region where $|\eta-\rho_7\xi|\lesssim 2^{-m/2}\langle m\rangle^{A}$, so we have $j\preceq m/2$. Under this assumption we may then assume $j_{1}\geq j_{2}$ and deduce\[\kappa\preceq \frac{m}{2}+\frac{k}{2}+m-j_{1}-\frac{k}{2}-\frac{3m}{2}+\frac{j_{2}}{2}-\frac{k}{2}\preceq0.\] Moreover, we have either $\kappa\leq-o\cdot m$ or $\min(j_{1},j_{2})\geq o\cdot m$, unless  $j_{1},j_{2},k=o\cdot m$ and $j=(1/2+o)m$. In this final scenario, we will \emph{not} insert any cutoff to the integral as above; instead we will use a special argument as follows.

Note that\begin{equation}\widetilde{I}_{jk}(x)=\int_{a}^{b}\int_{(\mathbb{R}^{3})^{2}}e^{{i}(s\Phi(\xi,\eta)+x\cdot\xi)}\varphi_{j}(x)\varphi_{k}(\xi){m}(\xi,\eta)\widehat{F}(s,\xi-\eta)\widehat{G}(s,\eta)\,\mathrm{d}\xi\mathrm{d}\eta\,\mathrm{d}s.\nonumber\end{equation} Since $j_{1}, j_{2}$ and $k$ are all $o(m)$ in absolute value, the function $\chi(2^{-k}\xi){m}(\xi,\eta)\widehat{F}(\xi-\eta)\widehat{G}(\eta)$ will be bounded in a suitable Schwartz norm by $2^{o\cdot m}$. Now we make a change of variables and let $\eta-\rho\xi=\zeta$, then we will be estimating, for fixed $s$, an integral of form\begin{equation}\label{double}\int_{(\mathbb{R}^{3})^{2}}e^{{i}(s\Psi(\xi,\zeta)+x\cdot\xi)}\phi(\xi,\zeta)\,\mathrm{d}\xi\mathrm{d}\zeta,\end{equation} where $\phi$ is a function with Schwartz norm bounded by $2^{o\cdot m}$ and\[\Psi(\xi,\zeta)=\rho_{8}|\zeta|^{2}+\sum_{q,r=1}^{3}\zeta^{q}\zeta^{r}h_{qr}(\xi,\zeta),\] with $h_{q,r}(0,0)=0$, so that we have $|\nabla_{\zeta}\Psi|\sim |\zeta|$ where $|\xi|+|\zeta|$ is small. Now, if we cutoff in the region $|\zeta|\gtrsim 2^{-2m/5}$, we will be able to integrate by parts in $\zeta$ to obtain sufficient decay; if we cutoff in the region $|\zeta|\ll 2^{-2m/5}$, we will then fix $\zeta$ and integrate by parts in $\xi$, noting that $|s\zeta^{q}\zeta^{r}|\lesssim 2^{m/5}$ and $|x|\sim 2^{j}\gtrsim 2^{2m/5}$, to obtain sufficient decay for this integral. This completes the proof in the final scenario and thus finishes the proof of (\ref{lowgoal}).

\section{Medium frequencies}\label{med} In this section we consider the case when $\max(k,k_{1},k_{2})<K_{0}^{2}$, and $\max(k_{1},k_{2})>-K_{0}$. By (\ref{normcomp}), we may also assume $\max(k_{1},k_{2})<K_{0}$.

First we shall prove the crucial bilinear lemma introduced in Section \ref{bililemma}. This allows us to restrict the angular variable under very mild restrictions, and will be used frequently in the proof below.
\begin{lemma}\label{angular}Restrict the integral (\ref{bilibili3}) by adding two cutoff functions and forming\begin{equation}\label{bilibili4}I':=\int_{\mathbb{R}^{3}}e^{{i}t\Phi(\xi,\eta)}\psi_{1}(2^{\kappa}\Phi(\xi,\eta))\varphi_{k}(\xi){m}(\xi,\eta)\psi_{2}(2^{\upsilon}\sin\angle(\xi,\eta))\widehat{F}(s,\xi-\eta)\widehat{G}(s,\eta)\,\mathrm{d}\eta,\end{equation} where $|t|\sim 2^{m}$, $\psi_{1}$ is Schwartz, $\psi_{2}(z)$ is supported in $|z|\sim 1$. Assume that\[\max(k_{1},k_{2})\geq -2K_{0}^{2},\quad\kappa\prec m,\quad\max(0,k_{1},k_{2}):=\overline{k},\quad\max(l_{1},l_{2}):=\overline{l}<m/10,\quad j_{1}-k_{1}\prec m,\] then we have $|I'|\lesssim 2^{-20m}\varepsilon_{1}^{2}$, in each of the following three cases:
\begin{enumerate}
\item When $|k_{1}-k_{2}|\leq 6$, $\upsilon\prec (m+k)/2-\overline{k}$, and $m+k\succ2\overline{l}$.
\item When $|k-k_{1}|\leq 6$, $\upsilon\prec (m+k_{2})/2-\overline{k}$, and $m+k_{2}\succ 2\overline{l}$.
\item When $|k-k_{2}|\leq 6$, $|k_{1}|\prec (m-\overline{l})/2$ and $\upsilon\prec (m-\overline{l})/2$.
\end{enumerate}
\end{lemma}
\begin{proof} Let $\max(k_{1},k_{2}):=k_{3}$. By symmetry, we may assume $\xi=|\xi|e_{1}$ where $|\xi|\sim 2^{k}$, and $e_{i}$ are coordinate vectors; we then make several reductions. Fix a time $s$, let \[\rho=|\eta|,\quad\theta=\angle(e_{1},\eta),\quad\phi=\angle(\eta-(\eta\cdot e_{1})e_{1},e_{2});\]
\[\rho'=|\xi-\eta|,\quad\theta'=\angle(e_{1},\xi-\eta),\quad\phi'=\angle(\xi-\eta-((\xi-\eta)\cdot e_{1})e_{1},e_{2})\] be the spherical coordinates, we expand as in (\ref{decomp})\[F(s,\xi-\eta)=\sum_{q\lesssim 2^{l_{1}}}\sum_{m=-q}^{1}f_{q}^{m}(\rho')Y_{q}^{m}(\theta',\phi');\qquad G(s,\eta)=\sum_{q'\lesssim 2^{l_{2}}}\sum_{m'=-q'}^{q'}g_{q'}^{m'}(\rho)Y_{q'}^{m'}(\theta,\phi),\] and fix a choice $(\rho,q,q',m,m')$. Using the Fourier transform of $\psi_{1}$ we can write\begin{equation}\label{modifytime}e^{it\Phi(\xi,\eta)}\psi_{1}(2^{\kappa}\Phi(\xi,\eta))=\int_{\mathbb{R}}2^{-\kappa}\widehat{\psi_{1}}(2^{-\kappa}r)e^{i(t+r)\Phi(\xi,\eta)}\,\mathrm{d}r\end{equation} and restrict $|r|\lesssim 2^{(m+\kappa)/2}$, then fix one $r$; similarly decomposing\[f_{q}^{m}(\rho')=\int_{\mathbb{R}}e^{i\sigma\rho'}\widehat{f_{q}^{m}}(\sigma)\,\mathrm{d}\sigma,\] we may assume $|\sigma|\lesssim 2^{(m+j_{1})/2}$ (otherwise $\widehat{f_{q}^{m}}(\sigma)$ decays rapidly), and then fix $\sigma$. After absorbing $m(\xi,\eta)$ into cutoff functions, we reduce to a $\theta$ integral\begin{equation}\label{angular0}I''=\int_{\mathbb{S}^{2}}e^{i(t+r)\Lambda(\rho')+\sigma\rho'}\psi_{2}(2^{\upsilon}\sin\theta)Y_{q}^{m}(\theta',\phi')Y_{q'}^{m'}(\theta,\phi)\,\mathrm{d}\theta\mathrm{d}\phi,\end{equation} where $\Lambda=\Lambda_{\nu}$ for some $\nu\in\mathcal{P}$, $\rho'$, $\theta'$ and $\phi'$ are functions of $\theta$ and $\phi$ with fixed $\rho$. Let $H=(t+r)\Lambda(\rho')+\sigma\rho'$, note that $(\rho',\theta',\phi')$ is smooth in $\theta$, and $\partial_{\theta}((\rho')^{2})=-2|\xi|\rho\sin\theta$ by the law of cosines. Using the formula \begin{equation}\label{spec}\frac{\mathrm{d}^{n}}{\mathrm{d}x^{n}}f(g(x))=\sum_{p=0}^{n}\sum_{\alpha_{1}+\cdots+\alpha_{p}=n-p}A(n,p;\alpha_{1},\cdots,\alpha_{p})\cdot f^{(p)}(g(x))\prod_{i=1}^{p}g^{(\alpha_{i}+1)}(x),\end{equation} which can be proved by induction, we deduce that\begin{equation}\label{ipsy}
\begin{aligned}|\partial_{\theta}\rho'|&\sim 2^{k-\upsilon},&|\partial_{\theta}H|&\sim 2^{m+k-\upsilon},&|\partial_{\theta}^{\alpha}(\theta',\phi')|&\lesssim (C\alpha)!2^{\alpha k_{3}};\\
|\partial_{\theta}^{\alpha}H|&\lesssim (C\alpha)!2^{\alpha k_{3}}2^{m+k},&|\partial_{\theta}^{\alpha}\rho'|&\lesssim (C\alpha)!2^{\alpha k_{3}}2^{k}
\end{aligned}
\end{equation} in case (1), and that
\begin{equation}\label{ipsz}
\begin{aligned}|\partial_{\theta}\rho'|&\sim 2^{k_{2}-\upsilon},&|\partial_{\theta}H|&\sim 2^{m+k_{2}-\upsilon},&|\partial_{\theta}^{\alpha}(\theta',\phi')|&\lesssim (C\alpha)!2^{\alpha k_{3}};\\
|\partial_{\theta}^{\alpha}H|&\lesssim (C\alpha)!2^{\alpha k_{3}}2^{m+k_{2}},&|\partial_{\theta}^{\alpha}\rho'|&\lesssim (C\alpha)!2^{\alpha k_{3}}2^{k_{2}}
\end{aligned}
\end{equation}
in case (2), and that
\begin{equation}\label{ipsv}
\begin{aligned}|\partial_{\theta}\rho'|&\gtrsim 2^{2k_{3}-k_{1}-\upsilon},&|\partial_{\theta}H|&\gtrsim 2^{m+2k_{3}-\upsilon},&&|\partial_{\theta}^{\alpha}(\theta',\phi')|\lesssim (C\alpha)!2^{\alpha(2k_{3}+n)};\\
|\partial_{\theta}^{\alpha}\rho'|&\lesssim (C\alpha)!2^{k_{1}}2^{\alpha(2k_{3}+n)},&|\partial_{\theta}^{\alpha}H|&\lesssim (C\alpha)!2^{m+2k_{1}}2^{\alpha(2k_{3}+n)},&&n=\max(-k_{1},-2k_{1}-\upsilon)
\end{aligned}
\end{equation} in case (3). We then integrate by parts in $\theta$, using Proposition \ref{ips0} and Remark \ref{remark0}, setting\[K=2^{m+k},\quad n=1,\quad \epsilon\sim 2^{-\upsilon},\quad\lambda\sim 2^{\max(l_{1}+k_{3},l_{2},\upsilon)},\quad\lambda'=2^{k_{3}}\] in case (1), \[K=2^{m+k_{2}},\quad n=1,\quad \epsilon\sim 2^{-\upsilon},\quad\lambda\sim 2^{\max(l_{1}+k_{3},l_{2},\upsilon)},\quad\lambda'=2^{k_{3}}\] in case (2), and \[K=2^{m+2k_{1}},\quad n=1,\quad\epsilon\sim2^{2k_{3}-2k_{1}-\upsilon},\quad\lambda\sim2^{\max(\upsilon,l_{2},l_{1}+2k_{3}+n)},\quad\lambda'\sim2^{2k_{3}+n}\] in case (3), so that we can bound $|J'|\lesssim \exp(-\gamma\langle m\rangle^{A\gamma})$ with some fixed constant $\gamma$ and conclude the proof, provided $A$ is chosen large enough. 
\end{proof}\subsection{Mixed inputs}\label{mix} Suppose $\min(k_{1},k_{2})\leq -K_{0}^{2}$. By symmetry, we may assume $-K_{0}<k_{1}<K_{0}^{2}$ and $k_{2}\leq -K_{0}^{2}$. If $k\leq -K_{0}$, using Proposition \ref{propp} we have $|\Phi(\xi,\eta)|\gtrsim 1$, so we can integrate by parts in $s$, then use Proposition \ref{linearbound0} to conclude, in the same way as estimating (\ref{ibps00}); thus we will now assume $k\geq -K_{0}$, so we need to prove \begin{equation}\label{mixgoal}\|I_{jk}\|_{L^{2}}\lesssim \langle m\rangle^{-100}\langle j\rangle^{N_{0}}2^{-5j/6}\varepsilon_{1}^{2},\qquad\|\widehat{I_{jk}}\|_{L^{1}}\lesssim 2^{-j}\langle j\rangle^{-N_{0}}\langle m\rangle^{-100}\varepsilon_{1}^{2}.\end{equation}

First note that, if $j\leq (1+o)\max(m,\min(j_1,j_2))$ and we use a cutoff to restrict $|\Phi(\xi,\eta)|\gtrsim 2^{-m/9}$, we can then integrate by parts in $s$ to get $I_{0}$, $I_{1}$ and $I_{2}$ terms as in (\ref{ibps00}). Consider for example \[\widehat{I_{1}}(\xi)=\int_{a}^{b}\int_{\mathbb{R}^{3}}\frac{e^{is\Phi(\xi,\eta)}}{\Phi(\xi,\eta)}\psi(2^{m/9}\Phi(\xi,\eta)){m}(\xi,\eta)\widehat{\partial_{s}F}(s,\xi-\eta)\widehat{G}(s,\eta)\,\mathrm{d}\eta\mathrm{d}s,\] where $\psi$ is some cutoff; recall as in the proof of Lemma \ref{angular} that\begin{equation}\label{trick}\frac{\psi(2^{m/9}\Phi)}{\Phi}=\int_{\mathbb{R}}e^{ir\Phi}\chi(2^{-m/9}r)\,\mathrm{d}r\end{equation} for some $\chi\in\mathcal{G}_{6}$, and that we can restrict $|r|\lesssim \langle m\rangle^{A}2^{m/9}$, so we can include the cutoff $\psi(2^{m/9}\Phi)$ as part of the phase.  We then use (\ref{basiccon}) to bound\begin{multline*}\|I_{1}\|_{L^{2}}\lesssim  2^{m/9}\sup_{|r|\lesssim 2^{m/8}}\bigg\|\int_{a}^{b}\int_{\mathbb{R}^{3}}e^{i(s+r)\Phi}{m}(\xi,\eta)\widehat{\partial_{s}F}(s,\xi-\eta)\widehat{G}(s,\eta)\,\mathrm{d}\eta\mathrm{d}s\bigg\|_{L^{2}}\\\lesssim2^{(10/9+o)m}2^{-9m/8}2^{-\max(m,j)}\varepsilon_{1}^{2}\lesssim 2^{-(1+o)j}\varepsilon_{1}^{2}.\end{multline*} In the same way, we can bound $I_{0}$ and $I_{2}$ using Proposition \ref{linearbound0}, so this term will be acceptable.

Next suppose $j_{2}\prec j$; by (\ref{basicj}) we may also assume $j\preceq m$ and $j_{2}\prec m$. Since $|\nabla_{\eta}\Phi(\xi,\eta)|\sim 1$, using Proposition \ref{ips0}, we must have $j_{1}\succeq m$. Using Proposition \ref{linearbound0} and  (\ref{basiccon}), we obtain\[\|I\|_{L^{2}}\lesssim 2^{m}\langle m\rangle^{A}\cdot 2^{-m-k_{1}/2}\cdot 2^{6\delta m}2^{-3m/2}\varepsilon_{1}^{2}\lesssim 2^{-8j/9}\varepsilon_{1}^{2},\] which proves the first half of (\ref{mixgoal}); as for the second half, note that the above estimate would actually give $\|I\|_{L^{2}}\lesssim 2^{-(1+o)m}\varepsilon_{1}^{2}$ if $|k_{1}|\leq (1-1/20)m$, and when $|k_{1}|\geq (1-1/20)m$, we would have\[\|\widehat{I}\|_{L^{\infty}}\lesssim \varepsilon_{1}2^{m}2^{3k_{1}/2}\|F\|_{L^{2}}\lesssim 2^{k_{1}}\langle m\rangle^{A}\varepsilon_{1}^{2} \] using (\ref{linearb1}) and (\ref{linearb2}). Note that from above we can also assume $|\Phi(\xi,\eta)|\lesssim 2^{-m/9}$ and hence $|\Phi(\xi,\xi)|\lesssim 2^{-m/9}$ (which restricts $\xi$ to some set of volume at most $2^{-m/18}$), so we could close by H\"{o}lder, using also Proposition \ref{propp}. 

Now suppose $j\preceq j_{2}$. We may also assume $j\preceq m$, since otherwise we must have $\min(j_{1},j_{2})\succeq j$, and the proof would be similar as below with trivial modifiactions. To prove the first part of (\ref{mixgoal}), we simply use Proposition \ref{linearbound0} and (\ref{basiccon}) to bound\[\|I\|_{L^{2}}\lesssim 2^{m}\cdot\langle m\rangle^{-N_{0}}2^{-m}\cdot \langle j_{2}\rangle^{N_{0}}2^{-5j_{2}/6}\varepsilon_{1}^{2}\lesssim\langle j\rangle^{N_{0}/2}2^{-5j/6}\varepsilon_{1}^{2};\]now we work on the Fourier $L^{1}$ bound. If $j_{1}-k_{1}\succeq m$, this would follow from estimating both $F$ and $G$ factors in Fourier $L^{1}$ norm, using (\ref{linearb2}), (\ref{linearb3}) and H\"{o}lder; so we will assume $j_{1}-k_{1}\prec m$ (so in particular $j_{1}\preceq m$ and hence $j\preceq m$).

Note that we may assume $|\Phi|\leq 2^{-m/9}$ as above; actually since $|k_{1}|\leq m/2$, using (\ref{linearb5}), the bound can be improved to $|\Phi|\leq 2^{-m/4}$ by the same argument. Next we treat the case when $|\Phi|\lesssim 2^{\rho_0}$ where $ \rho_0=k_{1}-m/15$. If $k_{1}+m/2\preceq\overline{k}$, this means $|k_{1}|\geq (1/2-6\delta)m$, which restricts $|\Phi(\xi,\xi)|\lesssim 2^{-2m/5}$. If we moreover assume $|\nabla_{\xi}\Phi|\sim 2^{-r}$, then this restricts $\xi$, by Proposition \ref{propp}, to a set of volume $\lesssim 2^{\min(-r,r-2m/5)}$. Moreover we may assume $j\preceq m-r$ if $r\prec m/2$ (or otherwise we integrate by parts in $\xi$, using the smallness of $\nabla_{\xi}\Phi$), thus we have\[\|\widehat{I_{jk}}\|_{L^{1}}\lesssim 2^{-(5/6-o)m}2^{\min(-r,r-2m/5)/2}\varepsilon_{1}^{2}\lesssim 2^{-(1+o)(m-r)}\varepsilon_{1}^{2}\lesssim 2^{-(1+o)j}\varepsilon_{1}^{2}.\]

If $k_{1}+m/2\succ \overline{k}$, using Proposition \ref{angular}, we can restrict that $|\sin\angle(\xi,\eta)|\lesssim 2^{-\rho}$, where $\rho=(1/2-\delta-o)m$. Now we write $\alpha=\pm|\xi|$, $\beta=\pm|\eta|$, and fix $s$ and the direction vectors $\widehat{\xi}=\theta$ and $\widehat{\eta}=\phi$, so that we reduce to an integral\[2^{m}2^{-2\rho}\int_{\mathbb{R}}\widehat{F}(\alpha\theta-\beta\phi)\widehat{G}(\beta\phi){m}(\alpha,\beta)\,\mathrm{d}\beta,\] where ${m}(\alpha,\beta)$ is bounded and supported in the region where $|\Phi^{+}(\alpha,\beta)|\lesssim 2^{\rho_{0}}$, and also $|\alpha-\beta|\lesssim 2^{k_{1}}$; note that we have omitted the exponential factor after taking absolute values. By Schur's test, we only need to bound the integrals\[\int_{\mathbb{R}}{m}(\alpha,\beta)\widehat{F}(\alpha\theta-\beta\phi)\,\mathrm{d}\alpha,\qquad \int_{\mathbb{R}}{m}(\alpha,\beta)\widehat{F}(\alpha\theta-\beta\phi)\,\mathrm{d}\beta.\] By (\ref{linearb1}) and the fact that $|\partial_{\beta}\Phi_{+}|\sim1$, we know that the first integral is trivially bounded by $\varepsilon_{1}$, and the second integral is bounded by $2^{-k_{1}+\rho_{0}}\varepsilon_{1}$.  This gives the bound\[\|\widehat{I_{jk}}\|_{L^{1}}\lesssim 2^{m}2^{24\delta m}2^{-2\rho}(2^{-k_{1}+\rho_{0}})^{1/2}2^{-j_{2}}\varepsilon_{1}^{2}\lesssim 2^{-(1+o)j}\varepsilon_{1}^{2}.\]

Now we are left with the case when $|\Phi|\gtrsim 2^{\rho_{0}}$ with $\rho_{0}$ as above. Since also $|\Phi|\lesssim 2^{-m/4}$, we must have $|k_{1}|\geq 2m/9$, so $\xi$ is already restricted to a region of volume $\lesssim 2^{-m/9}$; by H\"{o}lder, we only need to bound $\|I\|_{L^{2}}\lesssim 2^{-(17/18+o)m}\varepsilon_{1}^{2}$. Notice that $|k_{1}|\leq m/2$ (since $j_{1}-k_{1}\leq m$), we will integrate by parts in $s$. The boundary term $I_{0}$ is clearly acceptable by (\ref{basiccon}); the term $I_{1}$ where the derivative falls on $F$ is estimated in the same way as $I_{0}$, where one uses (\ref{disper1}) to deduce that\[\|e^{is\Lambda}\partial_{t}F\|_{L^{\infty}}\lesssim 2^{-3m/2}2^{3j_{1}/2}2^{-3m/2}2^{-k_{1}/2}\lesssim 2^{-5m/2}2^{j_{1}}\lesssim 2^{-3m/2}2^{k_{1}}\varepsilon_{1}^{2}\] and get\[\|I_{1}\|_{L^{2}}\lesssim 2^{m-\rho_{0}}2^{-5m/6}2^{-3m/2}2^{k_{1}}\varepsilon_{1}^{3}\lesssim 2^{-m}\varepsilon_{1}^{3}.\] Finally, using (\ref{linearb5}) and (\ref{basiccon}) as well as the trick used in (\ref{trick}), we can bound $\|I_{2}\|_{L^{2}}\lesssim 2^{\kappa}\varepsilon_{1}^{2}$ with\[\kappa\leq -\rho_{0}+m-m-3m/2\leq -(17/18+o)m.\]

\subsection{Medium frequency output}\label{medout} Here we assume $k>-K_{0}$, so we need to prove the bound\begin{equation}\label{medgoal}\|I_{jk}\|_{L^{2}}\lesssim 2^{-5j/6}\langle j\rangle^{A}\varepsilon_{1}^{2},\qquad\|\widehat{I_{jk}}\|_{L^{1}}\lesssim 2^{-(1+o)j}\varepsilon_{1}^{2}.\end{equation} The case when $j\succ m$ or $\min(j_{1},j_{2})\succeq m$ can be treated in the same way as in Section \ref{mix} above; if $|\Phi|\gtrsim 1$, we can integrate by parts in $s$. Neither case requires new arguments, so we will now focus on the main contribution when $j\preceq m$, $\min(j_{1},j_{2})\prec m$ and $|\Phi|\ll 1$. We next separate two different situations.
\subsubsection{The case of large $j$'s}\label{easy} Here we assume $j_{2}=\max(j_{1},j_{2})\succeq m/3$. Note that $j\preceq m$; we will insert a cutoff function to restrict $|\nabla_{\xi}\Phi(\xi,\eta)|\sim 2^{-r}$. If $r\succ \min(m-j,m/2)$, then either we can fix $\eta$ (or $\xi-\eta$) and integrate by parts in $\xi$ as in the arguments before, or we have $j\preceq\min(j_{1},j_{2})$ In the latter case we can use the same arguments in Section \ref{mix} above, using Proposition \ref{linearbound0} and (\ref{basiccon}) to close. Below we will assume $r\preceq\min(m-j,m/2)$.

Using Lemma \ref{angular}, we can assume $|\sin\angle(\xi,\eta)|\lesssim 2^{-(1/2-o)m}$ (note that when $|k|\lesssim 1$, the insertion of a cutoff of form $\psi(2^{r}\nabla_{\xi}\Phi)$ will not change this bound, as seen in the proof of Lemma \ref{angular}). Fix a time $s$, the direction vectors $\xi_{e}$ and $\eta_{e}$ of $\xi$ and $\eta$, let $\alpha=\pm|\xi|$ and $\beta=\pm|\eta|$, we reduce to an integral of form\[\int_{\mathbb{R}}e^{is\widetilde{\Phi}(\alpha,\beta)}{m}(\xi,\eta)\widehat{F}(\xi-\eta)\widehat{G}(\eta)\,\mathrm{d}\beta,\] where $\xi$ and $\eta$ are functions of $\alpha$ and $\beta$ respectively, and \[\widetilde{\Phi}(\alpha,\beta)=\Phi^{+}(\alpha,\beta)+O(2^{-(1-o)m});\] for notations see Section \ref{phaseprop0}. Now we choose a parameter $\rho\leq (1-o)m$, and consider the region where $|\Phi^{+}(\alpha,\beta)|\sim 2^{-\rho}$, or when $|\Phi^{+}(\alpha,\beta)|\lesssim 2^{-(1-o)m}$; in the first situation we integrate by parts in $s$, producing the $I_{0}$, $I_{1}$ and $I_{2}$ terms, in the second situation we will estimate the integral $I$ directly. In estimating all these terms we shall use Schur's lemma.

For the term with $|\Phi_{+}(\alpha,\beta)|\lesssim 2^{-(1-o)m}$, we will use Propositions \ref{linearbound0} to bound \[|\widehat{F}|\lesssim \varepsilon_{1};\qquad\sup_{\eta_{e}}\|\widehat{G}\|_{L_{\beta}^{2}}\lesssim 2^{-(5/18-6\delta)m}\varepsilon_{1},\] then use Propositions \ref{phaseprop0} and to bound\[\sup_{\alpha}\int_{\mathbb{R}}\mathbf{1}_{E}(\alpha,\beta)\,\mathrm{d}\beta\lesssim 2^{-(1/3-o)m};\qquad \sup_{\beta}\int_{\mathbb{R}}\mathbf{1}_{E}(\alpha,\beta)\,\mathrm{d}\alpha\lesssim 2^{-(1-o)m}2^{r},\] using the fact that $|\partial_{\alpha}\Phi_{+}|\sim 2^{-r}$, where $E$ is the set where $|\Phi_{+}(\alpha,\beta)|\lesssim 2^{-(1-o)m}$. Putting these into Schur's lemma, and considering that $\eta_{e}$ ranges in a ball centered at $\xi_{e}$ with radius $2^{-(1-o)m/2}$, we obtain\[\varepsilon_{1}^{-2}\|I\|_{L^{2}}\lesssim 2^{m}2^{-(1-o)m}2^{-(5/18-6\delta)m}(2^{-(1/3-o)m})^{1/2}(2^{-(1-o)m}2^{r})^{1/2}\lesssim 2^{-(17/18-6\delta-o)m}2^{r/2},\] which is bounded above by $2^{-(17/18-6\delta-o)j}$.

For the term $I_{0}$, the estimate is the same as above; simply replace $2^{-(1-o)m}$ by $2^{-\rho}$ and repeat the argument above. For $I_{1}$ and $I_{2}$, we shall estimate them in the same way as $I_{0}$, but with the additional time integral which counts as $2^{m}$; however, for the term with $\partial_{s}$ derivative, we have from (\ref{linearb5}) that\[\|\partial_{t}G\|_{L^{2}}\lesssim 2^{-(3/2-o)m}\varepsilon_{1}^{2},\qquad\|\widehat{\partial_{t}F}\|_{L^{\infty}}\lesssim 2^{-(1-12\delta-o)m}\varepsilon_{1}^{2},\] which is almost $2^{m}$ better than the corresponding $G$ and $F$ factors, so this cancels the time integration.

Finally, considering the Fourier $L^{1}$ bound, we may assume $j_{2}\leq m/2$, otherwise it would directly follow from the $L^{2}$ bound we proved above (where for $\|G\|_{L^{2}}$ we have a bound of $2^{-5m/12}\varepsilon_{1}$ instead of $2^{-5m/18}\varepsilon_{1}$). Now since $j_{2}\leq m/2$, we can assume $|\nabla_{\eta}\Phi|\lesssim 2^{-m/3}$, since otherwise we can integrate by parts in $\eta$ (or equivalently $\beta$) to get $2^{-20m}$ decay; also we can assume $|\Phi(\xi,\eta)|\lesssim 2^{-m/3}$, since otherwise we can integrate by parts in $s$, using the trick in (\ref{trick}) and the bound $\|\partial_{t}G\|_{L^{2}}\lesssim 2^{-(3/2-o)m}\varepsilon_{1}^{2}$ to close. Now using (\ref{qrad}) and (\ref{a33}), we can restrict $\xi$ to a region of volume $\lesssim 2^{-m/6}$, so we use the $L^{2}$ bound obtained above and H\"{o}lder to obtain a good $L^{1}$ bound.
\subsubsection{The case of small $j$'s}\label{smallj} Assume $j'=\max(j_{1},j_{2})\prec m/3$. In particular we know that $|\nabla_{\eta}\Phi|\ll 1$, so we are close to one of the spacetime resonance spheres described in Proposition \ref{propp}, say $(\alpha_{0},\beta_{0})$.

Recall that $|\sin\angle(\xi,\eta)|\lesssim 2^{-(1/2-o)m}$; in the discussion below we will further assume $|\partial_{\beta}^{2}\Phi^{+}|\ll 1$, where $\alpha=\pm|\xi|$ and $\beta=\pm|\eta|$. In fact, when $|\partial_{\beta}^{2}\Phi^{+}|\gtrsim 1$ we have a non-degenerate oscillatory integral (which is exactly the case in \cite{IP}), so for example we have $|\beta-R_{i}(\alpha)|\lesssim 2^{-(1/2-o)m}$ (instead of $2^{-m/3}\langle m\rangle^{A}$), and the proof will be completed in the same way as below, and every estimate will be strictly better.

 Now we will fix a time $s$, and restrict $||\xi|-\alpha_{0}|\sim 2^{-2l}$ with $0\leq l\leq m/2$, or $||\xi|-\alpha_{0}|\lesssim 2^{-m}$, using suitable cutoff functions. We are thus considering the integral\[H(s,\xi)=\int_{\mathbb{R}^{3}}e^{{i}s\Phi(\xi,\eta)}{m}(\xi,\eta)\widehat {F}(s,\xi-\eta)\widehat{G}(s,\eta)\,\mathrm{d}\eta.\]

As the first step, we can use Lemma \ref{angular} to restrict $|\sin\angle(\xi,\eta)|\lesssim 2^{-m/2}\langle m\rangle^{A}$. We then fix the direction vectors $\xi_{e}$ and $\eta_{e}$, noticing that $\eta_{e}$ stays in a set with volume $2^{-m}\langle m\rangle^{A}$ with fixed $\xi_{e}$, and let $\alpha=\pm|\xi|$ and $\beta=\pm|\eta|$. We then have $\Phi(\xi,\eta)=\Phi^{+}(\alpha,\beta)+O(2^{-m}\langle m\rangle^{A})$; with suitable choice of $A$, this remainder will be safely ignored. Below we will consider the integral\[\int_{\mathbb{R}}e^{is\Phi^{+}(\alpha,\beta)}{m}(\xi,\eta)\widehat{F}(s,\xi-\eta)\widehat{G}(s,\eta)\,\mathrm{d}\beta.\]

Next, recall that\[\partial_{\beta}\Phi^{+}(\alpha,\beta)=P(\alpha,\beta)\cdot (\beta-R_{1}(\alpha))\cdot[(\beta-R_{2}(\alpha))^{2}-Q(\alpha)],\] by our assumptions, we may assume that $\beta$ is to $R_{3}(\alpha)$ (the case of $R_{4}$ is similar).

If we restrict $|\beta-R_{3}(\alpha)|\gtrsim 2^{-\mu}$ by inserting some cutoff function, where $\mu\prec \max(\frac{m}{3},\frac{m-l}{2})$, then we have, by elementary calculus, that $|\partial_{\eta_{1}}\Phi|\gtrsim 2^{-\min(2\mu',\mu'+l)}$ and $|\partial_{\eta_{1}}^{2}\Phi|\lesssim 2^{-\min(\mu',l)}$. We then set in Proposition \ref{ips0} that \[K=2^{m},\quad n=2,\quad \epsilon_{1}\sim2^{-\min(2\mu',\mu'+l)},\quad\epsilon_{2}\sim2^{-\min(\mu',l)},\quad\lambda\sim2^{\max(\mu,j')},\] and check that the choice of parameters (in particular the choice that $\mu'\prec\max(m/3,(m-l)/2)$) guarantees sufficient decay.

Now, suppose $|\beta-R_{3}(\alpha)|\lesssim 2^{-\mu_{0}}\langle m\rangle^{A}$, where $\mu_{0}=\max(\frac{m}{3},\frac{m-l}{2})$, then we have\[|\Phi_{+}(\alpha,\beta)-\Phi_{+}(\alpha,R_{3}(\alpha))|\lesssim 2^{-m}\langle m\rangle^{A},\] which can be checked using the expressions of $\Phi_{+}$.
Using the support of $\beta$ and $\eta_{e}$, we can already bound $\|\widehat{I}\|_{L^{\infty}}\lesssim 2^{-m/3}\langle m\rangle^{A}\varepsilon_{1}^{2}$. To bound the $L^{2}$ and Fourier $L^{1}$ norms we let\[\lambda=(\partial_{\alpha}\Phi)(\alpha_{0},R_{2}(\alpha_{0}))\] as in Proposition \ref{propp}, we will consider the case when $\lambda\neq 0$ and when $\lambda=0$.

(1) Suppose $\lambda\neq 0$. If we assume $|\alpha-\alpha_{0}|\lesssim 2^{-2l}$ with $l\succeq m/2$, then we directly use the Fourier $L^{\infty}$ bound obtained above and the smallness of support to conclude; therefore we may assume $||\xi|-\alpha_{0}|\sim 2^{-2l}$ with $l\prec m/2$. From part (5) of Proposition \ref{propp}, we have that $|\Phi_{+}(\alpha,R_{3}(\alpha))|\sim 2^{-2l}$, while\[|\Phi(\xi,\eta)-\Phi_{+}(\xi,R_{3}(|\xi|))|\lesssim 2^{-m}\langle m\rangle^{A},\] so we have that $|\Phi(\xi,\eta)|\sim 2^{-2l}$, thus we will integrate by parts in $s$. The boundary term $I_{0}$ can be estimated in $L^{2}$ norm (using H\"{o}lder) by $2^{\kappa}\varepsilon_{1}^{2}$, where\[\kappa\preceq A\log m-l+2l-\frac{4m}{3}\preceq-\frac{5j}{6}\] (notice that the $2^{-4m/3}$ factor is due to the support of $\beta$ and $\eta_{e}$ together), and the Fourier $L^{1}$ norm would be bounded by $2^{\kappa'}$, where\[\kappa'\preceq A\log m-2l+2l-\frac{4m}{3}\preceq -5j/4,\] which will be acceptable. As for $I_{1}$ ($I_{2}$ is the same), we will have one more time integration, but the Fourier $L^{\infty}$ norm of $\partial_{s}F$ (or $\partial_{s}G$) will be $2^{m}$ better than the corresponding function themselves ($F$ or $G$) due to (\ref{linearb5}), so the proof will go exactly the same way.

(2) Suppose $\lambda=0$. We may again assume $||\xi|-\alpha_{0}|\sim 2^{-2l}$ with $l\prec m/2$; we next consider the case when\[j\succ\max(2l,m-l,m-\mu_{0},l+\mu_{0})=\max(2l,m-l)\] where we recall $\mu_{0}=\max(m/3,(m-l)/2)$. We then fix $\eta$ and $s$, and repeat the argument of integrating by parts in $\xi$ as before; for a fixed point $x$ with $|x|\sim 2^{j}$, we analyze the inegral\begin{equation}\int_{\mathbb{R}^{3}}e^{{i}(s\Phi+x\cdot\xi)}\chi\bigg(2^{n_{1}}\big(\eta-\frac{\eta\cdot\xi}{|\xi|^{2}}\xi\big)\bigg)\chi\bigg(2^{\mu}\big(\frac{\eta\cdot\xi}{|\xi|}-R_{3}(|\xi|)\big)\bigg)\widehat{\widetilde{f}}(\xi-p_{3}(\xi))\widehat{\widetilde{g}}(p_{3}(\xi))\,\mathrm{d}\xi\nonumber.\end{equation} Note that with the cutoff functions, we have\[|\nabla_{\xi}\Phi(\xi,\eta)|\lesssim 2^{-\mu}+|\nabla_{\xi}\Phi(\xi,p_{3}(\xi))|\lesssim 2^{-\mu}+2^{-l},\] the last inequality following easily from the proof of Proposition \ref{propp}. Now we can set in Proposition \ref{ips0} that $K=2^{m}$, $\epsilon=2^{j-m}$, $k=1$, $L=2^{100M}$ and $\lambda=\max(j',l+\mu,n_{1})$ to obtain a decay of $2^{-100M}\varepsilon_{1}^{2}$. Note that taking each derivative of the cutoff function \[\chi\bigg(2^{\mu}\big(\frac{\eta\cdot\xi}{|\xi|}-R_{3}(|\xi|)\big)\bigg)\] may cost us up to $2^{\max(2l,l+\mu)}$; this can be proved using the same argument as the proof of Proposition \ref{ips0}.

Now we may assume $j\preceq\max(2l,m-l)$. Using this and the bound for $\varepsilon_{1}^{-2}\|\widehat{I}\|_{L^{\infty}}$ we already have, we can bound $\varepsilon_{1}^{-2}\|I\|_{L^{2}}$ \emph{without integrating by parts in $s$} by $2^{\kappa}$, where\[\kappa\preceq m-l-4m/3\preceq -5l/3\preceq -5j/6\] when $l\geq m/3$ (the Fourier $L^{1}$ bound follows from H\"{o}lder), and\[\kappa\preceq m-l-4m/3\preceq -5(m-l)/6\preceq -5j/6\] when $3m/11\leq l\leq m/3$ (the Fourier $L^{1}$ bound also follows from H\"{o}lder). Finally when $l\leq 3m/11$, we note that $|\Phi|\sim 2^{-3l}$ so we can intgrate by parts in $s$ to bound $\varepsilon_{1}^{-2}\|I\|_{L^{2}}$ by $2^{\max(\kappa_{1},\kappa_{2})}$, where\[\kappa_{1}\preceq3l-l-4m/3\preceq -(m-l)-m/10\preceq -j-m/10,\] and\[\kappa_{2}\preceq m+3l-l-(4m/3)-m\preceq -j-m/10.\]
\subsection{Low frequency output}\label{lowout} In this section we assume $k\leq -K_{0}^{2}$. As before, we can now exclude the cases when $j\succ m$, or when $\min(j_{1},j_{2})\succeq m$; we may then assume $|\sin\angle(\xi,\eta)|\lesssim 2^{-(m+k)/2}$ due to Lemma \ref{angular}. The rest of the proof then goes in the same way as in Section \ref{medout}, making necessary changes due to the fact that $|\xi|\ll 1$.  We will only sketch the arguments here.

To be precise, if $j_{1}\geq (1/2-o)m$, we can first fix the angle $\angle(\xi,\eta)$, then reduce to a one-dimensional integral as before (note that $|\partial_{\alpha}\Phi^{+}|\gtrsim 1$ here), so that we can use Schur's test and integration by parts in time to obtain $\|I\|_{L^{2}}\lesssim 2^{-\kappa}\varepsilon_{1}^{2}$, where\[\kappa\leq k+\rho-(m+k)-\rho/2-(\rho/2)/2-5/6(1/2-o)m\leq -(1+o)m,\] where $\rho$ is such that $|\Phi^+(\alpha,\beta)|\sim 2^{\rho}$ as in Section \ref{easy}, and the first $2^{k}$ factor is due to the fact that\[\|\widehat{I}\|_{L^{2}}\lesssim 2^{k}\sup_{\theta\in\mathbb{S}^{2}}\|\widehat{I}(\alpha\theta)\|_{L_{\alpha}^{2}},\] when $\widehat{I}$ is supported where $|\xi|\sim 2^{k}$.

Now let us assume $\max(j_{1},j_{2})\leq(1/2-o)m$. Again we must have $|\sin\angle(\xi,\eta)|\lesssim 2^{-(m+k)/2}$, and again the relation\[\Phi(\xi,\eta)=\Phi^{+}(\alpha,\beta)+O(2^{-(1-o)m})\] holds. After fixing the directions of $\xi$ and $\eta$ and reducing to the one-dimensional integral, notice that when $|\Phi^{+}|+|\partial_{\beta}\Phi^{+}|\ll 1$ we have $|\partial_{\beta}^{2}\Phi^{+}|\gtrsim 1$ by Proposition \ref{propp}, thus under the condition $\max(j_{1},j_{2})\leq (1/2-o)m$ we can restrict $|\beta-p(\alpha)|\lesssim 2^{-(1/2-o)m}$ where $\beta=p(\alpha)$ is the unique solution to $\partial_{\beta}\Phi_{+}(\alpha,\beta)=0$, thus obtaining $\|\widehat{I}\|_{L^{\infty}}\lesssim 2^{-(1/2-o)m}\varepsilon_{1}^{2}$ in exactly the same way as Section \ref{smallj}, and again we have \[\Phi_{+}(\alpha,\beta)=\Phi^{+}(\alpha,p(\alpha))+O(2^{-(1-o)m});\] notice that $|\partial_{\alpha}\Phi^{+}(\alpha,p(\alpha))|\gtrsim 1$, we can then analyze the size of $|\Phi^{+}(\alpha,p(\alpha))|$ and integrate by parts in $s$ (and use (\ref{linearb5}) to bound the $L^{\infty}$ norm of $\widehat{\partial_{s}F}$ and $\widehat{\partial_{s}G}$) when necessary, as we did in Section \ref{medout} above. Compared with that case, the $L^{\infty}$ bound here is better, and the volume bound for $\{\alpha: |\Phi^{+}(\alpha,p(\alpha))|\sim 2^{-\rho}\}$ is also better with fixed $\rho$. This completes the proof.

\section{High frequency case}\label{high} In this section we assume $\overline{k}:=\max(k,k_{1},k_{2})\geq K_{0}$. By Proposition \ref{propp}, we have either $|\Phi|\gtrsim 2^{-2\overline{k}}$, or $\min(k,k_{1},k_{2})\geq -D_{0}$. In the latter case we also have that, either $c_{\sigma}=c_{\mu}=c_{\nu}$ and $b_{\sigma}-b_{\mu}-b_{\nu}=0$, or $|\nabla_{\eta}\Phi|\gtrsim 2^{-4\overline{k}}$.

Now, if $|\Phi|\gtrsim 2^{-2\overline{k}}$, we simply integrate by parts in $s$ and estimate the corresponding $I_{0}$, $I_{1}$ and $I_{2}$ terms as before; for example by a variant of (\ref{basiccon}) we have\[\|I_{0}\|_{L^{2}}\lesssim 2^{3\overline{k}}\cdot 2^{6\overline{k}}2^{-9m/8}\varepsilon_{1}^{2}\lesssim 2^{-(1+o)m}\varepsilon_{1}^{2}\] and\[\|I_{1}\|_{L^{2}}\lesssim 2^{m+3\overline{k}}\cdot 2^{-(1-o)m}2^{-9m/8}\varepsilon_{1}^{2}\lesssim 2^{-(1+o)m}\varepsilon_{1}^{2}.\] If $|\Phi|\lesssim 2^{-2\overline{k}}$ and $|\nabla_{\eta}\Phi|\gtrsim 2^{-4\overline{k}}$, then we may assume $j_{1}\geq (1-o)m-4\overline{k}$, thus we can use Lemma \ref{angular} to restrict $|\sin\angle (\xi,\eta)|\lesssim 2^{-(1/2-o)m}$; we then insert cutoff functions to restrict $|\Phi|\sim 2^{-\gamma}$, integrate by parts in $s$, and use Schur's lemma as before, to obtain $\|I\|_{L^{2}}\lesssim 2^{\max(\kappa_{0},\kappa_{1})}\varepsilon_{1}^{2}$, where\[\kappa_{0}=\gamma-(1-o)m+8\overline{k}-\gamma-5m/6\leq -(1+o)m\] and\[\kappa_{1}=\gamma-(1-o)m+8\overline{k}-\gamma-9m/8\leq -(1+o)m.\] Finally, when $c_{\sigma}=c_{\mu}=c_{\nu}=1$ (say) and $b_{\sigma}-b_{\mu}-b_{\nu}=0$, then we have $|\Phi|\gtrsim 2^{-2\overline{k}}$ so we can argue as above, unless $|\xi|,|\eta|,|\xi-\eta|\sim 2^{k}$, in which case we have \[|\nabla_{\xi}\Phi|\sim|\nabla_{\eta}\Phi|\sim 2^{-3k}|\eta-\rho\xi|\] for some $\rho$, as in Proposition \ref{propp}. Using the same integration by parts argument (in $\xi$ or $\eta$) as above, we may assume that $j\preceq m-3k-l\preceq\max(j_{1},j_{2})$, where $|\eta-\rho\xi|\sim 2^{l}$; moreover we can assume $|\sin\angle (\xi,\eta)|\lesssim 2^{-(1/2-o)m}$ by Lemma \ref{angular}, so we may insert cutoff functions to restrict $|\Phi|\sim 2^{-\gamma}$, then exploit the localization in the angular variable to reduce to one-dimensional integral, and then integrate by parts in $s$ so get that $\|I\|_{L^{2}}\lesssim 2^{\max(\kappa_{0},\kappa_{1})}\varepsilon_{1}^{2}$ where\[\kappa_{0}=\gamma-(1-o)m+(3k+l)-\gamma-\max(j_{1},j_{2})\leq -3j/2,\]\[\kappa_{1}=\gamma+m-(1-o)m+(3k+l)-\gamma-9m/8\leq -(1+1/20)j,\] so in any case we get the desired estimate.
\section{Auxiliary results}\label{phaseprop}
\subsection{Properties of phase functions}\label{phaseprop0} Let $\Phi=\Phi_{\sigma\mu\nu}$ and $\Phi^{+}$ be defined as in Section \ref{notatkg}. Define the spacetime resonance set\begin{equation}\label{stres1}\mathcal{R}=\{(\xi,\eta):\Phi(\xi,\eta)=\nabla_{\eta}\Phi(\xi,\eta)=0\},\end{equation}and also assume \[|\xi|\sim 2^{k},\quad|\xi-\eta|\sim 2^{k_{1}},\quad|\eta|\sim 2^{k_{2}};\quad\max(k,k_{1},k_{2}):=\overline{k}.\]
\begin{proposition}\label{propp}
(1) Suppose $|\overline{k}|\leq K_{0}^{2}$ and $(\xi,\eta)=(\alpha e,\beta e)$ for some $e\in\mathbb{S}^{2}$, then we have\begin{equation}\label{derbound}\sup_{\mu\leq 3}|\partial_{\beta}^{\mu}\Phi^{+}(\alpha,\beta)|\gtrsim 1,\end{equation} and the same holds for $\partial_{\alpha}$. If moreover $\min(k,k_{1},k_{2})\leq -D_{0}$, then the range $\mu\leq 3$ above can be improved to $\mu\leq2$. Note that (\ref{derbound}) implies the measure bound \begin{equation}\label{measureb}\sup_{\alpha}|\{\beta:|\Phi(\alpha,\beta)|\leq\epsilon\}|\lesssim\epsilon^{1/3}\end{equation} for bounded $\alpha,\beta$ by well-known results; see for example \cite{DIP14}, Section 8.

(2) If $c_{\sigma}$, $c_{\mu}$ and $c_{\nu}$ are not all equal, then the set\begin{equation}\label{stres2}\mathcal{R}=\{(\alpha \theta,\beta \theta):\theta\in\mathbb{S}^{2},(\alpha,\beta)\in R\},\end{equation} where $R\subset\mathbb{R}^{2}$ is a \emph{finite} set. If $c_{\sigma}=c_{\mu}=c_{\nu}$, then if $b_{\sigma}-b_{\mu}-b_{\nu}\neq 0$ then $\mathcal{R}=\emptyset$; otherwise $\mathcal{R}=\{(\xi,\rho\xi)\}$, where $\rho=b_{\nu}/b_{\sigma}$. In this case we have\[|\nabla_{\eta}\Phi|\sim (1+2^{5k_{1}})^{-1}|\eta-\rho\xi|\] if $|k_{1}-k_{2}|=O(1)$, and $|\nabla_{\eta}\Phi|\gtrsim 2^{-2\overline{k}}$ otherwise.

(3) Suppose $k\geq \max(k_{1}+D_{0},D_{0}^{2}/2)$, then we have $|\Phi|+|\nabla_{\eta}\Phi|\gtrsim 2^{-2k_{1}}$; if $|k_{1}-k_{2}|\leq 2D_{0}$ and $k_{1}\geq D_{0}^{2}/2$, then either $|\Phi|+|\nabla_{\eta}\Phi|\gtrsim 2^{-4k_{1}}$, or $c_{\sigma}=c_{\mu}=c_{\nu}$ and $b_{\sigma}-b_{\mu}-b_{\nu}=0$.

(4) For $\alpha\neq 0$, there exist smooth functions $R_{1}(\alpha)$ and $R_{2}(\alpha)$, strictly positive functions $P(\alpha,\beta)$, and a function $Q(\alpha)$ which has only simple zeros, such that\begin{equation}\label{qrad}\partial_{\beta}\Phi^{+}(\alpha,\beta)=\pm P(\alpha,\beta)\cdot (\beta-R_{1}(\alpha))\cdot[(\beta-R_{2}(\alpha))^{2}-Q(\alpha)].\end{equation} Moreover, near each point where $Q(\alpha)=0$, we have that $R_{1}(\alpha)-R_{2}(\alpha)$ is bounded away from zero. We will also define \[R_{3}(\alpha)=R_{2}(\alpha)+\sqrt{|Q(\alpha)|},\,\,\,\,\,\,R_{4}(\alpha)=R_{2}(\alpha)-\sqrt{|Q(\alpha)|}.\]

(5) Suppose at some point $\alpha_{0}$ we have $Q(\alpha_{0})=0=\Phi_{+}(\alpha_{0},R_{2}(\alpha_{0}))$. Let also $\lambda=(\partial_{\alpha}\Phi)(\alpha_{0},R_{2}(\alpha_{0}))$, then we have \begin{equation}\label{a33}\Phi^{+}(\alpha,R_{3}(\alpha))=\lambda(\alpha-\alpha_{0})+\lambda'|\alpha-\alpha_{0}|^{\frac{3}{2}}+O(|\alpha-\alpha_{0}|^{2}),\end{equation} where $\lambda'$ may take different values depending on whether $\alpha>\alpha_{0}$ or $\alpha<\alpha_{0}$, but is always nonzero; the same thing happens for $R_{4}$. Moreover, when $\lambda=0$, we also have \begin{equation}\label{more}|\partial_{\alpha}(\Phi^{+}(\alpha,R_{3}(\alpha)))|\sim |\alpha-\alpha_{0}|^{1/2}.\end{equation}
\end{proposition}
\begin{proof} (1) we only need to prove that there is no $(\alpha,\beta)$ where $\partial_{\beta}^{j}\Phi^{+}=0$ for all $j\leq 3$, unless $\alpha=\beta=0$. In fact, if $\partial_{\beta}^{j}\partial_{\beta}\Phi^{+}(\alpha,\beta)=0$ for $j\leq 2$, we must have $\mu\nu<0$, since the function $\sqrt{ax^{2}+b}$ is strictly convex; if we write $\beta-\alpha=\gamma$, we will obtain\begin{equation}\label{res01}\frac{c_{\nu}^{2}\beta}{\sqrt{c_{\nu}^{2}\beta^{2}+b_{\nu}^{2}}}=\frac{c_{\mu}^{2}\mu}{\sqrt{c_{\mu}^{2}\gamma^{2}+b_{\mu}^{2}}},\,\,\,\,\,\frac{c_{\nu}^{2}b_{\nu}^{2}}{(c_{\nu}^{2}\beta^{2}+b_{\nu}^{2})^{3/2}}=\frac{c_{\mu}^{2}b_{\mu}^{2}}{(c_{\mu}^{2}\gamma^{2}+b_{\mu}^{2})^{3/2}},\end{equation} and\begin{equation}\label{res02}\frac{c_{\nu}^{2}\beta}{(c_{\nu}^{2}\beta^{2}+b_{\nu}^{2})^{5/2}}=\frac{c_{\mu}^{2}\mu}{(c_{\mu}^{2}\gamma^{2}+b_{\mu}^{2})^{5/2}}.\end{equation} By elementary algebra, these three equations will force $c_{\nu}=c_{\mu}$, $b_{\tau}+b_{\sigma}=0$ and $\beta=\gamma$, but in this case we clearly have $\Phi^{+}(\alpha,\beta)\neq 0$.

For the case when $\min(k,k_{1},k_{2})\leq -D_{0}$, we again only need to show that there is no $(\alpha,\beta)$ where $\partial_{\beta}^{j}\Phi^{+}=0$ for all $j\leq 2$, and $\min(|\alpha|,|\beta|,|\alpha-\beta|)\leq 2^{-D_{0}}$. In fact, using part (2) below, we only need to consider spacetime resonance other than $(0,0)$ with $\alpha\beta(\alpha-\beta)=0$, unless $c_{\sigma}=c_{\mu}=c_{\nu}$ and $b_{\sigma}-b_{\mu}-b_{\nu}=0$ (in this latter case we must have $\beta=\rho\alpha$ using the notation in (2), in which case we can directly compute $\partial_{\beta}^{2}\Phi^{+}\neq 0$); clearly we cannot have $\beta=0$ or $\alpha-\beta=0$, and when $\alpha=0$ then $\beta=\gamma$ as above, so we can repeat the arguments to show that now $\partial_{\beta}\Phi^{+}=\partial_{\beta}^{2}\Phi_{+}=0$ implies $c_{\nu}=c_{\mu}$ and $b_{\nu}+b_{\mu}=0$, so that $\Phi^{+}(0,\beta)\neq 0$.

(2) If $\nabla_{\eta}\Phi=0$, then $\eta$ and $\xi-\eta$ must be collinear, so we have $\xi=\alpha \theta$ and $\eta=\beta\theta$, where $\theta\in\mathbb{S}^{2}$, and $\Phi^{+}(\alpha,\beta)=\partial_{\beta}\Phi^{+}(\alpha,\beta)=0$.

First assume $c_{\mu}\neq c_{\nu}$, then we have $|\partial_{\beta}\Lambda_{\mu}(\alpha-\beta)|=|\partial_{\beta}\Lambda_{\nu}(\beta)|:=k$, and hence\[\alpha-\beta=\frac{\pm b_{\mu}k}{c_{\mu}\sqrt{c_{\mu}^{2}-k^{2}}},\quad\Lambda_{\mu}(\alpha-\beta)=\frac{b_{\mu}c_{\mu}}{\sqrt{c_{\mu}^{2}-k^{2}}}\] and similarly for $\beta$, so we get\[\bigg(\frac{b_{\mu}c_{\mu}}{\sqrt{c_{\mu}^{2}-k^{2}}}+\frac{b_{\nu}c_{\nu}}{\sqrt{c_{\nu}^{2}-k^{2}}}\bigg)^{2}=c_{\sigma}^{2}\bigg(\frac{b_{\mu}k}{c_{\mu}\sqrt{c_{\mu}^{2}-k^{2}}}+\frac{b_{\nu}k}{c_{\nu}\sqrt{c_{\nu}^{2}-k^{2}}}\bigg)^{2}+b_{\sigma}^{2},\] which reduces to a polynomial equations in $k$ that is not an identity (since $c_{\mu}\neq c_{\nu}$), so we only have a finite number of $k$'s.

Now assume $c_{\mu}=c_{\nu}=1$, then with $\rho=b_{\nu}/b_{\sigma}$, it is easy to check that $\partial_{\beta}\Phi^{+}=0$ implies $\beta=\rho\alpha$. Plugging this into $\Phi^{+}=0$, we obtain an equation for $\alpha$ that is an identity when $c_{\sigma}=1$ and $b_{\sigma}-b_{\mu}-b_{\nu}=0$; thus this equation has at most two solutions when $c_{\sigma}\neq 1$, and no solution when $c_{\sigma}=1$ and $b_{\sigma}-b_{\mu}-b_{\nu}\neq0$. Finally the bound on $|\nabla_{\eta}\Phi|$ when $c_{\mu}=c_{\nu}$ follows from the fact that $\xi\mapsto \nabla\Lambda_{\mu}(\xi)$ is diffeomorphism whose Jacobian matrix has an inverse bounded by $(1+|\xi|)^{-3}$ pointwise.

(3)  In the first case, we may assume $c_{\sigma}=c_{\nu}$, otherwise we would have $|\Phi|\gtrsim 2^{k}$. Now if $c_{\sigma}=c_{\nu}=1$ and $c_{\mu}<1$, we would have \[|\nabla_{\eta}\Phi|\geq \frac{|\eta|}{\sqrt{|\eta|^{2}+b_{n\nu}^{2}}}-\frac{c_{\mu}^{2}|\xi-\eta|}{\sqrt{c_{\mu}^{2}|\xi-\eta|^{2}+b_{\mu}^{2}}}\geq\frac{1-c_{\mu}}{2};\] if $c_{\mu}>1$ we would have \[|\Phi|\geq \max(c_{\mu}|\xi-\eta|,|b_{\mu}|)-\big||\xi|-|\eta|\big|+O(2^{-k})\gtrsim 1.\] If $c_{\mu}=1$, we may directly compute that $|\nabla_{\eta}\Phi|\gtrsim 2^{-2k_{1}}$.

In the second case, assume $|\Phi|+|\nabla_{\eta}\Phi|\ll 2^{-4k_{1}}$, then we must have $c_{\mu}=c_{\nu}=1$, since when $c_{\mu}\neq c_{\nu}$ we have $|\nabla_{\eta}\Phi|\gtrsim 1$; this in turn implies $|\eta-\rho\xi|\lesssim 2^{-k_{1}}$ where $\rho=b_{\nu}/(b_{\mu}+b_{\nu})$ (or $|\xi|\lesssim 2^{-k_{1}}$ if $b_{\mu}+b_{\nu}=0$, which is easily seen to be impossible). Plugging into $|\Phi|\ll 2^{-4k_{1}}$, we see that the only possibility is $c_{\sigma}=1$ and $b_{\sigma}-b_{\mu}-b_{\nu}=0$.

(4) When $c_{\mu}=c_{\nu}$, the unique solution of $\partial_{\beta}\Phi^{+}(\alpha,\beta)=0$ is $\beta=\rho\alpha$ with $\rho=b_{\nu}/(b_{\mu}+b_{\nu})$, and $\partial_{\beta}^{2}\Phi^{+}(\alpha,\beta)\neq0$; below we will assume $c_{\mu}\neq c_{\nu}$.

After squaring, the equation $\partial_{\beta}\Phi_{+}(\alpha,\beta)=0$ becomes\begin{equation}\label{squared}\frac{c_{\mu}^{4}(\beta-\alpha)^{2}}{c_{\mu}^{2}(\beta-\alpha)^{2}+b_{\mu}^{2}}=\frac{c_{\nu}^{4}\beta^{2}}{c_{\nu}^{2}\beta^{2}+b_{\nu}^{2}}.\end{equation} Assume $\alpha>0$, note that (\ref{squared}) has at least two roots $\beta$: a unique solution exists in $(0,\alpha)$, and when $c_{1}>c_{2}$, a unique solution exists in $(\alpha,+\infty)$, and when $c_{1}<c_{2}$, a unique solution exists in $(-\infty,0)$. These two solutions are both simple, and depends smoothly on $\alpha$, and exactly one of them actually solves the equation $\partial_{\beta}\Phi^{+}(\alpha,\beta)=0$, depending on the signs of $b_{\mu}b_{\nu}$.

On the other hand, (\ref{squared}) simplifies to a polynomial equation of $\beta$ with degree four, so it has at most four roots. If we quotient out the two roots mentioned above, we are left with a quadratic equation\[P(\alpha)\beta^{2}+Q(\alpha)\beta+R(\alpha)=0,\] where $P(\alpha)> 0$. After completing the square, we can then write $\partial_{\beta}\Phi^{+}(\alpha,\beta)$ in the form of (\ref{qrad}). It is also clear from above that $R_{1}$ and $R_{2}$ are always separated. Next, note that $\partial_{\alpha}\partial_{\beta}\Phi^{+}(\alpha,\beta)\neq 0$; if we choose $\beta=R_{2}(\alpha)$, it then follows that $Q(\alpha)=0$ and $Q'(\alpha)=0$ cannot happen at the same point.

(5) First, note that \[\partial_{\alpha}(\Phi^{+}(\alpha,R_{2}(\alpha)))_{\alpha=\alpha_{0}}=\lambda+(\partial_{\beta}\Phi^{+})(\alpha_{0},R_{2}(\alpha_{0}))=\lambda,\] we thus have $\Phi_{+}(\alpha,R_{2}(\alpha))=\sigma(\alpha-\alpha_{0})+O(|\alpha-\alpha_{0}|^{2})$. Next, since we are close to the point $(\alpha_{0},R_{2}(\alpha_{0}))$, we will have \[\Phi^{+}(\alpha,R_{3}(\alpha))-\Phi^{+}(\alpha,R_{2}(\alpha))=\int_{0}^{\sqrt{|Q(\alpha)|}}P(\alpha,\beta+R_{2}(\alpha))\cdot(\beta^{2}-Q(\alpha))\,\mathrm{d}\beta,\]  where $P$ is (say) smooth and strictly positive. We then compute, with $\alpha$ close to $\alpha_{0}$, that this integral equals some $c|Q(\alpha)|^{3/2}$ plus an error of $O(|Q(\alpha)|^{2})$, where $c$ is nonzero but may depend on the sign of $Q(\alpha)$. Since $Q(\alpha)$ has a simple zero at $\alpha_{0}$, this proves (\ref{a33}). The proof of (\ref{more}) is similar, and will be omitted here.
\end{proof}
\subsection{The spherical harmonics decomposition}
\begin{proposition}\label{radial} For each function $f$ on $\mathbb{R}^{3}$, let its spherical harmonics decomposition be as in (\ref{decomp}) and define $S_{l}$ as in (\ref{spherical}). Then we have\begin{equation}\label{lp}\|S_{l}f\|_{L^{p}}\lesssim\|f\|_{L^{p}}\end{equation} uniformly in $l$, for $1\leq p\leq\infty$, and also\begin{equation}\label{sobolev}\|S_{l}f\|_{L^{\infty}}\lesssim 2^{l}\bigg(\sup_{r}r^{-2}\int_{|\xi|=r}|f|^{2}\,\mathrm{d}\omega\bigg)^{1/2}.\end{equation}
\end{proposition}
\begin{proof} In what follows we will assume $f$ is a function on $\mathbb{S}^{2}$, because we can always use polar coordinates $(\rho,\theta)\in\mathbb{R}^{+}\times\mathbb{S}^{2}$ and then fix $\rho$. For the standard identities about zonal harmonics, the reader may consult \cite{SW71}, Chapter 4.

(1) Assume $l\geq 1$. Recall the zonal harmonics $Z_{\theta}^{q}(\theta')$ of degree $q$, and define \[K(\theta,\theta')=\sum_{q\geq 0}\varphi(2^{-l}q)Z_{\theta}^{q}(\theta'),\]then we have \[S_{l}f(\theta)=\int_{\mathbb{S}^{2}}f(\theta')K(\theta,\theta')\,\mathrm{d}\omega(\theta').\] Therefore (\ref{lp}) would follow from a bound of $\|K(\theta,\cdot)\|_{L^{1}}$ uniform in $\theta$, which would be a consequence of the pointwise inequality\begin{equation}\label{point}|K(\theta,\theta')|\lesssim\min(2^{2l},2^{-l}|\theta-\theta'|^{-3}).\end{equation}

When $|\theta-\theta'|\lesssim 2^{-l}$ the bound (\ref{point}) will be trivial, since we have $|Z_{\theta}^{q}(\theta')|\lesssim 2q+1$ for each $q$. Now we assume $\epsilon=|\theta-\theta'|\gg 2^{-l}$, let $\lambda=\theta\cdot \theta'=1-\epsilon^{2}/2$, and write $\alpha_{r}=Z_{\theta}^{q}(\theta')$. Then we have the generating function identity\begin{equation}\label{identity}\sum_{q=0}^{\infty}\alpha_{q}\rho^{q}=\frac{1-\rho^{2}}{(\rho^{2}-2\lambda \rho+1)^{3/2}}\end{equation} for $0\leq \rho<1$. Since the function $\rho^{2}-2\lambda \rho+1$ never vanishes when $\rho\in\mathbb{C}$ and $|\rho|<1$, the right hand side of (\ref{identity}) will have a unique holomorphic continuation to the unit disc in $\mathbb{C}$, and it must equal the left hand side. If we now choose $\rho=\exp(\frac{{i}y-1}{N})$ for each $y\in\mathbb{R}$, and let $h(y)$ be the Fourier transform of $\varphi(x)e^{x}$ (which is Schwartz because $\varphi$ has compact support), we will get\begin{eqnarray}
K(\theta,\theta')&=&\sum_{q=0}^{\infty}e^{-\frac{q}{N}}\alpha_{q}\int_{\mathbb{R}}h(y)e^{{i}\frac{q}{N}y}\,\mathrm{d}y\nonumber\\
&=&\int_{\mathbb{R}}h(y)\,\mathrm{d}y\cdot\sum_{r=0}^{\infty}\alpha_{q}(e^{\frac{{i}y-1}{N}})^{q}\nonumber\\
&=&\int_{\mathbb{R}}h(y)\frac{1-e^{\frac{2({i}y-1)}{N}}}{\big(1-2\lambda e^{\frac{{i}y-1}{N}}+e^{\frac{2({i}y-1)}{N}}\big)^{3/2}}\,\mathrm{d}y.\nonumber
\end{eqnarray} Now, if $\rho=\exp(\frac{{i}y-1}{N})$, we will have $|1-\rho^{2}|\lesssim N^{-1}\langle y\rangle$. For the denominator we have\[|1-2\lambda \rho+\rho^{2}|=|\rho-\rho_{+}|\cdot|\rho-\rho_{-}|\geq (1-|\rho|)^{2}\sim N^{-2}\] where $\rho_{\pm}=\lambda\pm{i}\sqrt{1-\lambda^{2}}$, so the integral for $|y|\gtrsim N\epsilon$ will be bounded by\[N^{-1}N^{3}(N\epsilon)^{-3}\int_{|y|\gtrsim N\epsilon}\langle y\rangle^{-4}\lesssim N^{-1}\epsilon^{-3}.\] In the region $|y|\ll N\epsilon$, we will have $|\arg(\rho)|\ll \epsilon$ and $|\arg(\rho_{\pm})|=\arccos(\lambda)\sim\epsilon$, which implies $|1-2\lambda \rho+\rho^{2}|\sim\epsilon^{2}$ (note also that $1\ll N\epsilon$), which allows us to bound the integral by\[N^{-1}\epsilon^{-3}\int_{|y|\lesssim N\epsilon}\langle y\rangle^{-4}\lesssim N^{-1}\epsilon^{-3},\] and the proof is complete.

(2) Recall that $\{Y_{q}^{m}\}$ form an orthonormal basis for the space of spherical harmonics of degree $r$ (with $L^{2}$ inner product), where $-q\leq m\leq q$, then we have \begin{equation}\label{sb}Z_{\theta}^{q}(\theta')=\sum_{m=-q}^{q}Y_{q}^{m}(\theta)\overline{Y_{q}^{m}(\theta')}.\end{equation}Assume $l\geq 1$, since \[S_{l}f(\theta)=\int_{\mathbb{S}^{2}}\bigg(\sum_{q}\varphi_{1}(2^{-l}q)Z_{\theta}^{q}\bigg)(\theta')f(\theta')\,\mathrm{d}\omega(\theta'),\] we only need to bound the $L_{\theta'}^{2}$ norm of $\sum_{q}\varphi_{1}(2^{-l}q)Z_{\theta}^{q}$ for each $\theta$. Since $Z_{\theta}^{q}$ and $Z_{\theta}^{q'}$ are orthogonal for $q\neq q'$, from (\ref{sb}) and the properties of $Z_{\theta}^{r}$ we can deduce\[\bigg\|\sum_{q}\varphi_{1}(2^{-l}q)Z_{\theta}^{q}\bigg\|_{L^{2}}^{2}\lesssim\sum_{q\lesssim 2^{l}}\sum_{m=-q}^{q}|Y_{q}^{m}(\theta)|^{2}\sim \sum_{q\lesssim 2^{l}}|Z_{\theta}^{q}(\theta)|\sim\sum_{q\lesssim 2^{l}}(2q+1)\sim 2^{2l},\] which is exactly what we need.
\end{proof}

\end{document}